\documentclass[11pt]{amsart}
\usepackage{amsmath,amsfonts,amssymb,amsthm,amscd,latexsym,euscript,verbatim,
bbold}
\usepackage[all]{xy}
\makeatletter
\def\section{\@startsection{section}{1}%
  \z@{1.1\linespacing\@plus\linespacing}{.8\linespacing}%
  {\normalfont\Large\scshape\centering}}
\makeatother

\theoremstyle{plain}

\newtheorem*{conj*}{Root Groups Conjecture}
\newtheorem*{thm1.2}{(1.2) Theorem}
\newtheorem*{thm1.3}{(1.3) Theorem}
\newtheorem*{thm1.4}{(1.4) Theorem}
\newtheorem*{prop*}{Proposition}
\newtheorem*{thm*}{Theorem}

\def\eroman{\etype{\roman}}

\newtheorem{prop}{Proposition}[section]

\newtheorem{thm}[prop]{Theorem}
\newtheorem{cor}[prop]{Corollary}
\newtheorem{lemma}[prop]{Lemma}

\theoremstyle{definition}
\newtheorem{Def}[prop]{Definition}
\newtheorem*{Def*}{Definition}

\newtheorem{Defsnot}[prop]{Definitions and notation}

\newtheorem*{notation*}{Notation}
\newtheorem{remark}[prop]{Remark}

\newcommand{\etype}[1]{\renewcommand{\labelenumi}{(#1{enumi})}}

\newcommand{\ff}{\mathbb{F}}

\newcommand{\frakR}{\mathfrak{R}}

\newcommand{\ga}{\alpha}
\newcommand{\gb}{\beta}
\newcommand{\gc}{\gamma}

\newcommand{\gd}{\delta}

\newcommand{\gl}{\lambda}

\newcommand{\gt}{\tau}

\newcommand{\charc}{{\rm char}}

\newcommand{\half}{\frac 12}

\numberwithin{equation}{section}

\begin{document}
\title{Two axes in non-associative algebras with a Frobenius form }
\author[Yoav Segev]{Yoav Segev}

\address{Yoav Segev \\
         Department of Mathematics \\
         Ben-Gurion University \\
         Beer-Sheva 84105 \\
         Israel}
\email{yoavs@math.bgu.ac.il}

\thanks{}


%{Primary: 17A01, 17A15;
 %Secondary: 16S36, 16U40, 17A36, 17C27, 20M25}

\begin{abstract}
Throughout this paper $A$ is a commutative non-associative algebra over a field $\ff$ of characteristic not $2.$ In addition $A$ possesses a Frobenius form.  We obtain detailed information about the multiplication in $A$ given two axes of type half in $A.$

\end{abstract}
\date{\today}
\maketitle

\tableofcontents

%%%%%%%%%%%%%%%%%%%%%%%%%%%%%%%%%%%%%%%%%%%%%%%%
%%%%%%%%%%%%%%%%%%%%%%%%%%%%%%%%%%%%%%%%%%%%%%
%%%%%%%%%%%%%%%%%%%%%%%%%%%%%%%%%%%%%%%%%%%
\section{Introduction}
%%%%%%%%%%%%%%%%%%%%%%%%%%%%%%%%%%%%%%%%%%
%%%%%%%%%%%%%%%%%%%%%%%%%%%%%%%%%%%%%%%
%%%%%%%%%%%%%%%%%%%%%%%%%%%%%%%%%%%%%%
Throughout this paper $A$ is a commutative non-associative algebra over a field $\ff$ of characteristic not $2.$ Assume that $(\ ,\ )$ is a Frobenius form on $A,$ that is, $(\ ,\ )$ is a symmetric bilinear map such that $(xy,z)=(x,yz),$ $\forall x,y,z\in A.$ We denote by $\frakR$ the radical of $(\ ,\ )$. Recall that for $z\in A,\ \gl\in\ff,$ we denote $A_{\gl}(z)=\{x\in A\mid zx=\gl x\}.$  We direct the reader to  Definition \ref{axis} below and the beginning of \S\ref{SQP} for the definitions required for the introduction.

This paper was inspired by \cite{d}, which in turn was inspired by \cite{gss}.  These two papers reveal an unexpected, and rather amazing structure in primitive axial algebras of Jordan type half.  Since axial algebras are related to $3$-transposition groups, on the one hand, and to the Griess Monster algebra, on the other hand, there must be something real there (:)).

Now when one looks carefully on the arguments in \cite{d}, one realizes that all that is being used  for an idempotent $c\in A$ is that: $P_c(x,y)=0,\ \forall x,y\in A$ iff $c$ is an axis (where in characteristic $3$ one has to also require $(c,c)=1,$ see Lemma \ref{QP}).   However, some of the arguments in \cite{d} are sleek, smart (and short), but I thought there should be more elementary (albeit longer) ones.  Thus, I decided to embark on the tedious (though rewarding) effort to make the arguments elementary, and, on the way, there are some very interesting identities that are obtained.  

Indeed, Theorem \ref{ISI} introduces  new identities, and there are many additional identities in this paper which we believe will become very useful, in particular, for giving an elementary proof that finitely generated three transposition groups are finite.  Theorem \ref{Ider}  is new as well.

%%%%%%%%%%%%%%%%%%%%%%%%%%%%%%%%%%
\begin{Defsnot}\label{axis}
%%%%%%%%%%%%%%%%%%%%%%%%%%%%%%%
\begin{enumerate}
\item
Recall that an idempotent $c\in A$ is called {\it a primitive axis of Jordan type $\half,$} if $A=\ff c\oplus A_0(c)\oplus A_{\half}(c),$ and $(c,c)=1.$ Further, $c$ satisfies the following {\it fusion rules:} for $x_0,y_0\in A_0(c)$ and $x_{\half},y_{\half}\in A_{\half}(c),$ $x_0y_0\in A_0(c),\ x_0y_{\half}\in A_{\half}(c),$ and $x_{\half}y_{\half}\in \ff c+A_0(c).$ For short we will say that $c$ is an {\it axis}.

\item 
Throughout the paper, $a, b$ are fixed axes in $A.$ 

\item
Throughout the paper, except for \S\ref{SQP}, when we write $x=\ga_x a+x_0+x_{\half},\  x\in A,$ we mean $\ga_x\in\ff,\ x_{\gl}\in A_{\gl}(a),\ \gl\in\{0,\half\}.$  Lemma \ref{O} below shows that $\ga_x=(a,x).$ 

\item 
Throughout the paper, except for \S\ref{SQP}, when we write $x_{\gl},\ \gl\in\{0,\half\},$ (but there is no element $x$ around) we just mean $x_{\gl}\in A_{\gl}(a).$

\item 
Recall that for an axis $a$   the Miyamoto map $\gt_a: x\mapsto (a,x)a+x_0-x_{\half},$ where $x=(a,x)a+x_0+x_{\half}\in A,$ is an automorphism of order at most $2$ of $A.$
\end{enumerate}
\end{Defsnot}

%%%%%%%%%%%%%%%%%%%%%%%%%%%%%%%%
\begin{thm}[Theorem \ref{der}]\label{Ider}
%%%%%%%%%%%%%%%%%%%%%%%%%%%%%%%%%%%%%%5
Let $a, b$ be axes.  Then the following are equivalent
\begin{enumerate}\eroman
\item
$[L_a,L_b]$ is a derivation on $A.$
\item 
\begin{enumerate}
\item 
$b_{\half}(x_0y_0)=(b_{\half}x_0)y_0+(b_{\half}y_0)x_0,\ \forall x_0,y_0\in A_0(a).$

\item
$\Big(b_{\half}(x_{\half}y_0)\Big)_0=(b_{\half}x_{\half})y_0-\Big((b_{\half}y_0)x_{\half}\Big)_0 ,$ for all $x_{\half}\in A_{\half}(a)$ and $y_0\in A_0(a).$

\item 
$b_{\half}(x_{\half}y_{\half})+(b_{\half}x_{\half} )y_{\half} +(b_{\half}y_{\half} )x_{\half}=$\\
$\half(b,y_{\half} )x_{\half} +\half(b,x_{\half} )y_{\half}+(a,x_{\half}y_{\half})b_{\half},$\\ 
for all $x_{\half}, y_{\half}\in A_{\half}(a).$
\end{enumerate}
\end{enumerate}
\end{thm}

The following result is proven in \S5, and is essentially a consequence of \cite[Theorem 1.2 and Proposition 4.3]{d}, but here we prove it using elementary means.

%%%%%%%%%%%%%%%%%%%%%%%%%%%%%%%%%%%%%
\begin{thm}[Theorem \ref{0012}]\label{I0012}
%%%%%%%%%%%%%%%%%%%%%%%%%%%%%%%%%%%%%%%%
If $(a,b)\ne \frac 14,$ then
$b_{\half}(x_0y_0)=(b_{\half}x_0)y_0+ (b_{\half}y_0)x_0.$
\end{thm}

%%%%%%%%%%%%%%%%%%%%%%%%%%%%%%%%
\begin{thm}[Theorem \ref{SI}]\label{ISI}
%%%%%%%%%%%%%%%%%%%%%%%%%%%%%5
Let $x_0\in A_0(a)$ and $x_{\frac 12}\in A_{\frac 12}(a),$ then
\begin{enumerate}
\item 
$(x_{\frac 12}b_0)b_0=\frac 12(1-(a,b))x_{\frac 12}b_0.$

\item 
 $(x_0b_0)b_0=\frac 12(1-(a,b))x_0b_0+\frac 12(b,x_0)b_0.$

 \item 
 $(x_0b_{\frac 12}) b_{\frac 12}=\frac 12(a,b)(b,x_0)a+\frac 12(a,b)x_0b_0.$

 \item
 \begin{enumerate}
 \item 
$ \Big(x_{\half}b_{\half} \Big)_0b_{\half} =\frac 14(a,b)(1-(a,b))x_{\half}-\half(a,b)x_{\half}b_0+\frac 14(b,x_{\half})b_{\half} .$
\item
$(x_{\half}b_{\half})b_{\half}= \frac 14(a,b)(1-(a,b))x_{\half}-\half(a,b)x_{\half}b_0+\half(b,x_{\half})b_{\half}.$ 
\end{enumerate}

 \item 
$(x_{\frac 12}b_0)b_{\frac 12}=\frac 14(1-(a,b))(b,x_{\frac 12})a+\frac 14(b,x_{\frac 12})b_0.$ 

  \item 
$(x_{\frac 12}b_{\frac 12})b_0=\frac 12(1-(a,b))\Big(b_{\frac 12}x_{\frac 12}\Big)_0+\frac 14 (b,x_{\frac 12})b_0.$

\item
$(x_0b_0)b_{\frac 12}=\frac 12(1-(a,b))b_{\frac 12}x_0+\frac 14(b,x_0)b_{\frac 12}.$

\item
$(x_0b_{\frac 12})b_0=\frac 14(b,x_0)b_{\frac 12}.\quad$ 
\end{enumerate}
\end{thm}

Theorem \ref{2thms} (which infact are two separtate theorems) is proven in \S 6 and \S 7 by elemntary means, compare with \cite[Theorem 1.2 and Proposition 4.3]{d}).

%%%%%%%%%%%%%%%%%%%%%%%%%%%%%%%%%%%%%%%%%%%%%%
\begin{thm}[Theorems \ref{121212} and  \ref{thm120} respectively]\label{2thms}
%%%%%%%%%%%%%%%%%%%%%%%%%%%%%%%%%
Suppose $(a,b)\notin\{1,\frac 14\},$ then
\begin{enumerate}
 \item
$b_{\half}(x_{\half}y_{\half})+(b_{\half}x_{\half} )y_{\half} +(b_{\half}y_{\half} )x_{\half}=$\\
$\half(b,y_{\half} )x_{\half} +\half(b,x_{\half} )y_{\half}+(a,x_{\half}y_{\half})b_{\half}.$

\item
$\Big(b_{\half}(x_{\half}y_0 ) \Big)_0=(b_{\half}x_{\half})y_0-\Big((b_{\half}y_0)x_{\half}\Big)_0.$
\end{enumerate}
In particular, $[L_a,L_b]$ is a derivation.
\end{thm}

For case $(a,b)=1$, we have the following theorem, proven in \S 8 by elemntary means.
 
%%%%%%%%%%%%%%%%%%%%%%%%%%%%%%%%%%%
\begin{thm}
%%%%%%%%%%%%%%%%%%%%%%%%%%%%%%%%%%%
Assume $(a,b)=1.$  Then
\begin{enumerate}
\item 
If $b_0=0,$ then $[L_a,L_b]$ is a derivation.

\item 
If $\charc(\ff)\ne 3,$  then idetities (1) and (2) of Theorem \ref{2thms} hold.  In particular, $[L_a,L_b]$ is a derivation.
\end{enumerate}
\end{thm}

We left the case $(a,b)=\frac 14$ for future works.

%%%%%%%%%%%%%%%%%%%%%%%%%%%%%%%%
%%%%%%%%%%%%%%%%%%%%%%%%%%%%%%
%%%%%%%%%%%%%%%%%%%%%%%%%%%
\section{The polynomials $Q_c(x)$ and $P_c(x,y)$}\label{SQP}
%%%%%%%%%%%%%%%%%%%%%%%%%%%%%%%%%%%%%
%%%%%%%%%%%%%%%%%%%%%%%%%%%%%%
%%%%%%%%%%%%%%%%%%%%%%%%%%%%%%%
Recall Definition \ref{axis}(2).     Let  $c\in A$ be an idempotent.  Define, as in \cite{d},
\[
\textstyle{Q_c(x) :=c(cx) - \half(cx + (c, x)c),}
\]
\[
P_c(x,y) := 4(cx)(cy) - (c, y)cx - (cy)x - (c, x)cy - (cx)y - (c, xy)c + c(xy).
\]

%%%%%%%%%%%%%%%%%%%%%%%%%%%%%%%%%%%%%
\begin{lemma}\label{O}
%%%%%%%%%%%%%%%%%%%%%%%%%%%%%%%%
Let $z\in A,\ \gl,\gd\in\ff,\ x_{\gl}\in A_{\gl}(z),\ y_{\gd}\in A_{\gd}(z).$  If $\gd\ne\gl,$  then $(x_{\gl},y_{\gd})=0.$ In particular  $x=(a,x)a+x_0+x_{\half},$ where $x\in A,\ x_0\in A_0(a),\ x_{\half}\in A_{\half}(a).$
\end{lemma}
\begin{proof}
If $\gl=0,$ then, by hypothesis, $\gd\ne 0,$ and $(x_{\gl},y_{\gd})=(x_{\gl}, \frac{1}{\gd}zy_{\gd})=\frac{1}{\gd}(zx_{\gl},y_{\gd})=\frac{1}{\gd}(0,y_{\gd})=0.$  Assume $0\notin\{\gl,\gd\}.$  Then
\[
\textstyle{(x_{\gl},y_{\gd})=(\frac{1}{\gl}zx_{\gl},\ y_{\gd})=(\frac{1}{\gl}x_{\gl},\ zy_{\gd})=\frac{\gd}{\gl}(x_{\gl},y_{\gd}).}
\]
Hence if $\gl\ne\gd,$ then $(x_{\gl},y_{\gd})= 0.$    
\end{proof}

%%%%%%%%%%%%%%%%%%%%%%%%%%%%%%
\begin{lemma}[Serss' Lemma]\label{seress}
%%%%%%%%%%%%%%%%%%%%%%%%%%%%%%%%%
Let $a$ be an axis.  Then $a(xy)=(ax)y,$ for all $x\in A$ and $y\in \ff a+A_0(a).$
\end{lemma}
\begin{proof}
Write $y=(a,y)a+y_0,$  and $x=(a,x)a+x_0+x_{\half},$ with $x_0,y_0\in A_0(a)$ and $x_{\half}\in A_{\half}(a).$  Then
\[
\textstyle{xy=(a,x)(a,y)a+x_0y_0+\half (a,y)x_{\half}+x_{\half}y_0,}
\]
so
\[
\textstyle{a(xy)=(a,x)(a,y)a+\frac 14 (a,y)x_{\half}+\half x_{\half}y_0,}
\]
while
\[
\textstyle{(ax)y=((a,x)a+\half x_{\half})((a,y)a+y_0)=(a,x)(a,y)a+\frac 14(a,y)x_{\half}+\half x_{\half}y_0.}\qedhere
\]
\end{proof}
%%%%%%%%%%%%%%%%%%%%%%%%%%%%%%
\begin{lemma}\label{QP}
%%%%%%%%%%%%%%%%%%%%%%%%%%%%%
Let $c\in A$ be an idempotent, then
\begin{enumerate} 
\item 
If $\charc(\ff)\ne 3,$ and $P_c(c,c)=0,$ then $(c,c)=1.$

\item 
The following conditions on $c$ are equivalent:
\begin{enumerate}
\item
$Q_c(x)=0,$ for all $x\in A.$

\item
$A=\ff c\oplus A_0(c)\oplus A_{\half}(c),$ and $(c,c)=1.$  
\end{enumerate}

\item 
Assume $(c,c)=1.$  Then
\begin{enumerate}
\item
$P_c(c,y)=4Q_c(y).$

\item 
$c$ is a primitive  axis of Jordan type $\half$ iff $P_c(x,y)=0,$ for all $x,y\in A.$
\end{enumerate}
\end{enumerate}
\end{lemma}
\begin{proof}
(1):\ By definition,
$P_c(c,c)=4c - (c, c)c - c - (c,c)c - c - (c, c)c + c=3c-3(c,c)c,$ so (1) holds.
\medskip

\noindent
(2)  $(a)\implies (b):$ Assume  that $Q_c(x)=0,$ for all $x\in A.$  Then $cx-\half x-\half(c,x)c\in A_0(c)$, and $c(2cx-2(c,x)c)=2(c(cx))-2(c,x)c)=cx+(c,x)c-2(c,x)c=cx-(c,x)c.$  Thus $2cx-2(c,x)c\in A_{\half}(c).$  Also $x=(c,x) c-2 (cx-\half x-\half(c,x)c)+2cx-2(c,x)c.$  Hence, $A=\ff c\oplus A_0(c)\oplus A_{\half}(c).$

Since $0=Q_c(c)=\half c-\half(c,c)c,$ we see that $(c,c)=1.$
\medskip

\noindent
$(b)\implies (a):$  Assume that $(b)$ holds.  Since $(c,c)=1,$ $Q_c(c)=0.$   Let $x_0\in A_0(c)$ and $x_{\half}\in A_{\half}(c).$ By Lemma \ref{O}, $(c,x_0)=(c,x_{\half})=0.$ Now $Q_c(x_0)=c(cx_0) - \half(cx_0) -\half (c, x_0)c)=0,$  and  $Q_c(x_{\half})=c(cx_{\half}) - \half cx_{\half} + (c, x_{\half})c)=\frac 14 x_{\half}-\frac 14 x_{\half}-\half(c,x_{\half})c=0.$  
\medskip

\noindent
(3)(a):\  We have
\begin{align*}
P_c(c,y)&=4c(cy) - (c, y)c - (cy)c - (c,c)cy - cy - (c, cy)c + c(cy)\\
&=4c(cy)-2cy-2(c,y)c=4Q_c(y),
\end{align*}
because $(c,c)=1,$ and $(c,cy)=(c,y),$ since $(\ ,\ )$ is Frobenius.
\medskip

\noindent
(b):\ Assume that $A=\ff c\oplus A_0(c)\oplus A_{\half}(c).$
Let $x_0, y_0\in A_0(c)$ and $x_{\half}, y_{\half}\in A_{\half}(c).$ By Lemma \ref{O}, $(c,x_0)=(c,y_0)=(c,x_{\half})=(c,y_{\half})=0.$  Also since $(\ ,\ )$ is Frobenius, $(c,x_0y_0)=0=(c,x_0y_{\half}).$ Hence
\begin{equation}\label{eqP00}
\begin{aligned}
P_c(x_0,y_0) &=4(cx_0)(cy_0) - (c, y_0)cx_0 - (cy_0)x_0\\
&- (c, x_0)cy_0 - (cx_0)y_0 - (c, x_0y_0)c + c(x_0y_0)\\
&=0-0-0-0-0-0+c(x_0y_0)\\
&=c(x_0y_0),
\end{aligned}
\end{equation}
\begin{equation}\label{eqP120}
\begin{aligned}
P_c(x_0,y_{\half}) &= 4(cx_0)(cy_{\half}) - (c, y_{\half})cx_0 - (cy_{\half})x_0 - (c, x_0)cy_{\half}\\
&- (cx_0)y_{\half} - (c, x_0y_{\half})c + c(x_0y_{\half})\\
&=\textstyle{0-0-\half x_0y_{\half}-0-0-0+c(x_0y_{\half})}\\
&=\textstyle{-\half x_0y_{\half}+c(x_0y_{\half}),}
\end{aligned}
\end{equation}
and
\begin{equation}\label{eqP1212a}
\begin{aligned}
 P_c(x_{\half},y_{\half})  &= 4(cx_{\half})(cy_{\half}) - (c, y_{\half})cx_{\half} - (cy_{\half})x_{\half} \\
&- (c, x_{\half})cy_{\half}- (cx_{\half})y_{\half}   - (c, x_{\half}y_{\half})c + c(x_{\half}y_{\half})\\
&=\textstyle{x_{\half}y_{\half} - 0 - \half y_{\half}x_{\half} - 0 - \half x_{\half}y_{\half}}\\
&- (c,x_{\half}y_{\half})c + c(x_{\half}y_{\half})\\
&=- (c,x_{\half}y_{\half})c + c(x_{\half}y_{\half}).
\end{aligned}
\end{equation}
Assume first that $P_c(x,y)=0,$ for all $x,y\in A.$ By part (a) and by (2), the equivalent condition (2)(b) holds.  
It remains to show the fusion rules.    But these follow from \eqref{eqP00}, \eqref{eqP120} and \eqref{eqP1212a}.

Assume next that $c$ is a primitive  axis of Jordan type $\half$.  By (2), and by part (a), $P_c(c,y)=P_c(x,c)=0,$ for all $x,y\in A.$  

Using Lemma \ref{O} and the fusion rule $x_{\half}y_{\half}\in\ff c+A_0(c),$ we see that $c(x_{\half}y_{\half})=(c,x_{\half}y_{\half})c.$

Now, by the fusion rules and by \eqref{eqP00}, \eqref{eqP120} and \eqref{eqP1212a}, $P(x_0,y_0)=P(x_0,y_{\half})=P(y_{\half},x_0)=P(x_{\half},y_{\half})=0,$ for all $x_0,y_0\in A_0(c)$ and $x_{\half},y_{\half}\in A_{\half}(c).$ Hence, $P(x,y)=0,$ for all $x,y\in A.$
\end{proof}

The following lemma will be used in the next section.

%%%%%%%%%%%%%%%%%%%%%%%%%%%%%%%%%%%%%%%
\begin{lemma}\label{x0x12}
%%%%%%%%%%%%%%%%%%%%%%%%%%%%%%%%%%%%%
Let $a\in A$ be an axis. For $x\in A,$ write $x=(a,x)a+x_0+x_{\half},$ with $x_0\in A_0(a)$ and $x_{\half}\in A_{\half}(a).$
\begin{enumerate}
\item
$x_0=x-2ax+(a,x)a.$
\item
$x_{\half}=2ax-2(a,x)a.$
\end{enumerate}
\end{lemma}
\begin{proof}
 (1) $ax=(a,x)a+\half x_{\half}.$ Hence
$x+(a,x)a-2ax=(a,x)a+x_0+x_{\half}+(a,x)a-2(a,x)a-x_{\half} =x_0.$
\medskip

\noindent
(2)  $2ax-2(a,x)a=2(a,x)a+x_{\half}-2(a,x)a=x_{\half}.$  
\end{proof}
%%%%%%%%%%%%%%%%%%%%%%%%%%%
%%%%%%%%%%%%%%%%%%%%%%%%%%%
%%%%%%%%%%%%%%%%%%%%%%%%%
\section{Derivations}\label{DER}
%%%%%%%%%%%%%%%%%%%%%%%%%%%%%%
%%%%%%%%%%%%%%%%%%%%%%%%%%%%%%
%%%%%%%%%%%%%%%%%%%%%%%%%%
We recall Definition \ref{axis}. In particular, in this section $a$ is an axis. 
%%%%%%%%%%%%%%%%%%%%%%%%%%%%%%%%%%
\begin{Def}
%%%%%%%%%%%%%%%%%%%%%%%%%%%%%
\begin{enumerate}
\item
A derivation on $A$ is a $\ff$-linear map $D : A \to A$ that satisfies Leibniz's law:  $D ( x y ) = x D ( y ) + D ( x ) y.$
\item
For $x,y\in A$ we denote $L_x(y)=xy,$ and $[L_x,L_y]=L_xL_y-L_yL_x,$ so  $[L_x,L_y](z)=x(yz)-y(xz).$
%\[
%D_{a,b}(xy)=a(b(xy))-b(a(xy))=x[(a(by))-b(ay)]+y[a(bx)-b(ax)]
%\]
\end{enumerate}
\end{Def}

%%%%%%%%%%%%%%%%%%%%%%%%%%%%%%%%%%
\begin{lemma}\label{der1}
%%%%%%%%%%%%%%%%%%%%%%%%%%%%%%%%
Let $v\in A_{\half}(a)$  and $x\in A.$ Write $x=(a,x)a+x_0+x_{\half},$ then
\begin{enumerate}
\item
$[L_a,L_v](x)=-2v(ax) + \half(vx + (v, x)a + (a, x)v).$

\item
$[L_a,L_v](x)=\half(v,x)a-\half vx_{\half}-\frac{1}{4}(a,x)v+\half vx_0.$

\item
$\Big([L_a,L_v](x)\Big)_0=-\half \Big(vx_{\half}\Big)_0.$

\item
$\Big([L_a,L_v](x)\Big)_{\half}=\frac{1}{2}vx_0-\frac{1}{4}(a,x)v .$

\item
$[L_a,L_v](x_0)=\half (vx_0).$

\item
$[L_a,L_v](x_{\half})=\half (v,x_{\half})a-\half vx_{\half}.$
\end{enumerate}
\end{lemma}
\begin{proof}
(1)\ First note that $ax=(a,x)a+\half x_{\half}=\frac{1}{4} (x - x^{\tau_a}) + (a, x)a,$ hence
\begin{align*}
a(vx) &= \textstyle{\frac{1}{4} (vx - (vx)^{\tau_a}) + (a, vx)a=\frac{1}{4} (vx + vx^{\tau_a}) +\half (v, x)a}\\
&= \textstyle{\frac{1}{4} (vx + v(x + 4(a,x)a - 4(ax))) + \half (v, x)a}\\
&= \textstyle{\frac{1}{4} (vx + vx + 2(a, x)v - 4v(ax)) + \half (v, x)a}\\
&=\textstyle{-v(ax) + \half (vx + (v, x)a + (a, x)v).}
\end{align*}
\medskip

\noindent
(2)  By (1),
\begin{align*}
[L_a,L_v](x)&=\textstyle{-2v(ax)+\half(vx + (v, x)a + (a, x)v)}\\
&=\textstyle{-2v((a,x)a+\half x_{\half})+\half(\half(a,x)v+vx_0+vx_{\half}+(v,x)a+(a,x)v)}\\
&=\textstyle{-(a,x)v-vx_{\half}+\frac{3}{4}(a,x)v+\half vx_0+\half vx_{\half}+\half (v,x)a}\\
&=\textstyle{\half(v,x)a-\half vx_{\half}-\frac{1}{4}(a,x)v+\half vx_0.}
\end{align*}
\medskip

\noindent
(3), (4), (5) and (6) follow from (2).
\end{proof}

%%%%%%%%%%%%%%%%%%%%%%%%%%%%%%%%%%
\begin{lemma}\label{der2}
%%%%%%%%%%%%%%%%%%%%%%%%%%%%%%%%%%
Notation as in Lemma \ref{der1},
\begin{enumerate}
\item
$\Big([L_a,L_v](xy_0)\Big)_{\half}=\half v(y_0x_0).$

\item
$\Big(\big([L_a,L_v](x)\big)y_0\Big)_{\half}= \half(vx_0)y_0-\frac{1}{4}(a,x)vy_0.$

\item
$\Big(\big([L_a,L_v](y_0)\big)x\Big)_{\half}=\half (vy_0)x_0+\frac{1}{4}(a,x)vy_0.$

\item
$\Big([L_a,L_v](xy_0)\Big)_{\half}=\Big(([L_a,L_v](x))y_0\Big)_{\half}+\Big(([L_a,L_v](y_0))x\Big)_{\half},$ iff
\[
v(x_0y_0)=(vx_0)y_0+(vy_0)x_0.
\]
\end{enumerate}

\end{lemma}
\begin{proof}
(1)  This follows from Lemma \ref{der1}(4), since $(a, xy_0)=0,$ and $\Big(xy_0\Big)_0=x_0y_0.$
\medskip

\noindent
(2)  Using Lemma \ref{der1}(4) we get,
\[
\textstyle{\Big(\big([L_a,L_v](x)\big)y_0\Big)_{\half}=\Big([L_a,L_v](x)\Big)_{\half}y_0= \half(vx_0)y_0-\frac{1}{4}(a,x)vy_0.}
\]
\medskip

\noindent
(3) By Lemma \ref{der1}(5),
$
\Big(\big([L_a,L_v](y_0)\big)x\Big)_{\half}=\half \Big((vy_0)x\Big)_{\half}=\\
\half((vy_0)x_0+\half (a,x)vy_0)=\half (vy_0)x_0+\frac{1}{4}(a,x)vy_0.
$
\medskip

\noindent
(4) Follows from (1), (2) and (3).
\end{proof}

%%%%%%%%%%%%%%%%%%%%%%%%%%%%%%%%
\begin{lemma}\label{der3}
%%%%%%%%%%%%%%%%%%%%%%%%%%%%%
Notation as in Lemma \ref{der1},
\begin{enumerate}
\item
$\Big([L_a,L_v](xy_0)\Big)_0=-\half \Big(v(x_{\half}y_0)\Big)_0.$

\item
$\Big(\big([L_a,L_v](y_0)\big)x\Big)_0=\half\Big((vy_0)x_{\half}\Big)_0.$

%\item
%$([L_a,L_v](y_0))_0=0.$

\item
$\Big(\big([L_a,L_v](x)\big)y_0\Big)_0=-\half (vx_{\half})y_0.$

\item
$\Big([L_a,L_v](xy_0)\Big)_0=\Big([L_a,L_v](x)\Big)_0y_0+\Big([L_a,L_v](y_0)x\Big)_0,$ iff
\[
 \Big(v(x_{\half}y_0)\Big)_0=(vx_{\half})y_0- \Big((vy_0)x_{\half}\Big)_0. 
\]
\end{enumerate}
\end{lemma}
\begin{proof}
(1)  This follows from Lemma \ref{der1}(3), since $\Big(xy_0\Big)_{\half}=x_{\half}y_0.$
\medskip

\noindent
(2) By Lemma \ref{der1}(5), $\Big(\big([L_a,L_v](y_0)\big) x\Big)_0=\half\Big((vy_0)x\Big)_0=\half\Big((vy_0)x_{\half}\Big)_0.$
\medskip

\noindent
(3) By Lemma \ref{der1}(3), $\Big(\big([L_a,L_v](x)\big)y_0\Big)_0=\Big([L_a,L_v](x)\Big)_0y_0=-\half \Big(vx_{\half}\Big)_0y_0=-\half(vx_{\half})y_0.$
\medskip

\noindent
(4) Follows from (1), (2) and (3).

\end{proof}

%%%%%%%%%%%%%%%%%%%%%%%%%%%%%%%%%%%%%
\begin{lemma}\label{der4}
%%%%%%%%%%%%%%%%%%%%%%%%%%%%%%%%%%%
Notation as in Lemma \ref{der1}.
\begin{enumerate}
\item
$\Big([L_a,L_v](xy_{\half})\Big)_{\half}=\half v(x_{\half}y_{\half})-\half(a,x_{\half}y_{\half})v.$

\item
$\Big(\big([L_a,L_v](x)\big)y_{\half}\Big)_{\half}=\frac{1}{4}(v,x)y_{\half}-\half(vx_{\half})y_{\half}.$

\item
$\Big(\big([L_a,L_v](y_{\half})\big)x\Big)_{\half}=\frac{1}{4}(v,y)x_{\half}-\half(vy_{\half})x_{\half}.$

\item
$\Big([L_a,L_v](xy_{\half})\Big)_{\half}=\Big(\big([L_a,L_v](x)\big)y_{\half}\Big)_{\half}+\Big(\big([L_a,L_v](y_{\half})\big)x\Big)_{\half},$ iff
\[
\textstyle{v(x_{\half}y_{\half})+(vx_{\half})y_{\half}+(vy_{\half})x_{\half} =\half(v,x)y_{\half}+\half(v,y)x_{\half}+(a,x_{\half}y_{\half})v.}
\]
%
%\item
%Taking $x=x_{\half}=y_{\half}=v$ in (5) we get $3(v^2)v=\frac{3}{2}(v,v)v.$
%
%\item
%If $([L_a,L_v](xy_{\half}))_{\half}=([L_a,L_v](x)y_{\half})_{\half}+([L_a,L_v](y_{\half})x)_{\half},$ iff a derivation, and $A$ is power associative, and $\charc(\ff)\ne 3,$
%then either $(v^2)^2=0,$ or setting $\frac{1}{\mu}:=\half (v,v),$ we see that $\mu v^2$ is an idempotent.
\end{enumerate}
\end{lemma}
\begin{proof}
(1)  By Lemma \ref{der1}(4),
\begin{gather*}
\textstyle{\Big([L_a,L_v](xy_{\half})\Big)_{\half}=\frac{1}{2}v\Big(xy_{\half}\Big)_0-\frac{1}{4}(a,xy_{\half})v=\half v\Big(x_{\half}y_{\half}\Big)_0-\frac{1}{4}(a,x_{\half}y_{\half})v}\\
=\textstyle{\half v(x_{\half}y_{\half})-\half(a,x_{\half}y_{\half})av-\frac{1}{4}(a,x_{\half}y_{\half})v=\half v(x_{\half}y_{\half})-\half(a,x_{\half}y_{\half})v.}
\end{gather*}
\medskip

\noindent
(2) By Lemma \ref{der1}(2),\\
%\begin{gather*}
$\Big(\big([L_a,L_v](x)\big)y_{\half}\Big)_{\half}=(\half(v,x)a-\half vx_{\half})y_{\half}=\frac 14 (v,x)y_{\half}-\half (vx_{\half})y_{\half}.$
%\end{gather*}
\medskip

\noindent
(3) By Lemma \ref{der1}(6),
\[
\textstyle{\Big([L_a,L_v](y_{\half})x\Big)_{\half}=(\half(v, y_{\half})a-\half vy_{\half})x_{\half}=\frac{1}{4}(v,y)x_{\half}-\half(vy_{\half})x_{\half}.}
\]
\medskip

\noindent
(4)  Follows from (1), (2) and (3).
%\medskip
%
%\noindent
%(5)  We get
%\[
%v(v^2)-(a,v^2)v=\half(v,v)v+\half(v,v)v-(v^2)v-(v^2)v,
%\]
%since $(a,v^2)v=\half (v,v)v,$ (5) holds.
%\medskip
%
%\noindent
%(6) Follows from (5).
\end{proof}

%%%%%%%%%%%%%%%%%%%%%%%%%%%%%%%%%
\begin{lemma}\label{der5}
%%%%%%%%%%%%%%%%%%%%%%%%%%%%%%%
Notation as in Lemma \ref{der1},
\begin{enumerate}
\item
$\Big([L_a,L_v](xy_{\half})\Big)_0=-\frac 14 (a,x)\Big(vy_{\half}\Big)_0-\half\Big(v(x_0y_{\half})\Big)_0$

\item
$\Big(\big([L_a,L_v](x)\big)y_{\half}\Big)_0=\frac{1}{2}\Big((vx_0)y_{\half}\Big)_0-\frac 14 (a,x) \Big(vy_{\half}\Big)_0. $

\item
$\Big(\big([L_a,L_v](y_{\half})\big)x\Big)_0=-\half (vy_{\half})x_0.$

\item
$\Big([L_a,L_v](xy_{\half})\Big)_0=\Big(\big([L_a,L_v](x)\big)y_{\half}\Big)_0+\Big(\big([L_a,L_v](y_{\half})\big)x\Big)_0,$ iff
\[
\Big(v (y_{\half}x_0) \Big)_0=(vy_{\half}  )x_0-\Big((vx_0)y_{\half}\Big)_0   .
\]
(Note that is exactly  the same condition of Lemma \ref{der3}(4).)
%$(v(y_{\half}x_0))_0=(vy_{\half})x_0.$
\end{enumerate}
\end{lemma}
\begin{proof}
(1) By Lemma \ref{der1}(3),
\begin{gather*}
\textstyle{\Big([L_a,L_v](xy_{\half})\Big)_0=-\half \Big(v\Big(xy_{\half}\Big)_{\half}\Big)_0=}\\
\textstyle{-\half \Big(v((a,x)a+x_0)y_{\half})\Big)_0=
-\frac 14 (a,x)\Big(vy_{\half}\Big)_0-\half\Big(v(x_0y_{\half})\Big)_0.}\\
\end{gather*}
\medskip

\noindent
(2)  By Lemma \ref{der1}(4),
\begin{gather*}
\Big(\big([L_a,L_v](x)\big)y_{\half}\Big)_0=\Big(\Big([L_a,L_v](x)\big)\Big)_{\half}y_{\half}\Big)_0\\
=\textstyle{\Big((\frac{1}{2}vx_0-\frac{1}{4}(a,x)v)y_{\half}\Big)_0=\frac{1}{2}\Big((vx_0)y_{\half}\Big)_0-\frac 14 (a,x) \Big(vy_{\half}\Big)_0.}
\end{gather*}
\medskip

\noindent
(3)  By Lemma \ref{der1}(6),
\begin{gather*}
\textstyle{\Big(\big([L_a,L_v](y_{\half})\big)x\Big)_0=\Big(( \half (v,y_{\half})a-\half vy_{\half})x\Big)_0}\\
=\textstyle{-\half \Big((vy_{\half})x\Big)_0=-\half(vy_{\half})x_0.}
\end{gather*}
\medskip

\noindent
(4) This is immediate from (1), (2) and (3).
%\medskip
%
%\noindent
%(5)  This follows from (4), using  Lemma \ref{der2}(5) (and interchanging $x$ and~$y.$)
\end{proof}

%%%%%%%%%%%%%%%%%%%%%%%%%%%%%%%%
\begin{thm}\label{der}
%%%%%%%%%%%%%%%%%%%%%%%%%%%%%%%%%%%%%%5
Let $a, b$ be axes.  Then the following are equivalent
\begin{enumerate}\eroman
\item
$[L_a,L_b]$ is a derivation on $A.$

\item 
\begin{enumerate}
\item 
$b_{\half}(x_0y_0)=(b_{\half}x_0)y_0+(b_{\half}y_0)x_0,\ \forall x_0,y_0\in A_0(a).$

\item
$\Big(b_{\half}(x_{\half}y_0)\Big)_0=(b_{\half}x_{\half})y_0-\Big((b_{\half}y_0)x_{\half}\Big)_0 ,$ for all $x_{\half}\in A_{\half}(a)$ and $y_0\in A_0(a).$

\item 
$b_{\half}(x_{\half}y_{\half})+(b_{\half}x_{\half} )y_{\half} +(b_{\half}y_{\half} )x_{\half}=$\\
$\half(b,y_{\half} )x_{\half} +\half(b,x_{\half} )y_{\half}+(a,x_{\half}y_{\half})b_{\half},$\\ 
for all $x_{\half}, y_{\half}\in A_{\half}(a).$
\end{enumerate}
\end{enumerate}
\end{thm}
\begin{proof}
First, note that
\begin{equation}\label{lal0}
[L_a,L_{x_0}]=0,\text{ for all }x_0\in A_0(a),
\end{equation}
because by Seress' Lemma \ref{seress}, $[L_a,L_{x_0}](x)=a(x_0x)-x_0(ax)=0,$ for all $x\in A.$  Hence, since $[L_a,L_a]=0,$ we see that  for $x\in A,$ $[L_a, L_x]$ is a derivation iff $[L_a,L_{x_{\half}}]$ is a derivation.     
Let $v\in A_{\half}(a).$ By Lemma \ref{der1}(5\&6),
\begin{equation*}
\begin{aligned}
&\textstyle{[L_a,L_v](a)=a(av)-v(aa)=-\frac 14v,}\\
&\textstyle{[L_a,L_v](x_{\half})=a(vx_{\half})-\half vx_{\half},\text{ and }
[L_a,L_v](x_0)=\half vx_0.}
\end{aligned}
\end{equation*}
Hence,
\begin{gather*}
\textstyle{\big([L_a,L_v](x_0)\big)a+\big([L_a,L_v](a)\big)x_0=
\frac 14 vx_0-\frac 14 vx_0=0=[L_a,L_v](ax_0).}\\
\end{gather*}
Also,
\begin{gather*}
\textstyle{\big([L_a,L_v](a)\big)x_{\half}+a\big([L_a,L_v](x_{\half})\big)=-\frac 14 vx_{\half}+\half a(vx_{\half})=[L_a,L_v](ax_{\half}).}
\end{gather*}
Finally, 
\[
\textstyle{[L_a,L_v](a)=-\frac 18 v-\frac 18 v=a\big([L_a,L_v](a)\big)+a\big([L_a,L_v](a)\big)=[L_a,L_v](aa).}
\]
Thus $[L_a,L_v](xa)=x[L_a,L_v](a)+a[L_a,L_v](x).$
\medskip

Next we show that
\[
(a,[L_a,L_v](xy))=(a, \left([L_a,L_v](x)\right)y)+(a, \left([L_a,L_v](y)\right)x).
\]
By Lemma \ref{der1}(2),
\[
\textstyle{(a, [L_a,L_v](x))=\half(v,x)-\half(a,vx_{\half})=\frac 14(v,x),}
\]
so 
\begin{gather*}
\textstyle{(a, [L_a,L_v](xy))=\frac 14(v,xy)=\frac 14(v,\half(a,x)y_{\half}+\half(a,y)x_{\half}+x_0y_{\half}+y_0x_{\half})}\\
\textstyle{=\frac 18(a,x)(v,y)+\frac 18(a,y)(v,x)+\frac 14(v,x_0y_{\half})+\frac 14(v,y_0x_{\half}).}\\
%\textstyle{=\frac 18(a,x)(v,y)+\frac 18(a,y)(v,x)+\frac 14(v,x_0y_{\half})+\frac 14(v,y_0x_{\half}).}
\end{gather*}
Also, using Lemma \ref{der1}(2),
\begin{gather*}
\textstyle{(a,\left([L_a,L_v](x)\right)y)=(y,a\big([L_a,L_v](x)\big))=(y,\frac 14(v,x_{\half})a-\frac 18(a,x)v+\frac 14(vx_0))}\\
\textstyle{=\frac 14(a,y)(v,x)-\frac 18(a,x)(v,y)+\frac 14(y_{\half},vx_0).}
\end{gather*}
so
\begin{gather*}
\textstyle{(a,\big([L_a,L_v](x)\big)y)+(a,\big([L_a,L_v](y)\big)x)=\frac 14(a,y)(v,x)-\frac 18(a,x)(v,y)+\frac 14(y_{\half},vx_0)}\\
\textstyle{+\frac 14(a,x)(v,y)-\frac 18(a,y)(v,x)+\frac 14(x_{\half},vy_0)=}\\
\textstyle{\frac 18(a,x)(v,y)+\frac 18(a,y)(v,x)+\frac 14(x_{\half},vy_0)+\frac 14(y_{\half},vx_0).}
\end{gather*}

It remains to show that $\Big([L_a,L_{b_{\half}}](xy_{\gd})\Big)_{\gl}=\Big(\big([L_a,L_{b_{\half}}](x)\big)y_{\gd}\Big)_{\gl}+\Big(\big([L_a,L_{b_{\half}}](y_{\gd})\big)x\Big)_{\gl},$  for $\gl,\gd\in\{0,\half\},$ is equivalent to (ii) (a), (b) and (c).  But this follows from Lemma  \ref{der2}, \ref{der3}, \ref{der4} and \ref{der5}.
\end{proof}

%%%%%%%%%%%%%%%%%%%%%%%%%%%%%%%%%%%%%%%%%%%%
%%%%%%%%%%%%%%%%%%%%%%%%%%%%%%%%%%%%%%%%%%%
%%%%%%%%%%%%%%%%%%%%%%%%%%%%%%%%%%%%%%%%%%%%
\section{Obtaining identities in $A$}
%%%%%%%%%%%%%%%%%%%%%%%%%%%%%%%%%%%%%%%%%%%
%%%%%%%%%%%%%%%%%%%%%%%%%%%%%%%%%%%%%%%%
%%%%%%%%%%%%%%%%%%%%%%%%%%%%%%%%%%%%%%%%

We recall Definition \ref{axis}.  In particular $a$ and $b$ are axes.   

%%%%%%%%%%%%%%%%%%%%%%%%%%%
\begin{lemma}\label{Q}
%%%%%%%%%%%%%%%%%%%%%%%%%%%%
\begin{enumerate}
\item
Let $x\in A.$  Then,
\begin{gather*}
0=(a,b)ax+(a,b)(b,x)ab-2(a,b)a(bx)+b_0x +(b,x)bb_0\\
 -2(bx)b_0+b_{\half}x+(b,x)bb_{\half}-2(bx)b_{\half}.
\end{gather*}

\item
$\Big(b_{\half}^2\Big)_0=(a,b)\Big[b_0^2+\Big(b_{\half}^2\Big)_0\Big].$

\item
\begin{enumerate}
\item
$(b_{\half},b_{\half})=(b,b_{\half})=2(a,b)(1-(a,b)).$

\item
$(b,b_0)=(1-(a,b))^2.$

\item
$b_0^2=(1-(a,b))b_0$ and $\Big(b_{\half}^2\Big)_0=(a,b)b_0.$

\item
$b_{\half}^2=(a,b)(1-(a,b))a+(a,b)b_0.$
\item
$b_{\half}b_0=\half(1-(a,b))b_{\half}.$
\end{enumerate}

\item%4
$2\Big(b_{\half}(b_{\half}x_0)\Big)_0=x_0b_0+(b,x_0)b_0-2(x_0b_0)b_0 .$
 
\item%5
\begin{enumerate}
\item
$(1-(a,b))\Big(b_{\half}x_{\half} \Big)_0-2\Big(b_{\half}x_{\half} \Big)_0b_0 +(b,x_{\half})b_0-2\Big(b_{\half}(x_{\half}b_0) \Big)_0=0.$

\item
$\half(a,b)(1-(a,b))x_{\half}-2\Big(b_{\half}x_{\half}\Big)_0b_{\half} +\half(b,x_{\half})b_{\half}
-2(x_{\half}b_0)b_0+(1-2(a,b))x_{\half}b_0=0.$
\end{enumerate}
\end{enumerate}
\end{lemma}
\begin{proof}
(1) Write $b=(a,b)a+b_0+b_{\half}.$ Then, by Lemma \ref{x0x12}(2) applied to $b,$
\begin{gather*}
0=b(x+(b,x)b-2bx)=((a,b)a+b_0+b_{\half})(x+(b,x)b-2bx)\\
=(a,b)ax+(a,b)(b,x)ab-2(a,b)a(bx)+b_0x +(b,x)bb_0\\
 -2(bx)b_0+b_{\half}x+(b,x)bb_{\half}-2(bx)b_{\half}.
\end{gather*}
\medskip

\noindent
(2)  Replacing $x=a$ in (1) we get
\begin{gather*}
0=(a,b)a+(a,b)^2ab-2(a,b)a(ab) +(a,b)bb_0\\
\textstyle{ -2(ab)b_0+\half b_{\half}+(a,b)bb_{\half}-2(ab)b_{\half}}\\
 =\textstyle{w+\half(a,b)^2b_{\half}-\half(a,b)b_{\half}+(a,b)b_{\half}b_0} \\
 \textstyle{ -b_{\half}b_0+\half b_{\half}+(a,b)b_0b_{\half}+\half(a,b)^2b_{\half}-(a,b)b_{\half},}
\end{gather*}
where,
\begin{gather*}
w=(a,b)a+(a,b)^3a-2(a,b)^2a+(a,b)b_0^2+(a,b)b_{\half}^2-b_{\half}^2\\
\textstyle{=(a,b)((a,b)-1)^2a+\half((a,b)-1)(b_{\half},b_{\half})a}\\
+(a,b)b_0^2+((a,b)-1)\Big(b_{\half}^2\Big)_0.
\end{gather*}
Since $w=0,$ we see that $(a,b)(b_0^2+\Big(b_{\half}^2\Big)_0)=\Big(b_{\half}^2\Big)_0.$
\medskip

\noindent
(3) Since $b$ is an idempotent we have
\begin{align*}
b&=b^2=((a,b)a+b_0+b_{\half})^2=(a,b)^2a+b_0^2+b_{\half}^2+(a,b)b_{\half}+2b_0b_{\half}\\
&=(a,b)^2a+(a,b_{\half}^2)a+b_0^2+\Big(b_{\half}^2\Big)_0+(a,b)b_{\half}+2b_0b_{\half}.
\end{align*}
Since $(a,b_{\half}^2)=\half(b_{\half},b_{\half})=\half (b,b_{\half}),$ comparing the components of both sides yields (a) and (e) and $b_0^2+\Big(b_{\half}^2\Big)_0=b_0.$ Now by (2),  $\Big(b_{\half}^2\Big)_0=(a,b)b_0,$ so (c) holds.
Since
\[
1=(b,b)=(a,b)^2+(b_0,b_0)+(b_{\half},b_{\half}),
\]
(a) implies $(b,b_0)=(b_0,b_0)=(1-(a,b))^2,$ so (b) holds. Then, since $b_{\half}^2=(a,b_{\half}^2)a+\Big(b_{\half}^2\Big)_0,$ we see that (d) holds.
\medskip

\noindent
In the proofs of  (4) and (5) ahead, $w$ stands for the $1$-component plus the $0$-component of the right side of the equation.  
%Also, parts (a), (b) and (c) follows from the fact that the $1$-component, $0$-component and $\half$-component respectively, of  the right hand side of the equation equal $0.$
\medskip

\noindent
(4)  Replacing $x=x_0$ in (1) we get,
\begin{gather*}
0=(a,b)(b,x_0)ab-2(a,b)a(bx_0)+b_0x_0+(b,x_0)bb_0 \\
-2(bx_0)b_0+b_{\half}x_0+(b,x_0)bb_{\half}-2(bx_0)b_{\half}\\
=w+z_{\half},
%\half(a,b)(b,x_0)b_{\half}-(a,b)b_{\half}x_0+(b,x_0)b_{\half}b_0\\
%-2(b_{\half}x_0)b_0+b_{\half}x_0+\half(a,b)(b,x_0)b_{\half}+%(b,x_0)b_0b_{\half}-2(b_0x_0)b_{\half}.
\end{gather*}
with $w\in \ff a+A_0(a)$ and $z_{\half}\in A_{\half}(a),$ where,
\begin{gather*}
w=(a,b)^2(b,x_0)a+b_0x_0+(b,x_0)b_0^2-2(x_0b_0)b_0+(b,x_0)b_{\half}^2-2b_{\half}(b_{\half}x_0) \\
=(a,b)^2(b,x_0)a+x_0b_0+ (b,x_0)(1-(a,b))b_0-2(x_0b_0)b_0\\
+(b,x_0)(a,b)(1-(a,b))a+(b,x_0)(a,b)b_0-(a,b)(b,x_0)a-2\Big(b_{\half}(b_{\half}x_0) \Big)_0\\
=x_0b_0+(b,x_0)b_0-2(x_0b_0)b_0-2\Big(b_{\half}(b_{\half}x_0) \Big)_0,
\end{gather*}
where we used that $(a,b_{\half}(b_{\half}x_0) )=\half(b_{\half}^2,x_0)=\half(a,b)(b,x_0),$ and part (2) for $b_0^2$ and $b_{\half}^2.$
Since $w=0,$ (4) follows.
\medskip

\noindent
(5)  Replacing $x=x_{\half}$ in (1) we get,
\begin{gather*}
\textstyle{0=\half(a,b)x_{\half}+(a,b)(b,x_{\half})ab-2(a,b)a(bx_{\half})+b_0x_{\half} +(b,x_{\half})bb_0}\\
 -2(bx_{\half})b_0+b_{\half}x_{\half}+(b,x_{\half})bb_{\half}-2(bx_{\half})b_{\half}\\
 =\textstyle{w+\half(a,b)x_{\half}+\half(a,b)(b,x_{\half})b_{\half}-\half(a,b)^2x_{\half}-(a,b)x_{\half}b_0}\\
 +x_{\half}b_0+(b,x_{\half})b_{\half}b_0 \\
 \textstyle{-(a,b)x_{\half}b_0-2(x_{\half}b_0)b_0+\half(a,b)(b,x_{\half})b_{\half}+(b,x_{\half})b_{\half}b_0 
  -2(b_{\half}x_{\half})b_{\half}} \\
  \\
  =\textstyle{w+\half(a,b)(1-(a,b))x_{\half}+(a,b)(b,x_{\half})b_{\half}+(1-2(a,b))x_{\half}b_0}\\
 +(1-(a,b))(b,x_{\half})b_{\half}-2(x_{\half}b_0)b_0-2(b_{\half}x_{\half})b_{\half}\\
  \\
  \textstyle{=w+\half(a,b)(1-(a,b))x_{\half}+(b,x_{\half})b_{\half}+(1-2(a,b))x_{\half}b_0}\\
 -2(x_{\half}b_0)b_0 
  -2(b_{\half}x_{\half})b_{\half}\\
\textstyle{=w+\half(a,b)(1-(a,b))x_{\half}+\half(b,x_{\half})b_{\half}+(1-2(a,b))x_{\half}b_0}\\
 -2(x_{\half}b_0)b_0 
  -2\Big(b_{\half}x_{\half}\Big)_0b_{\half},
\end{gather*}
with $w\in \ff a+A_0(a),$ where we used $b_{\half}x_{\half}=\half(b,x_{\half})a+\Big(b_{\half}x_{\half}\Big)_0.$ This shows part (b).  Now
\begin{gather*}
w=(a,b)^2(b,x_{\half})a-2(a,b)a(b_{\half}x_{\half})+(b,x_{\half})b_0^2-2(b_{\half}x_{\half})b_0+\\
b_{\half}x_{\half}+(b,x_{\half})b_{\half}^2-(a,b)x_{\half}b_{\half}-2(b_0x_{\half})b_{\half}\\
\textstyle{=(b,x_{\half})((a,b)-1) ((a,b)-\half)a+\half(b_{\half},b_{\half})(b,x_{\half})a-(b_{\half},b_0x_{\half})a}\\
+(b,x_{\half})b_0^2-2(b_{\half}x_{\half})b_0+(1-(a,b))\Big(x_{\half}b_{\half}\Big)_0\\
+(b,x_{\half})\Big(b_{\half}^2\Big)_0 -2\Big((b_0x_{\half})b_{\half}\Big)_0,
\end{gather*}
where the last equality follows from
\begin{align*}
&b_{\half}x_{\half}=\half(b,x_{\half})a+\Big(b_{\half}x_{\half}\Big)_0\text{ so}\\
&(a,b)^2(b,x_{\half})a-2(a,b)a(b_{\half}x_{\half})+(1-(a,b))b_{\half}x_{\half}=\\
&\textstyle{\Big[(a,b)^2(b,x_{\half})-(a,b)(b,x_{\half})+\half(1-(a,b))(b,x_{\half})\Big]a+(1-(a,b))\Big(b_{\half}x_{\half}\Big)_0=}\\
&\textstyle{\Big[(b,x_{\half})(a,b)((a,b)-1))+\half(1-(a,b))(b,x_{\half})\Big]a+(1-(a,b))\Big(b_{\half}x_{\half}\Big)_0=}\\
&\textstyle{(b,x_{\half})((a,b)-1)((a,b)-\half)a+(1-(a,b))\Big(b_{\half}x_{\half}\Big)_0.}
\end{align*}
%Next,
%\begin{align*}
%(b,x_{\half})((a,b)-1) ((a,b)-\half)+\half(b_{\half},b_{\half})%(b,x_{\half})-(b_{\half},b_0x_{\half})\\
%=(b,x_{\half})((a,b)-1) ((a,b)-\half)+(a,b)(1-(a,b))(b,x_{\half})-\half(1-(a,b))(b,x_{\half})\\
%=(1-(a,b))\Big(\half-(a,b)+(a,b)-\half\Big)(b,x_{\half})=0.
%\end{align*}
We have
\begin{gather*}
0=(b,x_{\half})b_0^2-2(b_{\half}x_{\half})b_0+(1-(a,b))\Big(x_{\half}b_{\half}\Big)_0\\
+(b,x_{\half})\Big(b_{\half}^2\Big)_0 -2\Big((b_0x_{\half})b_{\half}\Big)_0\\
=(b,x_{\half})(1-(a,b))b_0-2\Big(b_{\half}x_{\half}\Big)_0b_0+(1-(a,b))\Big(x_{\half}b_{\half}\Big)_0\\
+(b,x_{\half})(a,b)b_0-2\Big((b_0x_{\half})b_{\half}\Big)_0\\
=(b,x_{\half})b_0-2\Big(b_{\half}x_{\half}\Big)_0b_0+(1-(a,b))\Big(x_{\half}b_{\half}\Big)_0-2\Big((b_0x_{\half})b_{\half}\Big)_0,
\end{gather*}
so part (a) holds.
\end{proof}

%%%%%%%%%%%%%%%%%%%%%%%%%%%%%%%%
\begin{lemma}\label{P}
%%%%%%%%%%%%%%%%%%%%%%%%%%%%%
Let $a, b$ be  axes in $A.$  Then
\begin{enumerate}
%1
\item
$a(x(y+(a,y)a-2ay))=(ax)(y+(a,y)a-2ay),$ for all $x,y\in A,$ that is
$a (xy) -2a ((ay) x) = (ax) y-2 (ax) (ay).$
%\[
%-2a((ay)x)+2a((ax)y)=(ax)y-(ay)x
%\]
%\[
%a((ax)y-(ay)x)=\half(((ax)y-(ay)x))
%\]

%2
\item
For all $x,y\in A,$
\begin{gather*}
(a,b)(1-(a,b))a(xy)+b_0(xy)+b_{\half}(xy)-2\Big((a,b)b_0((ay)x)+(a,b)b_{\half}((ay)x)\Big)\\
-2\Big((a,b)a((b_0y)x)+b_0((b_0y)x)+b_{\half}((b_0y)x)\\
+(a,b)a((b_{\half}y)x)+b_0((b_{\half}y)x)+b_{\half}((b_{\half}y)x)\Big)\\=
\\
(a,b)(1-(a,b))(ax)y+(b_0x)y+(b_{\half}x)y-2\Big((a,b)(ax)(b_0y)+(a,b)(ax)(b_{\half}y) \Big)\\
-2\Big((a,b)(b_0x)(ay)+(b_0x)(b_0y)+ (b_0x)(b_{\half}y)\Big)\\
-2\Big((a,b) (b_{\half}x)(ay)+ (b_{\half}x)(b_0y)+ (b_{\half}x)(b_{\half}y) \Big).
\end{gather*}
 
%3
\item
Taking $x=x_0\in A_0(a)$ and $y=a$ in (2) yields
%\begin{enumerate}
%\item
\[
(1-(a,b))b_{\half}x_0-2b_{\half}(x_0b_0)+2(b_{\half}x_0)b_0 =0.
\]

%4
\item
Taking $x=a$ and $y=y_{\half}$ in (2), we get
\begin{enumerate}
%a
\item
$\half(a,b)(1-(a,b))y_{\half}-2b_{\half}\Big(y_{\half}b_{\half}\Big)_0 
+\half(b, y_{\half})b_{\half}+2(y_{\half}b_0)b_0-y_{\half}b_0=0.$
 
%b
\item
$(x_{\half}b_0)b_0 =\half(1-(a,b))x_{\half}b_0.$

%c
\item
$b_{\half}\Big(b_{\half}x_{\half}\Big)_0 =\frac 14(a,b)(1-(a,b))x_{\half} -\half(a,b)x_{\half}b_0+\frac 14(b,x_{\half})b_{\half}.$

%d
\item 
$b_{\half}(b_{\half}x_{\half})= \frac 14(a,b)(1-(a,b))x_{\half}-\half(a,b)x_{\half}b_0+\half(b,x_{\half})b_{\half}.$
 \end{enumerate}

%5
 \item
 Taking $y=a$ and $x=x_{\half}$ in (2) we get
 \begin{enumerate}

 %a
 \item
$(1-(a,b))\Big(b_{\half}x_{\half}\Big)_0-2(b_{\half}x_{\half} )b_0+2\Big(b_{\half}(x_{\half}b_0)\Big)_0 =0.$

%b
\item
$\Big(b_{\half}(x_{\half}b_0) \Big)_0=\frac 14(b,x_{\half})b_0.$

%c
\item 
$b_{\half}(x_{\half}b_0)=\frac 14(1-(a,b))(b,x_{\half})a+\frac 14(b,x_{\half})b_0.$

%d
\item
From (a) and (b) it follows:
\[\textstyle{(b_{\half}x_{\half})b_0=\half(1-(a,b))\Big(b_{\half}x_{\half}\Big)_0+\frac 14 (b,x_{\half})b_0.}\]
\end{enumerate}

%6
\item
Taking $x=x_0,$ and $y=y_0$ in (2) we get
\begin{gather*}
b_0(x_0y_0)+b_{\half}(x_0y_0)
-2\Big[b_0((b_0y_0)x_0)+b_{\half}((b_0y_0)x_0) \\
\textstyle{+\half(a,b)(b_{\half}y_0)x_0+b_0((b_{\half}y_0)x_0)+b_{\half}((b_{\half}y_0)x_0)\Big]}\\
=(b_0x_0)y_0+(b_{\half}x_0)y_0\\
-2\Big[(b_0x_0)(b_0y_0)+ (b_0x_0)(b_{\half}y_0)\Big]\\
-2\Big[(b_{\half}x_0)(b_0y_0)+ (b_{\half}x_0)(b_{\half}y_0) \Big].
\end{gather*}

%7
\item
For all $x_0, y_0\in A_0(a),$
\begin{gather*}
b_{\half}(x_0y_0)+2 (b_{\half}y_0)(x_0b_0)
+2(b_{\half}x_0)(y_0b_0)\\
=(b_{\half}x_0)y_0+2b_{\half}(x_0(y_0b_0) ) 
+(a,b)(b_{\half}y_0)x_0+2((b_{\half}y_0)x_0)b_0.
\end{gather*}

%8
\item
Interchanging $x_0$ and $y_0$ in (7) we get
\begin{enumerate}

%a
\item 
By subtracting we get
\begin{gather*}
(1-(a,b))(b_{\half}x_0)y_0-(1-(a,b))(b_{\half}y_0)x_0  \\
=2b_{\half}(y_0(x_0b_0) )-2b_{\half}(x_0(y_0b_0) )+2((b_{\half}x_0)y_0)b_0-2((b_{\half}y_0)x_0)b_0.
\end{gather*}

%b
\item 
By adding we get
\begin{gather*}
2b_{\half}(x_0y_0)+4 (b_{\half}y_0)(x_0b_0)
+4(b_{\half}x_0)(y_0b_0)\\
=(1+(a,b))(b_{\half}x_0)y_0+(1+(a,b))(b_{\half}y_0)x_0+2b_{\half}((y_0b_0)x_0)+2b_{\half}((x_0b_0)y_0)\\
 +2((b_{\half}y_0)x_0)b_0+2((b_{\half}x_0)y_0)b_0.
\end{gather*}
\end{enumerate}
 
%9
\item 
For all $x_0, y_0\in A_0(a),$

\begin{gather*}
(x_0y_0)b_0+2(x_0b_0)(y_0b_0)+2\Big((b_{\half}x_0)(b_{\half}y_0)\Big)_0\\
=y_0(x_0b_0)+2(x_0(y_0b_0))b_0+2\Big(b_{\half}((b_{\half}y_0)x_0)\Big)_0.
\end{gather*}
\item
Interchanging $x_0$ and $y_0$ in (9), we get

\begin{enumerate}
\item\label{14a}
by subtracting, 
\begin{gather*}
2\Big(b_{\half}((b_{\half}x_0)y_0)\Big)_0-
2\Big(b_{\half}((b_{\half}y_0)x_0)\Big)_0\\
=y_0(x_0b_0)+2(x_0(y_0b_0))b_0-x_0(y_0b_0) -2(y_0(x_0b_0))b_0.
\end{gather*}
%that is, for $w_0:=x_0(y_0b_0)-y_0(x_0b_0),$
%\[
%2\Big(b_{\half}z_{\half}\Big)_0=2w_0b_0-w_0
%\]
%\begin{gather*}
%2\Big(b_{\half}((b_{\half}x_0)y_0)\Big)_0-
%2\Big(b_{\half}((b_{\half}y_0)x_0)\Big)_0
%
%=2w_0b_0-w_0=-(a,b)w_0
%\end{gather*}
%\begin{gather*}
%2\Big(b_{\half}((b_{\half}x_0)y_0)\Big)_0+
%2\Big(b_{\half}((b_{\half}y_0)x_0)\Big)_0
%
%=(a,b)(x_0y_0)b_0
%\end{gather*}
%\begin{gather*}
%4\Big(b_{\half}((b_{\half}x_0)y_0)\Big)_0=(a,b)\Big[(x_0y_0)b_0-w_0\Big]
%\end{gather*}

\item\label{14b} 
By adding,
\begin{gather*}
 4(x_0b_0)(y_0b_0)+4\Big((b_{\half}x_0)(b_{\half}y_0)\Big)_0-y_0(x_0b_0)-x_0(y_0b_0)+(x_0y_0)b_0\\
= 2(x_0(y_0b_0))b_0+2(y_0(x_0b_0))b_0\\
+2\Big(b_{\half}((b_{\half}y_0)x_0)\Big)_0+2\Big(b_{\half}((b_{\half}x_0)y_0)\Big)_0-(x_0y_0)b_0. 
\end{gather*}
\end{enumerate}

\end{enumerate}
\end{lemma}
\begin{proof}
(1)  Follows from Seress' lemma \ref{seress}, and Lemma \ref{x0x12}(1).
\medskip

\noindent
(2)  By (1),
\begin{gather*}
b(xy)-2b((by)x)=(bx)y-2(bx)(by)\quad\iff\\
((a,b)a+b_0+b_{\half})(xy)-2((a,b)a+b_0+b_{\half})((((a,b)a+b_0+b_{\half})y)x)\\
=(((a,b)a+b_0+b_{\half})x)y-2(((a,b)a+b_0+b_{\half})x)
(((a,b)a+b_0+b_{\half})y)\\\iff
\\
(a,b)a(xy)+b_0(xy)+b_{\half}(xy)-2\Big[(a,b)^2a((ay)x))+(a,b)b_0((ay)x)+(a,b)b_{\half}((ay)x)\Big]\\
-2\Big[(a,b)a((b_0y)x)+b_0((b_0y)x)+b_{\half}((b_0y)x)+(a,b)a((b_{\half}y)x)+b_0((b_{\half}y)x)+b_{\half}((b_{\half}y)x)\Big]\\
=(a,b)(ax)y+(b_0x)y+(b_{\half}x)y-2\Big[(a,b)^2(ax)(ay)+(a,b)(ax)(b_0y)+(a,b)(ax)(b_{\half}y) \Big]\\
-2\Big[(a,b)(b_0x)(ay)+(b_0x)(b_0y)+ (b_0x)(b_{\half}y) \Big]\\
-2\Big[(a,b) (b_{\half}x)(ay)+ (b_{\half}x)(b_0y)+ (b_{\half}x)(b_{\half}y) \Big]\\\iff
\\
(a,b)(1-(a,b))a(xy)+b_0(xy)+b_{\half}(xy)-2\Big[(a,b)b_0((ay)x)+(a,b)b_{\half}((ay)x)\Big]\\
-2\Big[(a,b)a((b_0y)x)+b_0((b_0y)x)+b_{\half}((b_0y)x)+(a,b)a((b_{\half}y)x)+b_0((b_{\half}y)x)+b_{\half}((b_{\half}y)x)\Big]\\
=(a,b)(1-(a,b))(ax)y+(b_0x)y+(b_{\half}x)y-2\Big[(a,b)(ax)(b_0y)+(a,b)(ax)(b_{\half}y) \Big]\\
-2\Big[(a,b)(b_0x)(ay)+(b_0x)(b_0y)+ (b_0x)(b_{\half}y)\Big]\\
-2\Big[(a,b) (b_{\half}x)(ay)+ (b_{\half}x)(b_0y)+ (b_{\half}x)(b_{\half}y) \Big]
\end{gather*}
where we used
$(a,b)^2\Big[a(xy)-2a((ay)x)\Big]=(a,b)^2\Big[(ax)y-2(ax)(ay)\Big].$

\noindent
(3)
Taking $y=a$ and $x=x_0$ in (2) we get,
\begin{gather*}
\textstyle{-\half(a,b)(b_{\half}x_0)- b_0(b_{\half}x_0)- b_{\half}(b_{\half}x_0)}\\
\textstyle{=\half b_{\half}x_0
-(b_0x_0)b_{\half}
-(a,b) (b_{\half}x_0)-(b_{\half}x_0)b_{\half}\iff}\\
\textstyle{\half((a,b)-1)b_{\half}x_0- b_0(b_{\half}x_0)=-(b_0x_0)b_{\half}.}
\end{gather*}
This shows (3).%  Now (b) follows from (a) and from Lemma \ref{Q}(4)(b), and (c) follows from (a) and (b).
%\medskip
%
%\noindent
%(5) %Indeed, replacing $x=a$ and $y=y_0$ in (2) we get
%\begin{gather*}
%-2\Big((a,b)a((b_{\half}y_0)a)+b_0((b_{\half}y_0)a)+b_{\half}((b_{\half}y_0)a)\Big)\\
%=(b_{\half}a)y_0-2\Big((a,b)a(b_{\half}y_0) \Big)\\
%-2\Big((b_{\half}a)(b_0y_0)+ (b_{\half}a)(b_{\half}y_0) \Big)\quad\quad\iff\\
%-\half (a,b)b_{\half}y_0-(b_{\half}y_0)b_0-b_{\half}(b_{\half}y_0)\\
%=\half b_{\half}y_0-(a,b)b_{\half}y_0-b_{\half}(b_0y_0)-b_{\half}(b_{\half}y_0),
%\end{gather*}
%which is the same as (4)(a).
\medskip

\noindent
(4) (a):\ Replacing $x=a$ and $y=y_{\half}$ in (2) we get
\begin{gather*}
\textstyle{\frac 14 (a,b)(1-(a,b))y_{\half}+\half y_{\half}b_0+\half y_{\half}b_{\half}-2\Big[\frac 14 (a,b)y_{\half}b_0+\frac 14 (a,b)y_{\half}b_{\half}\Big]}\\
\textstyle{-2\Big[\frac 14 (a,b)y_{\half}b_0+\half (y_{\half}b_0)b_0+\half(y_{\half}b_0)b_{\half}+(a,b)(a,b_{\half}y_{\half})a+\half (a, b_{\half}y_{\half})b_{\half}\Big]}\\
\textstyle{=\half(a,b)(1-(a,b))y_{\half}+\half b_{\half}y_{\half}-2\Big[\half(a,b)y_{\half}b_0+(a,b)(a,b_{\half}y_{\half})a \Big]}\\
\textstyle{-2\Big[\frac 14(a,b)y_{\half}b_{\half}+\half(y_{\half}b_0) b_{\half} + \half b_{\half}(b_{\half}y_{\half}) \Big].}
\end{gather*}
Hence, comparing the $\half$-components, we get
\begin{gather*}
\textstyle{\frac 14 (a,b)(1-(a,b))y_{\half}+\half y_{\half}b_0-\half (a,b)y_{\half}b_0}\\
\textstyle{-\half (a,b)y_{\half}b_0- (y_{\half}b_0)b_0-(a,y_{\half}b_{\half})b_{\half}}\\
\textstyle{=\half(a,b)(1-(a,b))y_{\half}-(a,b)y_{\half}b_0}\\
\textstyle{-\half(a, b_{\half}y_{\half})b_{\half}-\Big(y_{\half}b_{\half}\Big)_0b_{\half}.}
\end{gather*}
Hence
\begin{gather*}
\textstyle{\half y_{\half}b_0- (y_{\half}b_0)b_0
=\frac 14(a,b)(1-(a,b))y_{\half}
+\frac 14(b, y_{\half})b_{\half}-b_{\half}\Big(b_{\half}y_{\half} \Big)_0,}
\end{gather*}
or
\[
\textstyle{\half(a,b)(1-(a,b))y_{\half}-2b_{\half}\Big(b_{\half}y_{\half}\Big)_0 
+\half(b, y_{\half})b_{\half}+2(y_{\half}b_0)b_0-y_{\half}b_0=0.}
\]
%The $1$-components and the $0$-components are the same. This shows part (a).  
\medskip

\noindent
(b):\ Comparing (a) with Lemma \ref{Q}(5)(b) which states:
\begin{gather*}
\textstyle{\half(a,b)(1-(a,b))x_{\half}-2b_{\half}\Big(b_{\half}x_{\half}\Big)_0  +\half(b,x_{\half})b_{\half}}\\
-2(x_{\half}b_0)b_0+(1-2(a,b))x_{\half}b_0=0,
\end{gather*}
 we get
 \[
 2(x_{\half}b_0)b_0-x_{\half}b_0=-2(x_{\half}b_0)b_0+(1-2(a,b))x_{\half}b_0,
 \]
 which proves (b).
 \medskip

 \noindent
 (c)
 By lemma \ref{Q}(5)(b), and by (b),
\begin{gather*}
\textstyle{\half(a,b)(1-(a,b))x_{\half}-2b_{\half}\Big(b_{\half}x_{\half}\Big)_0  +\half(b,x_{\half})b_{\half}}\\
-2(x_{\half}b_0)b_0+(1-2(a,b))x_{\half}b_0=0\\\iff
 \\
 \textstyle{\half(a,b)(1-(a,b))x_{\half}-2\Big(b_{\half}x_{\half}\Big)_0b_{\half}+\half(b,x_{\half})b_{\half}-(a,b)x_{\half}b_0=0.} 
 \end{gather*}
 \medskip

 \noindent
(d)  Note that $(a,b_{\half}x_{\half})=\half(b,x_{\half}).$  Hence, by (c),
\begin{gather*}
b_{\half}(b_{\half}x_{\half})=b_{\half}\Big[(a,b_{\half}x_{\half})a+\Big(b_{\half}x_{\half}\Big)_0\Big]\\
\textstyle{=\frac 14(b,x_{\half})b_{\half}+\frac 14(a,b)(1-(a,b))x_{\half}+\frac 14(b,x_{\half})b_{\half}-\half(a,b)x_{\half}b_0.}
\end{gather*}
%\medskip
%
%\noindent
% (e) %By (d) we have,   
% \begin{gather*}
%(b,(b_{\half}y_{\half})x_{\half}))=(b_{\half},(b_{\half}y_{\half})x_{\half}))=(b_{\half}(b_{\half}x_{\half} ) ,y_{\half})\\
%=(\half(b,x_{\half})b_{\half}+\frac 14(a,b)(1-(a,b))x_{\half}-%\half(a,b)x_{\half}b_0,y_{\half})\\
%=\half(b,x_{\half})(b,y_{\half})+\frac 14(a,b)(1-(a,b))(x_{\half},y_{\half})-\half(a,b)(b,\Big(x_{\half}y_{\half}\Big)_0)
%\end{gather*}
\medskip

\noindent
(5)  Taking $y=a$ and $x=x_{\half}$ in (2) we get
\begin{gather*}
\textstyle{\frac 14 (a,b)(1-(a,b))x_{\half}+\half x_{\half}b_0+\half b_{\half}x_{\half}-(a,b)x_{\half}b_0 -(a,b)b_{\half}x_{\half}}\\
\textstyle{-(a,b)a(b_{\half}x_{\half})- b_0(b_{\half}x_{\half})- b_{\half}(b_{\half}x_{\half})\quad =}\\
\textstyle{\frac 14 (a,b)(1-(a,b))x_{\half}+\half x_{\half}b_0+a(b_{\half}x_{\half})-\half (a,b)x_{\half}b_{\half}}\\
-(a,b)x_{\half}b_0-(x_{\half}b_0)b_{\half}
-2(a,b) a(b_{\half}x_{\half})- (b_{\half}x_{\half})b_{\half}\\\iff
\\
\textstyle{\half b_{\half}x_{\half}-\half(a,b)b_{\half}x_{\half}
-(a,b)a(b_{\half}x_{\half})- b_0(b_{\half}x_{\half}) }\\
\textstyle{=a(b_{\half}x_{\half})
-(b_0x_{\half})b_{\half}
-2(a,b) a(b_{\half}x_{\half})}\quad\iff
\\
\textstyle{\half(1-(a,b)) b_{\half}x_{\half}+((a,b)-1)a(b_{\half}x_{\half})+(b_0x_{\half})b_{\half}-b_0(b_{\half}x_{\half})=0.}
\end{gather*}
Writing $b_{\half}x_{\half}=\half (b,x_{\half})a+\Big(b_{\half}x_{\half}\Big)_0,$ and $b_{\half}(b_0x_{\half})=\half (b,b_0x_{\half})a+\Big(b_{\half}(b_0x_{\half})\Big)_0,$ we get $((a,b)-1)a(b_{\half}x_{\half})=\half((a,b)-1)(b,x_{\half})a,$ so
\begin{gather*}
\textstyle{\frac 14((a,b)-1)(b,x_{\half})a+\half(b,b_0x_{\half})a}\\
\textstyle{+\half(1-(a,b))\Big(b_{\half}x_{\half}\Big)_0+\Big((b_0x_{\half})b_{\half}\Big)_0-b_0(b_{\half}x_{\half})=0.}
\end{gather*}
Now $\half(b,b_0x_{\half})=\half(b_{\half}b_0,x_{\half})=\frac 14(1-(a,b))(b,x_{\half}),$ so part (a) holds.  Then, (b) follows from (a) and Lemma \ref{Q}(5)(a).
\medskip

\noindent
(c) We only need to compute 
$(a,b_{\half}(x_{\half}b_0))=\half(b_{\half}b_0,x_{\half})=\frac 14(1-(a,b))(b,x_{\half}).$
%\medskip
%
%\noindent
%(d)% By ((3))(c),
%\[
%2b_{\half}(x_0b_0) =(1-(a,b))b_{\half}x_0+\half(b,x_0)b_{\half},
%\]
%Thus multiplying (a) by $b_{\half},$ and using (b) we get
%\begin{gather*}
%\half(1-(a,b))\Big(b_{\half}x_{\half}\Big)_0b_{\half}+\frac 14 (b,x_{\half})b_{\half}b_0-b_{\half}\Big(\Big(x_{\half}b_{\half}\Big)_0b_0\Big)=0\\\iff
%\\
%\half(1-(a,b))\Big(b_{\half}x_{\half}\Big)_0b_{\half}+\frac 14 (b,x_{\half})b_{\half}b_0\\
%-\half(1-(a,b))b_{\half}\Big(b_{\half}x_{\half}\Big)_0-\frac 14(b,\Big(b_{\half}x_{\half}\Big)_0)b_{\half}=0
%-b_{\half}\Big(\Big(x_{\half}b_{\half}\Big)_0b_0\Big)=0
%\\\iff
%\\
%(b,\Big(b_{\half}x_{\half}\Big)_0)b_{\half}=\half(1-(a,b))(b,x_{\half})b_{\half}\\
%\end{gather*}
%\medskip
%
%\noindent
%(e) % Set $z_0:=\Big(x_{\half}b_{\half}\Big)_0,$ then by (a)
%\begin{gather*}
%0=\Big(b,\half(1-(a,b))\Big(b_{\half}x_{\half}\Big)_0+\Big((b_0x_{\half})b_{\half}\Big)_0-(x_{\half}b_{\half})b_0\Big)\\
%=\half(1-(a,b))(b,z_0)+\frac 14(b,x_{\half})(b,b_0)-(b_0^2,z_0)\\
%=\half(1-(a,b))(b,z_0)+\frac 14(1-(a,b))^2(b,x_{\half})-(1-(a,b)(b,z_0)\\\iff
%\\
%\half(1-(a,b))(b,z_0)=\frac 14(1-(a,b))^2(b,x_{\half}).
%\end{gather*}
%
%\medskip
%
%\noindent
(7) Comparing the $\half$-components in (6) we get
\begin{gather*}
\textstyle{b_{\half}(x_0y_0)
-2\Big[b_{\half}((y_0b_0)x_0) 
+\half(a,b)(b_{\half}y_0)x_0+((b_{\half}y_0)x_0)b_0\Big]}\\
=(b_{\half}x_0)y_0
-2(b_{\half}y_0)(x_0b_0)
-2(b_{\half}x_0)(y_0b_0 )\\\iff
\\
b_{\half}(x_0y_0)+2 (b_{\half}y_0)(x_0b_0)
+2(b_{\half}x_0)(y_0b_0 )\\
=(b_{\half}x_0)y_0+2b_{\half}((y_0b_0)x_0) 
+(a,b)(b_{\half}y_0)x_0+2((b_{\half}y_0)x_0)b_0.
\end{gather*}
\medskip

\noindent
(8) Interchanging $x_0$ and $y_0$ in (7) we get, by subtracting,
\begin{gather*}
(b_{\half}x_0)y_0+2b_{\half}((y_0b_0)x_0) 
+(a,b)(b_{\half}y_0)x_0+2((b_{\half}y_0)x_0)b_0\\
=(b_{\half}y_0)x_0+2b_{\half}((x_0b_0)y_0) 
+(a,b)(b_{\half}x_0)y_0+2((b_{\half}x_0)y_0)b_0\\\iff
\\
(1-(a,b))(b_{\half}x_0)y_0-(1-(a,b))(b_{\half}y_0)x_0  \\
=2b_{\half}((x_0b_0)y_0)-2b_{\half}((y_0b_0)x_0)+2((b_{\half}x_0)y_0)b_0-2((b_{\half}y_0)x_0)b_0.
\end{gather*}
The other part of (8) is clear.
\medskip

\noindent
(9) Comparing the $0$-components in (6) we get,
\begin{gather*}
b_0(x_0y_0)-2b_0((b_0y_0)x_0) 
-2\Big(b_{\half}((b_{\half}y_0)x_0)\Big)_0\\
=(b_0x_0)y_0-2(b_0x_0)(b_0y_0)-2\Big((b_{\half}x_0)(b_{\half}y_0)\Big)_0\\\iff
\\
b_0(x_0y_0)+2(b_0x_0)(b_0y_0)+2\Big((b_{\half}x_0)(b_{\half}y_0)\Big)_0\\
=(b_0x_0)y_0+2b_0((b_0y_0)x_0)+2\Big(b_{\half}((b_{\half}y_0)x_0)\Big)_0.
\end{gather*}
\medskip

\noindent
(10)  This is immediate from (9).
\end{proof}

%%%%%%%%%%%%%%%%%%%%%%%%%%%%%%%%%
\begin{lemma}\label{120}
%%%%%%%%%%%%%%%%%%%%%%%%%%%%%%
Let $x_{\half}\in A_{\half}(a)$ and $y_0\in A_0(a),$ then
\begin{enumerate}

\item 
\begin{gather*} 
(x_{\half}y_0)b_0+b_{\half}(x_{\half}y_0)\\
-2\Big[(x_{\half}(y_0b_0))b_0+b_{\half}(x_{\half}(y_0b_0) )+(a,b)a((b_{\half}y_0)x_{\half})+ (x_{\half}(b_{\half}y_0))b_0\\
+b_{\half}(x_{\half}(b_{\half}y_0))\Big]\quad =\\
(x_{\half}b_0)y_0+(b_{\half}x_{\half})y_0-(a,b)x_{\half}(b_{\half}y_0)\\
-2\Big[(x_{\half}b_0)(y_0b_0)+ (x_{\half}b_0)(b_{\half}y_0)\Big]\\
-2\Big[(b_{\half}x_{\half})(y_0b_0)+ (b_{\half}x_{\half})(b_{\half}y_0) \Big].
\end{gather*}

%2
\item 
\begin{gather*} 
(x_{\half}y_0)b_0-2(x_{\half}(y_0b_0))b_0-2b_{\half}(x_{\half}(b_{\half}y_0) )\quad =\\
(x_{\half}b_0)y_0-2(x_{\half}b_0)(y_0b_0)
-2(b_{\half}x_{\half})(b_{\half}y_0).
\end{gather*}
%
%\begin{gather*}
%(1-(a,b))(x_{\half}y_0)b_0-2((x_{\half}b_0)y_0)b_0\\
%
%
%=(1-(a,b)) x_{\half}(y_0b_0)
%-2(x_{\half}b_0)( y_0b_0)
%\end{gather*}

%3
\item
\begin{gather*} 
b_{\half}(x_{\half}y_0)+(a,b)(b_{\half}y_0)x_{\half} 
-(b_{\half}x_{\half})y_0 
+2(x_{\half}b_0)(b_{\half}y_0)\\
+2(b_{\half}x_{\half})(y_0b_0)-2(b_{\half}(x_{\half}(y_0b_0) )-2(a,b)a(x_{\half}(b_{\half}y_0))-2((b_{\half}y_0)x_{\half} )b_0\\
= 0.
\end{gather*}
\end{enumerate}
\end{lemma}
\begin{proof}
(1) taking $x=x_{\half}$ and $y=y_0$ in Lemma \ref{P}(2), we get
\begin{gather*} 
(a,b)(1-(a,b))a(x_{\half}y_0)+b_0(x_{\half}y_0)+b_{\half}(x_{\half}y_0)-2\Big[(a,b)b_0((ay_0)x_{\half})+(a,b)b_{\half}((ay_0)x_{\half})\Big]\\
-2\Big[(a,b)a((b_0y_0)x_{\half}) +b_0((b_0y_0)x_{\half})+b_{\half}((b_0y_0)x_{\half})+(a,b)a((b_{\half}y_0)x_{\half})+b_0((b_{\half}y_0)x_{\half})\\
+b_{\half}((b_{\half}y_0)x_{\half})\Big]\\=
\\
(a,b)(1-(a,b))(ax_{\half})y_0+(b_0x_{\half})y_0+(b_{\half}x_{\half})y_0-2\Big[(a,b)(ax_{\half})(b_0y_0)+(a,b)(ax_{\half})(b_{\half}y_0) \Big]\\
-2\Big[(a,b)(b_0x_{\half})(ay_0)+(b_0x_{\half})(b_0y_0)+ (b_0x_{\half})(b_{\half}y_0)\Big]\\
-2\Big[(a,b) (b_{\half}x_{\half})(ay_0)+ (b_{\half}x_{\half})(b_0y_0)+ (b_{\half}x_{\half})(b_{\half}y_0) \Big].
\end{gather*}
So,
\begin{gather*} 
\textstyle{\half(a,b)(1-(a,b))x_{\half}y_0+(x_{\half}y_0)b_0+b_{\half}(x_{\half}y_0)}\\
\textstyle{-2\Big[\half(a,b)x_{\half}(y_0b_0) +(x_{\half}(y_0b_0))b_0+b_{\half}(x_{\half}(y_0b_0) )+(a,b)a((b_{\half}y_0)x_{\half})+ ((b_{\half}y_0)x_{\half})b_0}\\
+b_{\half}((b_{\half}y_0)x_{\half})\Big]\\=
\\
\textstyle{\half(a,b)(1-(a,b))x_{\half}y_0+(x_{\half}b_0)y_0+(b_{\half}x_{\half})y_0-2\Big[\half(a,b)x_{\half}(y_0b_0 )+\half(a,b)x_{\half}(b_{\half}y_0) \Big]}\\
-2\Big[(x_{\half}b_0)(y_0b_0)+ (x_{\half}b_0)(b_{\half}y_0)\Big)\\
-2\Big((b_{\half}x_{\half})(y_0b_0)+ (b_{\half}x_{\half})(b_{\half}y_0) \Big].
\end{gather*}
Hence,
\begin{gather*} 
(x_{\half}y_0)b_0+b_{\half}(x_{\half}y_0))\\
-2\Big[(x_{\half}(y_0b_0))b_0+b_{\half}(x_{\half}(b_0y_0) )+(a,b)a((b_{\half}y_0)x_{\half})+ ((b_{\half}y_0)x_{\half})b_0\\
+b_{\half}((b_{\half}y_0)x_{\half})\Big]\\
=\textstyle{(x_{\half}b_0)y_0+(b_{\half}x_{\half})y_0-2\Big(\half(a,b)x_{\half}(b_{\half}y_0) \Big)}\\
-2\Big[(x_{\half}b_0)(y_0b_0)+ (b_0x_{\half})(b_{\half}y_0)\Big]\\
-2\Big[(b_{\half}x_{\half})(b_0y_0)+ (b_{\half}x_{\half})(b_{\half}y_0) \Big].
\end{gather*}
\medskip

\noindent
(2) Comparing the $\half$-components in (1) we get
\begin{gather*} 
(x_{\half}y_0)b_0-2(x_{\half}(y_0b_0))b_0-2b_{\half}((b_{\half}y_0)x_{\half})\\
=(x_{\half}b_0)y_0-2(x_{\half}b_0)(y_0b_0)-2(b_{\half}x_{\half})(b_{\half}y_0).
\end{gather*}
This shows (2).
\medskip

\noindent
(3)  Comparing the $1$-component plus the $0$-components on both sides of (1) gives:
\begin{gather*} 
b_{\half}(x_{\half}y_0)
-2b_{\half}(x_{\half}(y_0b_0))-2(a,b)a(x_{\half}(b_{\half}y_0))-2(x_{\half}(b_{\half}y_0))b_0\\
=(b_{\half}x_{\half})y_0-(a,b)x_{\half}(b_{\half}y_0)
-2(x_{\half}b_0)(b_{\half}y_0)
-2(b_{\half}x_{\half})(y_0b_0).\qedhere
\end{gather*}
\end{proof}

%%%%%%%%%%%%%%%%%%%%%%%%%%%%%%%%%%
\begin{lemma}\label{012}
%%%%%%%%%%%%%%%%%%%%%%%%%%%%%%%%%
Let $x_0\in A_0(a)$ and $y_{\half}\in A_{\half}(a), then$
\begin{enumerate}
\item 
\begin{gather*}
\textstyle{\half(a,b)(1-(a,b))y_{\half}x_0+(1-(a,b))(y_{\half}x_0)b_0+(1-(a,b))b_{\half}(y_{\half}x_0)}\\
-(a,b)(y_{\half}b_0)x_0-2((y_{\half}b_0)x_0)b_0-2b_{\half}((y_{\half}b_0)x_0)\\
-2((b_{\half}y_{\half})x_0)b_0-2b_{\half}((b_{\half}y_{\half})x_0)\\
=(1-(a,b)) y_{\half}(x_0b_0)+(1-(a,b))y_{\half}(b_{\half}x_0) \\
-2(y_{\half}b_0)(b_0x_0)-2(b_{\half}y_{\half})(b_0x_0)\\
-2(b_{\half}x_0)(y_{\half}b_0)-2(b_{\half}x_0)(b_{\half}y_{\half}).
\end{gather*}

\item 
\begin{gather*}
\textstyle{\half(a,b)(1-(a,b))y_{\half}x_0+(1-(a,b))(y_{\half}x_0)b_0
-(a,b)(y_{\half}b_0)x_0-2((y_{\half}b_0)x_0)b_0}\\
-2b_{\half}((b_{\half}y_{\half})x_0)\\
=(1-(a,b)) y_{\half}(x_0b_0)
-2(y_{\half}b_0)(b_0x_0)
-2(b_{\half}x_0)(b_{\half}y_{\half}).
\end{gather*}

\item 
\begin{gather*}
(1-(a,b))b_{\half}(y_{\half}x_0)
-(1-(a,b))(b_{\half}x_0)y_{\half}+2(b_{\half}x_0)(y_{\half}b_0)\\
 +2(b_{\half}y_{\half})(x_0b_0 )-2((b_{\half}y_{\half})x_0)b_0-2b_{\half}((y_{\half}b_0)x_0)
  \\
= 0 .
\end{gather*}
\end{enumerate}
\end{lemma}
\begin{proof}

(1)
Replacing $x=x_0$ and $y=y_{\half}$ in Lemma \ref{P}(2) we get
\begin{gather*}
(a,b)(1-(a,b))a(x_0y_{\half})+b_0(x_0y_{\half})+b_{\half}(x_0y_{\half})-2\Big((a,b)b_0((ay_{\half})x_0)+(a,b)b_{\half}((ay_{\half})x_0)\Big)\\
-2\Big((a,b)a((b_0y_{\half})x_0)+b_0((b_0y_{\half})x_0)+b_{\half}((b_0y_{\half})x_0)+(a,b)a((b_{\half}y_{\half})x_0)\\
+b_0((b_{\half}y_{\half})x_0)+b_{\half}((b_{\half}y_{\half})x_0)\Big)\\=
\\
(a,b)(1-(a,b))(ax_0)y_{\half}+(b_0x_0)y_{\half}+(b_{\half}x_0)y_{\half}\\
-2\Big((a,b)(ax_0)(b_0y_{\half})+(a,b)(ax_0)(b_{\half}y_{\half}) \Big)\\
-2\Big((a,b)(b_0x_0)(ay_{\half})+(b_0x_0)(b_0y_{\half})+ (b_0x_0)(b_{\half}y_{\half})\Big)\\
-2\Big((a,b) (b_{\half}x_0)(ay_{\half})+ (b_{\half}x_0)(b_0y_{\half})+ (b_{\half}x_0)(b_{\half}y_{\half}) \Big)
\end{gather*}
\begin{gather*}
\iff\\
\textstyle{\half(a,b)(1-(a,b))y_{\half}x_0+(y_{\half}x_0)b_0+b_{\half}(y_{\half}x_0)-2\Big(\half(a,b)(y_{\half}x_0)b_0 +\half(a,b)b_{\half}(y_{\half}x_0)\Big)}\\
\textstyle{-2\Big(\half(a,b)(y_{\half}b_0)x_0+((y_{\half}b_0)x_0)b_0+b_{\half}((y_{\half}b_0)x_0)}\\
+((b_{\half}y_{\half})x_0)b_0+b_{\half}((b_{\half}y_{\half})x_0)\Big)\\=
\\
y_{\half}(x_0b_0 )+y_{\half}(b_{\half}x_0) \\
\textstyle{-2\Big(\half(a,b)y_{\half}(x_0b_0 ) +(y_{\half}b_0)(x_0b_0 ) +  (b_{\half}y_{\half})(x_0b_0 )\Big)}\\
\textstyle{-2\Big(\half(a,b)y_{\half} (b_{\half}x_0) + (b_{\half}x_0)(y_{\half}b_0)+ (b_{\half}x_0)(b_{\half}y_{\half}) \Big)}\\\iff
\\
\textstyle{\half(a,b)(1-(a,b))y_{\half}x_0+(1-(a,b))(y_{\half}x_0)b_0+(1-(a,b))b_{\half}(y_{\half}x_0)}\\
\textstyle{-2\Big(\half(a,b)(y_{\half}b_0)x_0+((y_{\half}b_0)x_0)b_0+b_{\half}((y_{\half}b_0)x_0)}\\
+((b_{\half}y_{\half})x_0)b_0+b_{\half}((b_{\half}y_{\half})x_0)\Big)\\=
\\
(1-(a,b)) y_{\half}(x_0b_0)+(1-(a,b))y_{\half}(b_{\half}x_0) \\
-2\Big((y_{\half}b_0)(x_0b_0 ) +  (b_{\half}y_{\half})(x_0b_0 )\Big)\\
-2\Big((b_{\half}x_0)(y_{\half}b_0)+ (b_{\half}x_0)(b_{\half}y_{\half}) \Big).
\end{gather*}
\medskip

\noindent
(2)
Comparing the $\half$-components  we get,
\begin{gather*}
\textstyle{\half(a,b)(1-(a,b))y_{\half}x_0+(1-(a,b))(y_{\half}x_0)b_0}\\
-(a,b)(y_{\half}b_0)x_0-2((y_{\half}b_0)x_0)b_0
-2b_{\half}((b_{\half}y_{\half})x_0)\quad =\\
(1-(a,b)) y_{\half}(x_0b_0)  
-2(y_{\half}b_0)(b_0x_0)
-2(b_{\half}x_0)(b_{\half}y_{\half}).
\end{gather*}
\medskip

\noindent
(3)  This follows from comparing the $1$ plus $0$ components of both sides of (2).
\end{proof}

%%%%%%%%%%%%%%%%%%%%%%%%%%%%%
\begin{lemma}\label{1212}
%%%%%%%%%%%%%%%%%%%%%%%%%%%%%%%
Let $x_{\half}, y_{\half}\in A_{\half}(a),$ then
\begin{enumerate}
    \item 
 \begin{gather*}
\textstyle{\half(a,b)(1-(a,b))(x_{\half},y_{\half})a+(1-(a,b))(x_{\half}y_{\half})b_0+(1-(a,b))b_{\half}(x_{\half}y_{\half})}\\
-(a,b)(x_{\half},y_{\half}b_0)a-2(x_{\half}(y_{\half}b_0))b_0-2b_{\half}(x_{\half}(y_{\half}b_0))\\
-2 ((b_{\half}y_{\half})x_{\half})b_0-2b_{\half}((b_{\half}y_{\half})x_{\half})\\
=\textstyle{\half(a,b)(1-(a,b))x_{\half}y_{\half}+(1-(a,b))y_{\half}(x_{\half}b_0)+(1-(a,b))(b_{\half}x_{\half})y_{\half}-(a,b)x_{\half}(y_{\half}b_0)}\\
-2(x_{\half}b_0)(y_{\half}b_0)-2 (b_{\half}y_{\half})(x_{\half}b_0)\\
-2(b_{\half}x_{\half})(y_{\half}b_0)-2 (b_{\half}x_{\half})(b_{\half}y_{\half})
\end{gather*}

\item 
Comparing the $\half$-component in (1) we get
\begin{gather*}
(1-(a,b))b_{\half}(x_{\half}y_{\half})
+2(b_{\half}y_{\half})(x_{\half}b_0)
+2(b_{\half}x_{\half})(y_{\half}b_0)\\
=
(1-(a,b))(b_{\half}x_{\half})y_{\half} +2b_{\half}(x_{\half}(y_{\half}b_0))
 +2((b_{\half}y_{\half})x_{\half})b_0.
\end{gather*}

 \item 
Interchanging $x_{\half}$ and $y_{\half}$ in (2) we get:
\begin{enumerate}
\item
by subtracting:
\begin{gather*}
(1-(a,b))(b_{\half}x_{\half})y_{\half} +2b_{\half}(x_{\half}(y_{\half}b_0))
 +2((b_{\half}y_{\half})x_{\half})b_0\quad =\\
 (1-(a,b))(b_{\half}y_{\half})x_{\half} +2b_{\half}(y_{\half}(x_{\half}b_0))
 +2((b_{\half}x_{\half})y_{\half})b_0.
\end{gather*}

\item
By adding:
\begin{gather*}
2(1-(a,b))b_{\half}(x_{\half}y_{\half})
+4(b_{\half}y_{\half})(x_{\half}b_0)
+4(b_{\half}x_{\half})(y_{\half}b_0)\quad =
\\
(1-(a,b))\Big[(b_{\half}x_{\half})y_{\half}+(b_{\half}y_{\half})x_{\half}\Big]
+2b_{\half}(x_{\half}(y_{\half}b_0))+2b_{\half}(y_{\half}(x_{\half}b_0))\\
 +2((b_{\half}y_{\half})x_{\half})b_0+2((b_{\half}x_{\half})y_{\half})b_0\\\iff
 \\
   ((a,b)-1)\Big[(b_{\half}x_{\half})y_{\half}+(b_{\half}y_{\half})x_{\half}\Big]
+4(b_{\half}y_{\half})(x_{\half}b_0)
+4(b_{\half}x_{\half})(y_{\half}b_0)\quad =
\\
 2((a,b)-1)b_{\half}(x_{\half}y_{\half})+2b_{\half}(x_{\half}(y_{\half}b_0))+2b_{\half}(y_{\half}(x_{\half}b_0))\\
 +2((b_{\half}y_{\half})x_{\half})b_0+2((b_{\half}x_{\half})y_{\half})b_0
\end{gather*} 
\end{enumerate}
\item 
Comparing the $0$-components in (1) we get
\begin{gather*}
\textstyle{(1-(a,b))(x_{\half}y_{\half})b_0-\half(a,b)(1-(a,b))\Big(x_{\half}y_{\half}\Big)_0}\\
+2\Big((x_{\half}b_0)(y_{\half}b_0)\Big)_0+2 \Big((b_{\half}x_{\half})(b_{\half}y_{\half})\Big)_0
 \\
= (1-(a,b))\Big(y_{\half} (x_{\half}b_0)\Big)_0-(a,b)\Big(x_{\half}(y_{\half}b_0)\Big)_0\\
+2((x_{\half}(y_{\half}b_0))b_0+2\Big(b_{\half}((b_{\half}y_{\half})x_{\half})\Big)_0\\\iff
\\
(a,b)(1-(a,b))\Big(x_{\half}y_{\half}\Big)_0
-4\Big((x_{\half}b_0)(y_{\half}b_0)\Big)_0-4 \Big((b_{\half}x_{\half})(b_{\half}y_{\half})\Big)_0
 \\
= -2(1-(a,b))\Big(y_{\half} (x_{\half}b_0)\Big)_0+2(a,b)\Big(x_{\half}(y_{\half}b_0)\Big)_0\\
-4((x_{\half}(y_{\half}b_0))b_0-4\Big(b_{\half}((b_{\half}y_{\half})x_{\half})\Big)_0+2(1-(a,b))(x_{\half}y_{\half})b_0
\end{gather*}
 
 \item 
Interchanging $x_{\half}, y_{\half}$ in (4) and subtracting we get,
\begin{gather*}
\Big(y_{\half}(x_{\half}b_0)\Big)_0+2(x_{\half}(y_{\half}b_0))b_0+2\Big(b_{\half}((b_{\half}y_{\half})x_{\half})\Big)_0\\
=\Big(x_{\half}(y_{\half}b_0)\Big)_0+2(y_{\half}(x_{\half}b_0))b_0+2\Big(b_{\half}((b_{\half}x_{\half})y_{\half})\Big)_0
\end{gather*}

\item 
Interchanging $x_{\half}, y_{\half}$ in (4) and adding we  get
\begin{gather*}
2(1-(a,b))(x_{\half}y_{\half})b_0-(a,b)(1-(a,b))\Big(x_{\half}y_{\half}\Big)_0\\
+4\Big((x_{\half}b_0)(y_{\half}b_0)\Big)_0+4 \Big((b_{\half}x_{\half})(b_{\half}y_{\half})\Big)_0
 \\
= (1-2(a,b))\Big(y_{\half} (x_{\half}b_0)\Big)_0+(1-2(a,b))\Big(x_{\half}(y_{\half}b_0)\Big)_0\\
+2((x_{\half}(y_{\half}b_0))b_0+2((y_{\half}(x_{\half}b_0))b_0+2\Big(b_{\half}((b_{\half}y_{\half})x_{\half})\Big)_0+2\Big(b_{\half}((b_{\half}x_{\half})y_{\half})\Big)_0\\\iff
\\
 (a,b)((a,b)-1)\Big(x_{\half}y_{\half}\Big)_0+(2(a,b)-1)\Big[\Big(y_{\half} (x_{\half}b_0)\Big)_0+\Big(x_{\half}(y_{\half}b_0)\Big)_0\Big]\\
+4\Big((x_{\half}b_0)(y_{\half}b_0)\Big)_0+4 \Big((b_{\half}x_{\half})(b_{\half}y_{\half})\Big)_0
 \\
=-2(1-(a,b))(x_{\half}y_{\half})b_0+  \\
+2((x_{\half}(y_{\half}b_0))b_0+2((y_{\half}(x_{\half}b_0))b_0+2\Big(b_{\half}((b_{\half}y_{\half})x_{\half})\Big)_0+2\Big(b_{\half}((b_{\half}x_{\half})y_{\half})\Big)_0
\end{gather*}

\item 
If $(a,b)\ne 1,$ then
$((b_{\half}x_{\half})y_{\half})b_0+((b_{\half}y_{\half})x_{\half})b_0=\\
\half\Big[(b,y_{\half})x_{\half}b_0+(b,x_{\half})y_{\half}b_0\Big]+\frac 14(1-(a,b))(a,x_{\half}y_{\half})b_{\half}-\frac 14(b_0,x_{\half}y_{\half})b_{\half}.$

\item 
If $(a,b)\ne 1,$ then
$
b_{\half}((x_{\half}b_0)y_{\half})+(b_{\half}y_{\half} )(x_{\half}b_0)\\
=\half(b,y_{\half})x_{\half}b_0 -\frac 14(b,x_{\half})y_{\half}b_0+\frac 18(1-(a,b)(b,x_{\half})y_{\half} \\
   +\half(b_0, x_{\half}y_{\half})b_{\half} 
$
\end{enumerate}
\end{lemma}
\begin{proof}
(1)  Replacing $x=x_{\half}, y=y_{\half}$ in Lemma \ref{P}(2) we get
\begin{gather*}
(a,b)(1-(a,b))a(x_{\half}y_{\half})+b_0(x_{\half}y_{\half})+b_{\half}(x_{\half}y_{\half})-2\Big((a,b)b_0((ay_{\half})x_{\half})+(a,b)b_{\half}((ay_{\half})x_{\half})\Big)\\
-2\Big((a,b)a((b_0y_{\half})x_{\half})+b_0((b_0y_{\half})x_{\half})+b_{\half}((b_0y_{\half})x_{\half})\Big)\\
-2\Big((a,b)a((b_{\half}y_{\half})x_{\half})+b_0((b_{\half}y_{\half})x_{\half})+b_{\half}((b_{\half}y_{\half})x_{\half})\Big)\\
=(a,b)(1-(a,b))(ax_{\half})y_{\half}+(b_0x_{\half})y_{\half}+(b_{\half}x_{\half})y_{\half}-2\Big((a,b)(ax_{\half})(b_0y_{\half})+(a,b)(ax_{\half})(b_{\half}y_{\half}) \Big)\\
-2\Big((a,b)(b_0x_{\half})(ay_{\half})+(b_0x_{\half})(b_0y_{\half})+ (b_0x_{\half})(b_{\half}y_{\half})\Big)\\
-2\Big((a,b) (b_{\half}x_{\half})(ay_{\half})+ (b_{\half}x_{\half})(b_0y_{\half})+ (b_{\half}x_{\half})(b_{\half}y_{\half}) 
\\\iff
\\
\overbrace{\textstyle{\half(a,b)(1-(a,b))(x_{\half},y_{\half})a}}^{(1)}+\overbrace{(x_{\half}y_{\half})b_0}^{(2)}+\overbrace{b_{\half}(x_{\half}y_{\half})}^{(3)}-2\Big(\overbrace{\textstyle{\half(a,b)(x_{\half}y_{\half})b_0}}^{(2)}+\overbrace{\textstyle{\half(a,b)b_{\half}(x_{\half}y_{\half})}}^{(3)}\Big)\\
-2\Big(\overbrace{\textstyle{\half(a,b)(x_{\half},y_{\half}b_0)a}}^{(4)}+\overbrace{(x_{\half}(y_{\half}b_0))b_0}^{(5)}+\overbrace{b_{\half}(x_{\half}(y_{\half}b_0))}^{(6)}\\
-2\Big(\overbrace{\textstyle{\half(a,b)(b_{\half}y_{\half})x_{\half}}}^{(7)}+ \overbrace{((b_{\half}y_{\half})x_{\half})b_0}^{(8)}+\overbrace{b_{\half}((b_{\half}y_{\half})x_{\half})}^{(9)}\Big)\\
=\overbrace{\textstyle{\half(a,b)(1-(a,b))x_{\half}y_{\half}}}^{(10)}+\overbrace{y_{\half}(x_{\half}b_0)}^{(11)}+\overbrace{(b_{\half}x_{\half})y_{\half}}^{(12)}-2\Big(\overbrace{\textstyle{\half(a,b)x_{\half}(y_{\half}b_0)}}^{(13)}+\overbrace{\textstyle{\half(a,b)x_{\half}(b_{\half}y_{\half})}}^{(7)} \Big)\\
-2\Big(\overbrace{\textstyle{\half(a,b)y_{\half}(x_{\half}b_0)}}^{(14)} +(x_{\half}b_0)(y_{\half}b_0)+  (b_{\half}y_{\half})(x_{\half}b_0)\Big)\\
-2\Big(\overbrace{\textstyle{\half(a,b)(b_{\half}x_{\half})y_{\half}}}^{(12)}+ (b_{\half}x_{\half})(y_{\half}b_0)+ (b_{\half}x_{\half})(b_{\half}y_{\half}) \Big)
\end{gather*}
\begin{gather*}
\iff\\
\overbrace{\textstyle{\half(a,b)(1-(a,b))(x_{\half},y_{\half})a}}^{(1)}+\overbrace{(1-(a,b))(x_{\half}y_{\half})b_0}^{(2)}+\overbrace{(1-(a,b))b_{\half}(x_{\half}y_{\half})}^{(3)}\\
\overbrace{-(a,b)(x_{\half},y_{\half}b_0)a}^{(4)}\overbrace{-2(x_{\half}(y_{\half}b_0))b_0}^{(5)}\overbrace{-2b_{\half}(x_{\half}(y_{\half}b_0))}^{(6)}\\
-2\overbrace{((b_{\half}y_{\half})x_{\half})b_0}^{(8)}-2\overbrace{b_{\half}((b_{\half}y_{\half})x_{\half})}^{(9)}\\
=\overbrace{\textstyle{\half(a,b)(1-(a,b))x_{\half}y_{\half}}}^{(10)}+\overbrace{y_{\half}(x_{\half}b_0)}^{(11)}+\overbrace{(1-(a,b))(b_{\half}x_{\half})y_{\half}}^{(12)}-\overbrace{(a,b)x_{\half}(y_{\half}b_0)}^{(13)}\\
\overbrace{-(a,b)y_{\half}(x_{\half}b_0)}^{(14)} -2(x_{\half}b_0)(y_{\half}b_0)-2 (b_{\half}y_{\half})(x_{\half}b_0)\\
-2(b_{\half}x_{\half})(y_{\half}b_0)-2 (b_{\half}x_{\half})(b_{\half}y_{\half}).
\end{gather*}
\medskip

\noindent
(2)  Comparing the $\half$-component in (1) we get
\begin{gather*}
(1-(a,b))b_{\half}(x_{\half}y_{\half})
-2b_{\half}((y_{\half}b_0)x_{\half})
-2((b_{\half}y_{\half})x_{\half})b_0\\
=\quad (1-(a,b))(b_{\half}x_{\half})y_{\half}
-2(b_{\half}y_{\half})(x_{\half}b_0)
-2(b_{\half}x_{\half})(y_{\half}b_0).
\end{gather*}
\medskip

\noindent
(3)  This is obvious.
\medskip

\noindent
(4) Comparing the $0$-components in (1) we get
  \begin{gather*}
(1-(a,b))(x_{\half}y_{\half})b_0
-2(x_{\half}(y_{\half}b_0))b_0
-2\Big(b_{\half}((b_{\half}y_{\half})x_{\half})\Big)_0\\
=\textstyle{\half(a,b)(1-(a,b))\Big(x_{\half}y_{\half}\Big)_0+(1-(a,b))\Big(y_{\half}(x_{\half}b_0)\Big)_0-(a,b)\Big(x_{\half}(y_{\half}b_0)\Big)_0}\\
-2\Big((x_{\half}b_0)(y_{\half}b_0)\Big)_0
-2 \Big((b_{\half}x_{\half})(b_{\half}y_{\half})\Big)_0.
\end{gather*}
\medskip

\noindent
(5)  Interchanging $x_{\half}, y_{\half}$ in (4) we get,
\begin{gather*}
(1-(a,b))\Big(y_{\half}(x_{\half}b_0)\Big)_0-(a,b)\Big(x_{\half}(y_{\half}b_0)\Big)_0+2(x_{\half}(y_{\half}b_0))b_0+2\Big(b_{\half}((b_{\half}y_{\half})x_{\half})\Big)_0\\
=(1-(a,b))\Big(x_{\half}(y_{\half}b_0)\Big)_0-(a,b)\Big(y_{\half}(x_{\half}b_0)\Big)_0+2(y_{\half}(x_{\half}b_0))b_0+2\Big(b_{\half}((b_{\half}x_{\half})y_{\half})\Big)_0\\\iff
\\
\Big(y_{\half}(x_{\half}b_0)\Big)_0+2(x_{\half}(y_{\half}b_0))b_0+2\Big(b_{\half}((b_{\half}y_{\half})x_{\half})\Big)_0\\
=\Big(x_{\half}(y_{\half}b_0)\Big)_0+2(y_{\half}(x_{\half}b_0))b_0+2\Big(b_{\half}((b_{\half}x_{\half})y_{\half})\Big)_0.
\end{gather*}
\medskip

\noindent
(6) This is obvious.
\medskip

\noindent
(7) We multiply (2) by $b_0$ to get
\begin{gather*}
(1-(a,b))(b_{\half}(x_{\half}y_{\half}))b_0
+2((b_{\half}y_{\half})(x_{\half}b_0))b_0
+2((b_{\half}x_{\half})(y_{\half}b_0))b_0\\
=
(1-(a,b))\Big[((b_{\half}x_{\half})y_{\half})b_0+ ((b_{\half}y_{\half})x_{\half})b_0\Big]+2(b_{\half}(x_{\half}(y_{\half}b_0)))b_0.
 \end{gather*}
Using Lemma \ref{SSI}(2), and Lemma \ref{P1212}(5) we get
\begin{gather*}
\textstyle{(1-(a,b))(b_{\half}(x_{\half}y_{\half}))b_0
+\half(1-(a,b))\Big[(b,y_{\half})x_{\half}b_0+(b,x_{\half})y_{\half}b_0\Big]}\\
=
\textstyle{(1-(a,b))\Big[((b_{\half}x_{\half})y_{\half})b_0+ ((b_{\half}y_{\half})x_{\half})b_0\Big]+\half(1-(a,b))(b_0,x_{\half}y_{\half})b_{\half}.}
\end{gather*}
So if $(a,b)\ne 1,$ then, by Lemma \ref{SSI}(1),
\begin{gather*}
\textstyle{\frac 14(1-(a,b))(a,x_{\half}y_{\half})b_{\half}+\frac 14(b_0,x_{\half}y_{\half})b_{\half} +\half\Big[(b,y_{\half})x_{\half}b_0+(b,x_{\half})y_{\half}b_0\Big]}\\
\textstyle{=((b_{\half}x_{\half})y_{\half})b_0+ ((b_{\half}y_{\half})x_{\half})b_0+ \half(b_0,x_{\half}y_{\half})b_{\half}.}
\end{gather*}
\medskip

\noindent
(8)
We replace  $x_{\half}$ with $x_{\half}b_0$ in (2).  We  use Lemma \ref{SSI}(10) and Lemma \ref{P1212}((5)\&(7)), to get
\begin{gather*}
(1-(a,b))b_{\half}((x_{\half}b_0)y_{\half})
+2(b_{\half}y_{\half})((x_{\half}b_0)b_0)
+2(b_{\half}(x_{\half}b_0))(y_{\half}b_0)\\
=
(1-(a,b))(b_{\half}(x_{\half}b_0))y_{\half} +2b_{\half}((x_{\half}b_0)(y_{\half}b_0))
 +2((b_{\half}y_{\half})(x_{\half}b_0))b_0\\\iff
\\
\textstyle{(1-(a,b))b_{\half}(y_{\half}(x_{\half}b_0) )
+(1-(a,b))(b_{\half}y_{\half})(x_{\half}b_0)}
\textstyle{+\half(1-(a,b))(b,x_{\half})y_{\half}b_0}\\
=
\textstyle{\frac 18(1-(a,b)^2(b,x_{\half})y_{\half}+\frac 14(1-(a,b)(b,x_{\half})y_{\half}b_0}\\
\textstyle{+\half(1-(a,b))(b_0,x_{\half}y_{\half})b_{\half}
 +\half(1-(a,b))(b,y_{\half})x_{\half}b_0}\\\text{since $(a,b)\ne 1$}\implies
 \\
 b_{\half}(y_{\half}(x_{\half}b_0) )
+(b_{\half}y_{\half})(x_{\half}b_0)\\
=
\textstyle{\frac 18(1-(a,b)(b,x_{\half})y_{\half}-\frac 14(b,x_{\half})y_{\half}b_0}\\
\textstyle{\half(b_0,x_{\half}y_{\half})b_{\half}
 +\half(b,y_{\half})x_{\half}b_0.}
\end{gather*}
Hence by Lemma \ref{P1212}(7), the claim holds.
\end{proof}

%%%%%%%%%%%%%%%%%%%%%%%%%%%%%%
\begin{lemma}\label{P1}
%%%%%%%%%%%%%%%%%%%%%%%%%%%%%%%%%%%%
Let $b\in A$ be an axis of Jordan type half.  Then
\begin{enumerate}
\item 
$4(bx)(by)-(bx)y -(by)x-(b,y)bx -(b, x)by- (b,xy)b + b(xy)=0.$

\item 
\begin{gather*}
4(bx)(by)=4(((a,b)a+b_0+b_{\half})x)((a,b)a+b_0+b_{\half})y)\\
=4(a,b)^2(ax)(ay)+4(a,b)(ax)(yb_0)+4(a,b)(ay)(xb_0)+4(a,b)(ax)(yb_{\half})\\
+4(a,b)(ay)(xb_{\half})+4(xb_0)(yb_0)+4(xb_{\half})(yb_{\half})+4(xb_0)(yb_{\half})+4(yb_0)(xb_{\half})
\end{gather*}

\item 
$-(bx)y=-(a,b)(ax)y-(xb_0)y-(xb_{\half})y$  and   $-(by)x=-(a,b)(ay)x-(yb_0)x-(yb_{\half})x.$

\item 
$-(b,y)bx=-(b,y)(a,b)ax-(b,y)xb_0-(b,y)xb_{\half}$ and  $-(b,x)by=-(b,x)(a,b)ay-(b,x)yb_0-(b,x)yb_{\half}.$

\item 
$-(b,xy)b+b(xy)=-(b,xy)(a,b)a-(b,xy)b_0-(b,xy)b_{\half}+(a,b)a(xy)+(xy)b_0+(xy)b_{\half}.$
 
\end{enumerate}
\end{lemma}
\begin{proof}
(1) follows from Lemma \ref{QP}, since $P_b(x,y)=0,$ for all $x,y\in A.$ Then (2)--(5) are obvious.

%\begin{gather*}
%b\Big((2bx-2(b,x)b)(2by-2(b,y)b)\Big)=\Big(b,(2bx-2(b,x)b)(2by-2(b,y)b)\Big)b \\
%=\Big(bx-(b,x)b,2by-2(b,y)b\Big)b=(bx,2by-2(b,y)b)b\\
%=(x,2b(by)-2(b,y)b)b=2((bx,by)-(b,x)(b,y))b
%\end{gather*}
%Also
%\begin{gather*}
%4b\Big((bx-(b,x)b)(by-(b,y)b)\Big)=4b\Big((bx)(by)-(b,y)b(bx)-(b,x)b(by)+(b,x)(b,y)b\Big)
%\end{gather*}
\end{proof}

%%%%%%%%%%%%%%%%%%%%%%%%%%%%%
\begin{remark}
In Lemmas \ref{Pa0},  \ref{P00} and \ref{P1212} ahead we replace $x,y$ by appropriate elements, we sum up parts (2),(3),(4) and (5) of Lemma \ref{P1} and equal the sum to zero.
\end{remark}

%%%%%%%%%%%%%%%%%%%%%%%%%%%%%%%%%%%
%\begin{lemma}\label{Paa}
%%%%%%%%%%%%%%%%%%%%%%%%%%%%%%%
%Putting $x=y=a$ in Lemma \ref{P1} we get 
%\[
%b_{\half}^2=(1-(a,b))(a,b)a+(a,b)b_0.
%\]
%\end{lemma}
%\begin{proof}
%\begin{gather*}
%4(a,b)^2(aa)(aa)+4(a,b)(aa)(ab_0)+4(a,b)(aa)(ab_0)+4(a,b)(aa)(ab_{\half})+4(a,b)(aa)(ab_{\half})\\
%
%+4(ab_0)(ab_0)+4(ab_{\half})(ab_{\half})+4(ab_0)(ab_{\half})+4(ab_0)(ab_{\half})\\
%
%-2(a,b)(aa)a-2(ab_0)a-2(ab_{\half})a\\
%
%-2(b,a)(a,b)aa-2(b,a)(ab_0)-2(b,a)ab_{\half}\\
%
%-(b,aa)b+b(aa)\\
%\\
%=4(a,b)^2a+2(a,b)b_{\half}\\
%+(b_{\half})^2\\
%-2(a,b)a-\half b_{\half}\\
%-2(a,b)^2a-(a,b)b_{\half}\\
%-(a,b)^2a-(a,b)b_0-(a,b)b_{\half}+(a,b)a+\half b_{\half}\\
%= b_{\half}^2+(a,b)(4(a,b)-2-2(a,b)-(a,b)+1)a-(a,b)b_0\\\text{hence}
%\\
%b_{\half}^2=(1-(a,b))(a,b)a+(a,b)b_0.
%\end{gather*}
%\end{proof}

%%%%%%%%%%%%%%%%%%%%%%%%%%%%%%%
\begin{lemma}\label{Pa0}
%%%%%%%%%%%%%%%%%%%%%%%%%%%%%
Putting $x=x_0, y=a$ in Lemma \ref{P1} gives
\begin{enumerate}
\item 
$2b_{\half}(x_0b_0)=(1-(a,b))b_{\half}x_0+\half(b,x_0)b_{\half}.$

\item 
$2b_{\half}(b_{\half}x_0)=(a,b)(b,x_0)a+(a,b)x_0b_0.$

\item 
$2(x_0b_0)b_0=(1-(a,b))x_0b_0+(b,x_0)b_0.$

\item 
$(b_{\half}x_0)b_0=\frac 14(b,x_0)b_{\half}.$
%
%\item
%$(b,y_0(x_0b_0))=\half(1-(a,b)(b,x_0y_0)+\half(b,x_0)(b,y_0).$
%
%\item 
%For $w_0:=x_0(y_0b_0)-y_0(x_0b_0),$ we have $w_0b_0=\half(1-(a,b))w_0.$
\end{enumerate}
\end{lemma}
\begin{proof}
(1)\&(2):  By Lemma \ref{P1}(1),
\begin{gather*}
0=\\
4(a,b)^2(ax_0)(aa)+4(a,b)(ax_0)(ab_0)+4(a,b)(aa)(x_0b_0)+4(a,b)(ax_0)(ab_{\half})\\
+4(a,b)(aa)(x_0b_{\half})+4(x_0b_0)(ab_0)+4(x_0b_{\half})(ab_{\half})+4(x_0b_0)(ab_{\half})+4(ab_0)(x_0b_{\half})\\
-(a,b)(ax_0)a-(x_0b_0)a-(x_0b_{\half})a-(a,b)(aa)x_0-(ab_0)x_0-(ab_{\half})x_0\\
-(b,a)(a,b)ax_0-(b,a)x_0b_0-(b,a)x_0b_{\half}-(b,x_0)(a,b)aa-(b,x_0)ab_0-(b,x_0)ab_{\half}\\
-(b,x_0a)b+b(x_0a)\\
=\textstyle{2(a,b)b_{\half}x_0 
+2b_{\half}(b_{\half}x_0 ) +2b_{\half}(x_0b_0)-\half b_{\half}x_0-\half b_{\half}x_0}\\
\textstyle{-(a,b)x_0b_0-(a,b)b_{\half}x_0 -(a,b)(b,x_0)a-\half(b,x_0)b_{\half}.}
\end{gather*}
Hence
\begin{gather*}
\textstyle{((a,b)-1)b_{\half}x_0+2b_{\half}(x_0b_0)-\half(b,x_0)b_{\half}=0,}
\end{gather*}
and
\[
2b_{\half}(b_{\half}x_0)-(a,b)(b,x_0) a-(a,b)x_0b_0=0
\]
\medskip

\noindent
(3) Using (2) and Lemma \ref{Q}(4),
\begin{gather*}
(a,b)x_0b_0=x_0b_0+(b,x_0)b_0-2(x_0b_0)b_0\ \iff\\
2(x_0b_0)b_0=(1-(a,b))x_0b_0+(b,x_0)b_0.
\end{gather*}
\medskip

\noindent
(4)  By Lemma \ref{P}(3), $(1-(a,b))b_{\half}x_0-2b_{\half}(x_0b_0)+2(b_{\half}x_0)b_0=0,$  hence, by (1), $-\half(b,x_0)b_{\half}+2(b_{\half}x_0)b_0 =0.$  This shows (4).
%\medskip
%
%\noindent
%(5)  By (3),
%\begin{gather*}
%2(b,y_0(x_0b_0))=(2x_0b_0)b_0,y_0)=(1-(a,b))(x_0b_0,y_0)+(b,x_0)(b_0,y_0)\\
%=(1-(a,b)(b,x_0y_0)+(b,x_0)(b,y_0).
%\end{gather*}
%\bigskip
%
%\noindent
%(6), By Lemma \ref{P}(\ref{19}), $4(w_0b_0)b_0=2(2-(a,b))w_0b_0-(1-(a,b))w_0.$  Also, since $(b,w_0)=0,$ $(w_0b0)b_0=\half(1-(a,b))w_0b_0,$
%by (3).  Hence
%\[
%2(1-(a,b))w_0b_0=2(2-(a,b))w_0b_0-(1-(a,b))w_0\implies w_0b_0=\half(1-(a,b))w_0.
%\]
\end{proof}

We summarize the results of Lemma \ref{Pa0} and part of Lemma \ref{P} in the following theorem.

%%%%%%%%%%%%%%%%%%%%%%%%%%%%%%%%
\begin{thm}\label{SI}
%%%%%%%%%%%%%%%%%%%%%%%%%%%%%5
Let $x_0\in A_0(a)$ and $x_{\half}\in A_{\half}(a),$ then
\begin{enumerate}
\item 
$(x_{\frac 12}b_0)b_0=\frac 12(1-(a,b))x_{\frac 12}b_0.$

\item 
 $(x_0b_0)b_0=\frac 12(1-(a,b))x_0b_0+\frac 12(b,x_0)b_0.$

 \item 
 $(x_0b_{\frac 12}) b_{\frac 12}=\frac 12(a,b)(b,x_0)a+\frac 12(a,b)x_0b_0.$

 \item
 \begin{enumerate}
 \item 
$ \Big(x_{\half}b_{\half} \Big)_0b_{\half} =\frac 14(a,b)(1-(a,b))x_{\half}-\half(a,b)x_{\half}b_0+\frac 14(b,x_{\half})b_{\half} .$
\item
$(x_{\half}b_{\half})b_{\half}= \frac 14(a,b)(1-(a,b))x_{\half}-\half(a,b)x_{\half}b_0+\half(b,x_{\half})b_{\half}.$ 
\end{enumerate}

 \item 
$(x_{\frac 12}b_0)b_{\frac 12}=\frac 14(1-(a,b))(b,x_{\frac 12})a+\frac 14(b,x_{\frac 12})b_0.$ 

  \item 
$(x_{\frac 12}b_{\frac 12})b_0=\frac 12(1-(a,b))\Big(b_{\frac 12}x_{\frac 12}\Big)_0+\frac 14 (b,x_{\frac 12})b_0.$

\item
$(x_0b_0)b_{\frac 12}=\frac 12(1-(a,b))b_{\frac 12}x_0+\frac 14(b,x_0)b_{\frac 12}.$

\item
$(x_0b_{\frac 12})b_0=\frac 14(b,x_0)b_{\frac 12}.\quad$ 
\end{enumerate}
\end{thm}
\begin{proof}
(1) is Lemma \ref{P}(4b). (2), (3), (7) and (8) come from Lemma \ref{Pa0}.  (4) is Lemma \ref{P}(4)(c)\&(d)), (5) is Lemma \ref{P}(5c) and (6) is Lemma \ref{P}(5d).
\end{proof}

\begin{lemma}\label{P00}, 
%%%%%%%%%%%%%%%%%%%%%%%%%%%%%
Putting $x=x_0$ and $y=y_0$ in Lemma \ref{P1} we get
\begin{enumerate}
\item\label{P001} 
\begin{gather*}
0=\\
4(b_{\half}y_0 )(x_0b_0)+4(b_{\half}x_0 )(y_0b_0)
-(b_{\half}x_0 )y_0-(b_{\half}y_0 )x_0\\
-(b,y_0)b_{\half}x_0-(b,x_0)b_{\half}y_0-(b,x_0y_0)b_{\half}+b_{\half}(x_0y_0). 
\end{gather*}

\item\label{P002} 
Hence
\begin{gather*}
b_{\half}(x_0y_0)-(b_{\half}x_0 )y_0-(b_{\half}y_0 )x_0\\
=-4(b_{\half}y_0 )(x_0b_0)-4(b_{\half}x_0 )(y_0b_0)
+(b,y_0)b_{\half}x_0+(b,x_0)b_{\half}y_0
+(b,x_0y_0)b_{\half}.  
\end{gather*}

\item\label{P003} 
$ ((b_{\half}y_0)x_0)b_0+((b_{\half}x_0)y_0)b_0=\frac 14(b,x_0y_0)b_{\half}.$

\item\label{P004} 
\begin{gather*}
4(b_{\half}x_0 )(b_{\half}y_0)+4(x_0b_0)(y_0b_0)
-y_0(x_0b_0)-x_0(y_0b_0)+(x_0y_0)b_0\\
=(b,y_0)x_0b_0+(b,x_0)y_0b_0
 +(b,x_0y_0)b_0 +(b,x_0y_0)(a,b)a.
\end{gather*}

\end{enumerate}
\end{lemma}
\begin{proof}
(\ref{P001})\&(\ref{P004})  By Lemma \ref{P1},
\begin{gather*}
0=\\
4(a,b)^2(ax_0)(ay_0)+4(a,b)(ax_0)(y_0b_0)+4(a,b)(ay_0)(x_0b_0)+4(a,b)(ax_0)(y_0b_{\half})\\
+4(a,b)(ay_0)(x_0b_{\half})+4(x_0b_0)(y_0b_0)+4(x_0b_{\half})(y_0b_{\half})+4(x_0b_0)(y_0b_{\half})+4(y_0b_0)(x_0b_{\half})\\
-(a,b)(ax_0)y_0-(x_0b_0)y_0-(x_0b_{\half})y_0-(a,b)(ay_0)x_0-(y_0b_0)x_0-(y_0b_{\half})x_0\\
-(b,y_0)(a,b)ax_0-(b,y_0)(x_0b_0)-(b,y_0)x_0b_{\half}-(b,x_0)(a,b)ay_0-(b,x_0)(y_0b_0)-(b,x_0)y_0b_{\half}\\
-(b,x_0y_0)(a,b)a-(b,x_0y_0)b_0-(b,x_0y_0)b_{\half}+(a,b)a(x_0y_0)+(x_0y_0)b_0+(x_0y_0)b_{\half}\\=
\\
4(x_0b_0)(y_0b_0)+4(x_0b_{\half})(y_0b_{\half})+4(x_0b_0)(y_0b_{\half})+4(y_0b_0)(x_0b_{\half})\\
-(x_0b_0)y_0-(x_0b_{\half})y_0-(y_0b_0)x_0-(y_0b_{\half})x_0\\
-(b,y_0)(x_0b_0)-(b,y_0)x_0b_{\half}-(b,x_0)(y_0b_0)-(b,x_0)y_0b_{\half}\\
-(b,x_0y_0)(a,b)a-(b,x_0y_0)b_0-(b,x_0y_0)b_{\half}+(x_0y_0)b_0+(x_0y_0)b_{\half}
\end{gather*}

Hence
\begin{gather*}
4(b_{\half}y_0 )(x_0b_0)+4(b_{\half}x_0 )(y_0b_0)-(b_{\half}x_0 )y_0-(b_{\half}y_0 )x_0\\
-(b,y_0)b_{\half}x_0-(b,x_0)b_{\half}y_0-(b,x_0y_0)b_{\half}+b_{\half}(x_0y_0)=0. 
\end{gather*}
This shows (1) and similarly (4) holds.
\medskip

\noindent
(\ref{P002})  This immediate from (\ref{P001}).
\medskip

\noindent
(\ref{P003})  By Lemma \ref{P}(8b) we have,
\begin{gather*}
2b_{\half}(x_0y_0)+4 (b_{\half}y_0)(x_0b_0)
+4(b_{\half}x_0)(b_0y_0)\\
=(1+(a,b))(b_{\half}x_0)y_0+(1+(a,b))(b_{\half}y_0)x_0+2b_{\half}((y_0b_0)x_0)+2b_{\half}((x_0b_0)y_0)\\
 +2((b_{\half}y_0)x_0)b_0+2((b_{\half}x_0)y_0)b_0.
\end{gather*}

Hence, by (\ref{P002}),
\begin{gather*}
b_{\half}(x_0y_0)-4(b_{\half}y_0 )(x_0b_0)-4(b_{\half}x_0 )(y_0b_0)
+(b,y_0)b_{\half}x_0+(b,x_0)b_{\half}y_0
+(b,x_0y_0)b_{\half}\\
+4 (b_{\half}y_0)(x_0b_0)
+4(b_{\half}x_0)(b_0y_0)\\
=(a,b)(b_{\half}x_0)y_0+(a,b)(b_{\half}y_0)x_0+2b_{\half}((y_0b_0)x_0)+2b_{\half}((x_0b_0)y_0)\\
 +2((b_{\half}y_0)x_0)b_0+2((b_{\half}x_0)y_0)b_0\\\iff
 \\
 b_{\half}(x_0y_0) 
+(b,y_0)b_{\half}x_0+(b,x_0)b_{\half}y_0
+(b,x_0y_0)b_{\half} \\
=(a,b)(b_{\half}x_0)y_0+(a,b)(b_{\half}y_0)x_0+2b_{\half}((y_0b_0)x_0)+2b_{\half}((x_0b_0)y_0)\\
 +2((b_{\half}y_0)x_0)b_0+2((b_{\half}x_0)y_0)b_0.
\end{gather*}
Multiplying by $b_0$  using $(z_{\half}b_0)b_0=\half(1-(a,b))z_{\half}b_0,\ \forall z_{\half}\in A_{\half}(a),$ and $(b_{\half}z_0)b_0=\frac 14(b,z_0)b_0,\ \forall z_0\in A_0(a),$ we get
\begin{gather*}
\textstyle{ \frac 14(b,x_0y_0)b_{\half}+\frac 12 (b,y_0)(b,x_0)b_{\half} 
+\half(1-(a,b))(b,x_0y_0)b_{\half}} \\
=
(a,b)(((b_{\half}x_0)y_0))b_0+(a,b)((b_{\half}y_0)x_0)b_0+(b,(y_0b_0)x_0)b_{\half}\\
+(1-(a,b))((b_{\half}y_0)x_0)b_0+(1-(a,b))(((b_{\half}x_0)y_0))b_0.
\end{gather*}
Now we use $(y_0b_0)b_0=\half(1-(a,b))y_0b_0+\half(b,y_0)b_0$ to see that
\begin{gather*}
\textstyle{(b,(y_0b_0)x_0)=((y_0b_0)b_0,x_0)=\half(1-(a,b))(y_0b_0,x_0)+\half(b,y_0)(b_0,x_0)}\\
=\textstyle{\half(1-(a,b))(b_0,x_0y_0)+\half(b,x_0)(b,y_0),}
\end{gather*}
so after cancellation we are left with 
$\frac 14(b,x_0y_0)b_{\half}
=((b_{\half}y_0)x_0)b_0+((b_{\half}x_0)y_0)b_0.$
%\medskip
%
%\noindent
%Taking $y_0=b_0,$ we get
%\begin{gather*}
%2(1-(a,b))b_{\half}(x_0b_0)+4(1-(a,b))(b_{\half}x_0 )b_0\\
%\\
%-(b_{\half}x_0 )b_0-\frac 12(1-(a,b))b_{\half}x_0\\
%\\
%-(1-(a,b))^2b_{\half}x_0-\half (1-(a,b))(b,x_0)b_{\half}\\
%\\
%-(1-(a,b))(b,x_0)b_{\half}+b_{\half}(x_0b_0) 
%\end{gather*}
%Also
%\begin{gather*}
%4(x_0b_0)(y_0b_0)+4(b_{\half}x_0 )(b_{\half}y_0)\\
%\\
%-y_0(x_0b_0)-x_0(y_0b_0)\\
%\\
%-(b,y_0)(x_0b_0)-(b,x_0)(y_0b_0)\\
%\\
%-(b,x_0y_0)(a,b)a-(b,x_0y_0)b_0+(x_0y_0)b_0
%\end{gather*}
\end{proof}

%%%%%%%%%%%%%%%%%%%%%%%%%%%%%%%%%%%%%%
%\begin{lemma}
%%%%%%%%%%%%%%%%%%%%%%%%%%%%%%%%
%Consider the identity 
%\begin{gather*}
%4(b_{\half}x_0 )(b_{\half}y_0)-(b,x_0y_0)(a,b)a+4(x_0b_0)(y_0b_0)
%-y_0(x_0b_0)-x_0(y_0b_0)\\
%-(b,y_0)(x_0b_0)-(b,x_0)(y_0b_0)
% -(b,x_0y_0)b_0+(x_0y_0)b_0,
%\end{gather*}

%\begin{gather*}
%(x_0y_0)b_0+2(x_0b_0)(y_0b_0)+2\Big((b_{\half}x_0)(b_{\half}y_0)\Big)_0\\
%=y_0(x_0b_0)+2(x_0(y_0b_0))b_0+2\Big(b_{\half}((b_{\half}y_0)x_0)\Big)_0.
%\end{gather*}

%from Lemma \ref{P00}.  Replacing $y_0$ with $y_0b_0$ we get
%\end{lemma}
%\begin{proof}
%We have
%\begin{gather*}
%4(b_{\half}x_0 )(b_{\half}(y_0b_0))-(b,x_0(y_0b_0))(a,b)a+4(x_0b_0)((y_0b_0)b_0)
%-(y_0b_0)(x_0b_0)-x_0((y_0b_0)b_0)\\
%-(b,(y_0b_0))(x_0b_0)-(b,x_0)((y_0b_0)b_0)
% -(b,x_0(y_0b_0))b_0+(x_0(y_0b_0))b_0\\
% \\
% =2(1-(a,b))(b_{\half}x_0)(b_{\half}y_0)+(b,y_0)b_{\half}(b_{\half}x_0)\\
% +2(1-(a,b))(x_0b_0)(y_0b_0)+2(b,y_0)(x_0b_0)b_0-(y_0b_0)(x_0b_0)\\
% -\half(1-(a,b))x_0(y_0b_0)-\half(b,y_0)x_0b_0-(1-(a,b))(b,y_0)x_0b_0\\
% -\half(1-(a,b))(b,x_0)y_0b_0-\half(b,x_0)(b,y_0)b_0\\
%  -(b,x_0(y_0b_0))b_0+(x_0(y_0b_0))b_0-(b,x_0(y_0b_0))(a,b)a\\
%  \\
%  =2(1-(a,b))(b_{\half}x_0)(b_{\half}y_0)+2(1-(a,b))(x_0b_0)(y_0b_0)\\
%  -(y_0b_0)(x_0b_0)
%\end{gather*}
%\end{proof}

%%%%%%%%%%%%%%%%%%%%%%%%%%%%%%%%%%%%
\begin{lemma}\label{P120}
%%%%%%%%%%%%%%%%%%%%%%%%%%%%%%%%%
Putting $x=x_{\half}$ and $y=y_0$ in Lemma \ref{P1} we get
\begin{enumerate}
\item 
\begin{gather*}
0=\\
(2(a,b)-1)x_{\half}(y_0b_0)
+4(x_{\half}b_0)(y_0b_0)+4(x_{\half}b_{\half})( b_{\half}y_0)
-(x_{\half}b_0)y_0\\
\textstyle{-\half(b,y_0)(a,b)x_{\half}-(b,y_0)x_{\half}b_0-(b,x_{\half})b_{\half}y_0 
-(b,x_{\half}y_0)b_{\half}+(x_{\half}y_0)b_0.}
\end{gather*}

\item 
\begin{gather*} 
0=\\
b_{\half}(x_{\half}y_0) +(2(a,b)-1) (b_{\half}y_0)x_{\half}-(b_{\half}x_{\half} )y_0
+4(x_{\half}b_0)(b_{\half}y_0) \\
+4 (b_{\half}x_{\half} )(y_0b_0)-(b,y_0)x_{\half}b_{\half}-(b,x_{\half})y_0b_0 
-(b,x_{\half}y_0)(a,b)a-(b,x_{\half}y_0)b_0.
\end{gather*}
\end{enumerate}
\end{lemma}
\begin{proof}
By Lemma \ref{P1}, replacing $x$ with $x_{\half}$ and $y$ with $y_0,$
\begin{gather*}
0=\\
4(a,b)^2(ax_{\half})(ay_0)+4(a,b)(ax_{\half})(y_0b_0)+4(a,b)(ay_0)(x_{\half}b_0)+4(a,b)(ax_{\half})(y_0b_{\half})\\+4(a,b)(ay_0)(x_{\half}b_{\half})+4(x_{\half}b_0)(y_0b_0)+4(x_{\half}b_{\half})(y_0b_{\half})+4(x_{\half}b_0)(y_0b_{\half})+4(y_0b_0)(x_{\half}b_{\half})\\
-(a,b)(ax_{\half})y_0-(x_{\half}b_0)y_0-(x_{\half}b_{\half})y_0-(a,b)(ay_0)x_{\half}-(y_0b_0)x_{\half}-(y_0b_{\half})x_{\half}\\
-(b,y_0)(a,b)ax_{\half}-(b,y_0)x_{\half}b_0-(b,y_0)x_{\half}b_{\half}-(b,x_{\half})(a,b)ay_0-(b,x_{\half})y_0b_0-(b,x_{\half})y_0b_{\half}\\
-(b,x_{\half}y_0)(a,b)a-(b,x_{\half}y_0)b_0-(b,x_{\half}y_0)b_{\half}+(a,b)a(x_{\half}y_0)+(x_{\half}y_0)b_0+(x_{\half}y_0)b_{\half}\\=
\end{gather*}
\begin{gather*}
2(a,b)x_{\half}(y_0b_0)+2(a,b)x_{\half}(b_{\half}y_0)\\
+4(x_{\half}b_0)(y_0b_0)+4(x_{\half}b_{\half})( b_{\half}y_0)+4(x_{\half}b_0)(b_{\half}y_0 )+4(y_0b_0)(x_{\half}b_{\half})\\
\textstyle{-\half(a,b)x_{\half}y_0-(x_{\half}b_0)y_0-(x_{\half}b_{\half})y_0-x_{\half}(y_0b_0) -(b_{\half}y_0 )x_{\half}}\\
\textstyle{-\half(b,y_0)(a,b)x_{\half}-(b,y_0)x_{\half}b_0-(b,y_0)x_{\half}b_{\half}-(b,x_{\half})y_0b_0-(b,x_{\half})b_{\half}y_0} \\
\textstyle{-(b,x_{\half}y_0)(a,b)a-(b,x_{\half}y_0)b_0-(b,x_{\half}y_0)b_{\half}+\half(a,b)x_{\half}y_0+(x_{\half}y_0)b_0+(x_{\half}y_0)b_{\half}.}
\end{gather*}
Hence, since the $\half$-component is $0,$
\begin{gather*}
0=\\
2(a,b)x_{\half}(y_0b_0)
+4(x_{\half}b_0)(y_0b_0)+4(x_{\half}b_{\half})( b_{\half}y_0)\\
\textstyle{-\half(a,b)x_{\half}y_0-(x_{\half}b_0)y_0-x_{\half}(y_0b_0)}\\
\textstyle{-\half(b,y_0)(a,b)x_{\half}-(b,y_0)x_{\half}b_0-(b,x_{\half})b_{\half}y_0} \\
\textstyle{-(b,x_{\half}y_0)b_{\half}+\half(a,b)x_{\half}y_0+(x_{\half}y_0)b_0.}
\end{gather*}
And, since the $1$-component plus the $0$-compoments are $0,$
\begin{gather*}
0=\\
2(a,b)x_{\half}(b_{\half}y_0)
+4(x_{\half}b_0)(b_{\half}y_0)+4(y_0b_0)(x_{\half}b_{\half})\\
-(x_{\half}b_{\half})y_0 -(b_{\half}y_0 )x_{\half}
-(b,y_0)x_{\half}b_{\half}-(b,x_{\half})y_0b_0 \\
-(b,x_{\half}y_0)(a,b)a-(b,x_{\half}y_0)b_0+(x_{\half}y_0)b_{\half}.\qedhere
\end{gather*}
\medskip

\noindent
\end{proof} 

%%%%%%%%%%%%%%%%%%%%%%%
\begin{lemma}\label{SSI}
 \begin{enumerate}
    \item
$(b_{\half}(x_{\half}y_{\half}))b_0=\frac 14(1-(a,b))(a,x_{\half}y_{\half})b_{\half}+\frac 14(b_0,x_{\half}y_{\half})b_{\half}.$

\item
$(b_{\half}(x_{\half}(y_{\half}b_0))b_0=\frac 14(1-(a,b))(b_0,x_{\half}y_{\half})b_{\half}.$

\item 
$(a,(x_{\half}b_0)y_{\half})=\half(b_0,x_{\half}y_{\half}).$

\item 
$(b_0,(x_{\half}b_0)y_{\half})=\half(1-(a,b))(b_0,x_{\half}y_{\half}).$

\item 
$(a, (x_{\half}b_0)(y_{\half}b_0))=\frac 14(1-(a,b))(b_0, x_{\half}y_{\half}).$

\item 
$(b_0,(x_{\half}b_0)(y_{\half}b_0))=\frac 14(1-(a,b))^2(b_0, x_{\half}y_{\half}).$

\item 
$(b, (x_{\half}b_0)y_{\half})=\half(b_0, x_{\half}y_{\half}).$

\item 
$(b,(x_{\half}b_0)(y_{\half}b_0))=\frac 14(1-(a,b))(b_0,x_{\half}y_{\half}). $

\item 
$(b_{\half}(x_{\half}b_0))y_{\half}=\frac 18(1-(a,b)(b,x_{\half})y_{\half}+\frac 14(b,x_{\half})y_{\half}b_0.$

\item 
$(b_{\half}(x_{\half}b_0))(y_{\half}b_0)=\frac 14(1-(a,b)(b,x_{\half})y_{\half}b_0.$
 \end{enumerate}   
\end{lemma}
\begin{proof}
(1) By Theorem \ref{SI}(1),
%\begin{gather*}
$(b_{\half}(x_{\half}y_{\half}))b_0=\Big(\Big[(a,x_{\half}y_{\half})a+\Big(x_{\half}y_{\half}\Big)_0\Big]b_{\half}\Big)b_0\\
\textstyle{=\frac 14(1-(a,b))(a,x_{\half}y_{\half})b_{\half}+\frac 14(b_0,x_{\half}y_{\half})b_{\half}.}$
\medskip

\noindent
(2) Replacing $y_{\half}$ with $y_{\half}b_0$ in (1), and using Theorem \ref{SI}(1), we see that
$\textstyle{(b_{\half}(x_{\half}(y_{\half}b_0))b_0=\frac 14(1-(a,b))(a,x_{\half}(y_{\half}b_0))b_{\half}+\frac 14(b_0,x_{\half}(y_{\half}b_0))b_{\half} }\\
=\textstyle{\frac 14(1-(a,b))(b_0,x_{\half}y_{\half})b_{\half}.}$
\medskip

\noindent
(3) is clear and (4) follows from Theorem \ref{SI}(1).  (5) follows from (3) replacing $y_{\half}$ with $y_{\half}b_0.$  Similarly (6) follows from (4).  For (7) we have $(b, (x_{\half}b_0)y_{\half})=\half(a,b)(b_0,x_{\half}y_{\half})+\half(1-(a,b))(b_0,x_{\half}y_{\half})=\half(b_0, x_{\half}y_{\half}),$ and (8) follow from (7).
\medskip

\noindent
(9)\&(10) $(b_{\half}(x_{\half}b_0)))y_{\half}=\Big[\frac 14(1-(a,b)(b,x_{\half})a+\frac 14(b,x_{\half})b_0\Big]y_{\half}=\frac 18(1-(a,b)(b,x_{\half})y_{\half}+\frac 14(b,x_{\half})y_{\half}b_0.$  This shows (9), and (10) is obtained by replacing in (9) $y_{\half}$ with $y_{\half}b_0.$
\end{proof}

%%%%%%%%%%%%%%%%%%%%%%%%%%%%%%%%%%%%%%%%%%
\begin{lemma}\label{P1212}
%%%%%%%%%%%%%%%%%%%%%%%%%%%%%%%%%%%%
For all $x_{\half}, y_{\half}\in A_{\half}(a),$
\begin{enumerate}
\item
\begin{gather*}
0=\\
(a,b)((a,b)-1)x_{\half}y_{\half}+(2(a,b)-1)x_{\half}(y_{\half}b_0)+(2(a,b)-1)y_{\half}(x_{\half}b_0)\\
+4(x_{\half}b_0)(y_{\half}b_0)+4(x_{\half}b_{\half})(y_{\half}b_{\half})\\
-(b,y_{\half})x_{\half}b_{\half}-(b,x_{\half})y_{\half}b_{\half}\\
-(b,x_{\half}y_{\half})(a,b)a-(b,x_{\half}y_{\half})b_0+(a,b)a(x_{\half}y_{\half})+(x_{\half}y_{\half})b_0\\\iff
\\
(2(a,b)-1)\Big(x_{\half}(y_{\half}b_0)\Big)_0+(2(a,b)-1)\Big(y_{\half}(x_{\half}b_0)\Big)_0\\
-(b,y_{\half})\Big(x_{\half}b_{\half}\Big)_0-(b,x_{\half})\Big(y_{\half}b_{\half}\Big)_0
+(x_{\half}y_{\half})b_0-(b,x_{\half}y_{\half})b_0\qquad =\\
(a,b)(1-(a,b))\Big(x_{\half}y_{\half}\Big)_0-4\Big((x_{\half}b_{\half})(y_{\half}b_{\half})\Big)_0-4\Big((x_{\half}b_0)(y_{\half}b_0)\Big)_0.
\end{gather*}

\item 
\begin{gather*}
0=\\
+(2(a,b)-1)\Big[(b_{\half}x_{\half} )y_{\half} +(b_{\half}y_{\half} )x_{\half}  \Big]
 +4(b_{\half}x_{\half} )(y_{\half}b_0) +4(b_{\half}y_{\half})(x_{\half}b_0)  \\
\textstyle{-\half(a,b)\Big[(b,y_{\half})x_{\half}+(b,x_{\half})y_{\half}\Big]-(b,y_{\half})x_{\half}b_0-(b,x_{\half})y_{\half}b_0}\\
 -(b,x_{\half}y_{\half})b_{\half}+(x_{\half}y_{\half})b_{\half}\\\iff
 \\
(2(a,b)-1)\Big[(b_{\half}x_{\half} )y_{\half} +(b_{\half}y_{\half} )x_{\half}  \Big]
 +4(b_{\half}x_{\half} )(y_{\half}b_0) +4(b_{\half}y_{\half})(x_{\half}b_0)\quad =  \\
\textstyle{\half(a,b)\Big[(b,y_{\half})x_{\half}+(b,x_{\half})y_{\half}\Big]+(b,y_{\half})x_{\half}b_0+(b,x_{\half})y_{\half}b_0}\\
 +(b,x_{\half}y_{\half})b_{\half}-(x_{\half}y_{\half})b_{\half}.
\end{gather*}

\item
\begin{gather*}
-(a,b)\Big[(b_{\half}x_{\half} )y_{\half} +(b_{\half}y_{\half} )x_{\half}  \Big]\\
\textstyle{+\half(a,b)\Big[(b,y_{\half})x_{\half}+(b,x_{\half})y_{\half}\Big]+(b,y_{\half})x_{\half}b_0+(b,x_{\half})y_{\half}b_0}\\
 +(b,x_{\half}y_{\half})b_{\half}\\=
\\
+2b_{\half}(x_{\half}(y_{\half}b_0))+2b_{\half}(y_{\half}(x_{\half}b_0))\\
+2((b_{\half}y_{\half})x_{\half})b_0+2((b_{\half}x_{\half})y_{\half})b_0\\
 +(2(a,b)-1)b_{\half}(x_{\half}y_{\half})
\end{gather*}

\item
\begin{equation}\label{eqP1212}
\begin{aligned}
&\Big[(b,y_{\half})x_{\half}b_0+(b,x_{\half})y_{\half}b_0\Big]-
2\Big[((b_{\half}y_{\half})x_{\half})b_0+((b_{\half}x_{\half})y_{\half})b_0\Big]\\
&\textstyle{=\half(b_0,x_{\half}y_{\half})b_{\half}-\half(1-(a,b))(a,x_{\half}y_{\half})b_{\half}}\\
&\textstyle{=\half(b,x_{\half}y_{\half})b_{\half}-\half(a,x_{\half}y_{\half})b_{\half}.}
 \end{aligned}
 \end{equation}

 \item
 $((b_{\half}y_{\half})(x_{\half}b_0))b_0=\frac 14(1-(a,b))(b,y_{\half})x_{\half}b_0.$ Hence,\\
 $\Big[\Big(b_{\half}y_{\half}\Big)_0(x_{\half}b_0)\Big]b_0=\frac 14(1-(a,b))(b,y_{\half})x_{\half}b_0-\frac 14(b,y_{\half})(x_{\half}b_0)b_0=\frac 18(1-(a,b))(b,y_{\half})x_{\half}b_0$

 \item 
\begin{gather*}
(1-2(a,b))b_{\half}(x_{\half}y_{\half})-(a,b)\Big[(b_{\half}x_{\half} )y_{\half} +(b_{\half}y_{\half} )x_{\half}  \Big]\\
\textstyle{+\half(a,b)\Big[(b,y_{\half})x_{\half}+(b,x_{\half})y_{\half}\Big]}\\
+\textstyle{\half(b_0,x_{\half}y_{\half})b_{\half}-\half(1-(a,b))(a,x_{\half}y_{\half})b_{\half}
 +(b,x_{\half}y_{\half})b_{\half}}\\=
\\
2b_{\half}(x_{\half}(y_{\half}b_0))+2b_{\half}(y_{\half} (x_{\half}b_0)).
\end{gather*}

\item 
\begin{enumerate}
\item
$b_{\half}((x_{\half}b_0)(y_{\half}b_0))=\frac 14(1-(a,b))(b_0,x_{\half}y_{\half})b_{\half}.$

\item 
$((x_{\half}b_0)(y_{\half}b_0))b_0=\frac 14(1-(a,b))(b_0,x_{\half}y_{\half})b_0$
\end{enumerate}

\item 
If $(a,b)\ne 0,$ then
\begin{gather*}
b_{\half}(y_{\half}(x_{\half}b_0) )+(b_{\half}y_{\half} )(x_{\half}b_0)\\
\textstyle{=\half(b,y_{\half})x_{\half}b_0 -\frac 14(b,x_{\half})y_{\half}b_0+\frac 18(1-(a,b)(b,x_{\half})y_{\half}} \\
   \textstyle{+\half(b_0, x_{\half}y_{\half})b_{\half}.} 
\end{gather*}

\end{enumerate}
\end{lemma}
\begin{proof}
(1)\&(2)  By Lemma \ref{P1},
\begin{gather*}
0=\\
4(a,b)^2(ax_{\half})(ay_{\half})+4(a,b)(ax_{\half})(y_{\half}b_0)+4(a,b)(ay_{\half})(x_{\half}b_0)+4(a,b)(ax_{\half})(y_{\half}b_{\half})\\
+4(a,b)(ay_{\half})(x_{\half}b_{\half})
+4(x_{\half}b_0)(y_{\half}b_0)+4(x_{\half}b_{\half})(y_{\half}b_{\half})+4(x_{\half}b_0)(y_{\half}b_{\half})+4(y_{\half}b_0)(x_{\half}b_{\half})\\
-(a,b)(ax_{\half})y_{\half}-(x_{\half}b_0)y_{\half}-(x_{\half}b_{\half})y_{\half}-(a,b)(ay_{\half})x_{\half}-(y_{\half}b_0)x_{\half}-(y_{\half}b_{\half})x_{\half}\\
-(b,y_{\half})(a,b)ax_{\half}-(b,y_{\half})x_{\half}b_0-(b,y_{\half})x_{\half}b_{\half}-(b,x_{\half})(a,b)ay_{\half}-(b,x_{\half})y_{\half}b_0-(b,x_{\half})y_{\half}b_{\half}\\
-(b,x_{\half}y_{\half})(a,b)a-(b,x_{\half}y_{\half})b_0-(b,x_{\half}y_{\half})b_{\half}+(a,b)a(x_{\half}y_{\half})+(x_{\half}y_{\half})b_0+(x_{\half}y_{\half})b_{\half}\\=
\end{gather*}
\begin{gather*}
(a,b)^2x_{\half}y_{\half}+2(a,b)x_{\half}(y_{\half}b_0)+2(a,b)y_{\half}(x_{\half}b_0)+2(a,b)x_{\half}(y_{\half}b_{\half})+2(a,b)y_{\half}(x_{\half}b_{\half})\\
+4(x_{\half}b_0)(y_{\half}b_0)+4(x_{\half}b_{\half})(y_{\half}b_{\half})+4(x_{\half}b_0)(y_{\half}b_{\half})+4(y_{\half}b_0)(x_{\half}b_{\half})\\
\textstyle{-\half(a,b)x_{\half}y_{\half}-(x_{\half}b_0)y_{\half}-(x_{\half}b_{\half})y_{\half}-\half(a,b)y_{\half}x_{\half}-(y_{\half}b_0)x_{\half}-(y_{\half}b_{\half})x_{\half}}\\
\textstyle{-\half(a,b)(b,y_{\half})x_{\half}-(b,y_{\half})x_{\half}b_0-(b,y_{\half})x_{\half}b_{\half}-\half(a,b)(b,x_{\half})y_{\half}-(b,x_{\half})y_{\half}b_0-(b,x_{\half})y_{\half}b_{\half}}\\
-(b,x_{\half}y_{\half})(a,b)a-(b,x_{\half}y_{\half})b_0-(b,x_{\half}y_{\half})b_{\half}+(a,b)a(x_{\half}y_{\half})+(x_{\half}y_{\half})b_0+(x_{\half}y_{\half})b_{\half}\\=
\end{gather*}
\begin{gather*}
(a,b)((a,b)-1)x_{\half}y_{\half}+(2(a,b)-1)x_{\half}(y_{\half}b_0)+(2(a,b)-1)y_{\half}(x_{\half}b_0)\\
+(2(a,b)-1)x_{\half}(y_{\half}b_{\half})+(2(a,b)-1)y_{\half}(x_{\half}b_{\half})\\
+4(x_{\half}b_0)(y_{\half}b_0)+4(x_{\half}b_{\half})(y_{\half}b_{\half})+4(x_{\half}b_0)(y_{\half}b_{\half})+4(y_{\half}b_0)(x_{\half}b_{\half})\\
\textstyle{-\half(a,b)(b,y_{\half})x_{\half}-(b,y_{\half})x_{\half}b_0-(b,y_{\half})x_{\half}b_{\half}-\half(a,b)(b,x_{\half})y_{\half}-(b,x_{\half})y_{\half}b_0-(b,x_{\half})y_{\half}b_{\half}}\\
-(b,x_{\half}y_{\half})(a,b)a-(b,x_{\half}y_{\half})b_0-(b,x_{\half}y_{\half})b_{\half}+(a,b)a(x_{\half}y_{\half})+(x_{\half}y_{\half})b_0+(x_{\half}y_{\half})b_{\half}.
\end{gather*}
Thus the $\{0,1\}$ component is $0,$ which is (1), and the $\half$-component is $0,$ which is (2).\qedhere 
\medskip

\noindent
(3)
To prove (3) we use  Lemma \ref{1212}(3)(b) and part (2) to get:
\begin{gather*}
-(a,b)\Big[(b_{\half}x_{\half} )y_{\half} +(b_{\half}y_{\half} )x_{\half}  \Big]\\
\textstyle{+\half(a,b)\Big[(b,y_{\half})x_{\half}+(b,x_{\half})y_{\half}\Big]+(b,y_{\half})x_{\half}b_0+(b,x_{\half})y_{\half}b_0}\\
 +(b,x_{\half}y_{\half})b_{\half}-b_{\half}(x_{\half}y_{\half})\\=
\\
+2b_{\half}(x_{\half}(y_{\half}b_0))+2b_{\half}(y_{\half}(x_{\half}b_0))\\
+2((b_{\half}y_{\half})x_{\half})b_0+2((b_{\half}x_{\half})y_{\half})b_0\\
 +2((a,b)-1)b_{\half}(x_{\half}y_{\half}),
\end{gather*}
and (3) follows.
\medskip

\noindent
(4) Multiplying (3) by $b_0,$ using Theorem \ref{SI}(1), we get
 \begin{gather*}
-(a,b)\Big[((b_{\half}x_{\half} )y_{\half})b_0 +((b_{\half}y_{\half} )x_{\half})b_0  \Big]\\
\textstyle{\half(a,b)\Big[(b,y_{\half})x_{\half}b_0+(b,x_{\half})y_{\half}b_0\Big]}\\
\textstyle{+\half(1-(a,b))\Big[(b,y_{\half})x_{\half}b_0+(b,x_{\half})y_{\half}b_0\Big]}\\
 \textstyle{+\half(1-(a,b))(b,x_{\half}y_{\half})b_{\half}}\\=
\\
+2(b_{\half}(x_{\half}(y_{\half}b_0)))b_0+2(b_{\half}(y_{\half}(x_{\half}b_0)))b_0
+(1-(a,b))\Big[((b_{\half}y_{\half})x_{\half})b_0+((b_{\half}x_{\half})y_{\half})b_0\Big]\\
 +(2(a,b)-1)(b_{\half}(x_{\half}y_{\half}))b_0\\\iff
 \\
\textstyle{\half\Big[(b,y_{\half})x_{\half}b_0+(b,x_{\half})y_{\half}b_0\Big]}\\
 \textstyle{+\half(1-(a,b))(b,x_{\half}y_{\half})b_{\half}}\\=
\\
+2(b_{\half}(x_{\half}(y_{\half}b_0)))b_0+2(b_{\half}(y_{\half}(x_{\half}b_0)))b_0\\
\Big[((b_{\half}y_{\half})x_{\half})b_0+((b_{\half}x_{\half})y_{\half})b_0\Big]\\
 +(2(a,b)-1)(b_{\half}(x_{\half}y_{\half}))b_0.
\end{gather*}
Using Lemma \ref{SSI}(1\&2) we get
\begin{gather*}
\textstyle{\half\Big[(b,y_{\half})x_{\half}b_0+(b,x_{\half})y_{\half}b_0\Big]}\\
 \textstyle{+\half(1-(a,b))(a,b)(a,x_{\half}y_{\half})b_{\half}+\half(1-(a,b))(b_0,x_{\half}y_{\half})b_{\half}}\\=
\\
(1-(a,b))(b_0,x_{\half}y_{\half})b_{\half}\\
+\Big[((b_{\half}y_{\half})x_{\half})b_0+((b_{\half}x_{\half})y_{\half})b_0\Big]\\
\textstyle{(2(a,b)-1)\Big[\frac 14(1-(a,b))(a,x_{\half}y_{\half})b_{\half}+\frac 14(b_0,x_{\half}y_{\half})b_{\half}\Big]}\\\iff
\\
\textstyle{\half\Big[(b,y_{\half})x_{\half}b_0+(b,x_{\half})y_{\half}b_0\Big]-\Big[((b_{\half}y_{\half})x_{\half})b_0+((b_{\half}x_{\half})y_{\half})b_0\Big]}\\=
\\
\textstyle{-\frac 14(1-(a,b))(a,x_{\half}y_{\half})b_{\half}+\frac 14(b_0,x_{\half}y_{\half})b_{\half}.}
\end{gather*}
\medskip

\noindent
(5)
 Putting $x_{\half}b_0$ in place of $x_{\half}$ in equation \eqref{eqP1212}  we get (using also Lemma \ref{SSI}(9)),
 \begin{gather*}
\Big[(b,y_{\half})(x_{\half}b_0)b_0+(b,x_{\half}b_0)y_{\half}b_0\Big]-
2\Big[((b_{\half}y_{\half})(x_{\half}b_0))b_0+((b_{\half}(x_{\half}b_0))y_{\half})b_0\Big]\\
=\textstyle{\half(b_0,(x_{\half}b_0)y_{\half})b_{\half}-\half(1-(a,b))(a,(x_{\half}b_0)y_{\half})b_{\half}}\\\iff
\\
\textstyle{\half(1-(a,b))\Big[(b,y_{\half})x_{\half}b_0+(b,x_{\half})y_{\half}b_0\Big]-
2\Big[((b_{\half}y_{\half})(x_{\half}b_0))b_0+\frac 14 (1-(a,b))(b,x_{\half})y_{\half}b_0\Big]}\\
=\textstyle{\frac 14(1-(a,b))(b_0,x_{\half}y_{\half})b_{\half}-\frac 14(1-(a,b))(b_0,x_{\half}y_{\half})b_{\half}}\\\iff
\\
\textstyle{((b_{\half}y_{\half})(x_{\half}b_0))b_0=\frac 14(1-(a,b))(b,y_{\half})x_{\half}b_0.}
 \end{gather*}
 \medskip

 \noindent
 (6)  This follows from (3) and (4).
 \medskip

 \noindent
 (7) 
(a) Replacing in (6), $x_{\half}, y_{\half}$ with $x_{\half}b_0, y_{\half}b_0,$ using Lemma \ref{SSI} (5), (6), (8) and (10), we get
\begin{gather*}
\textstyle{(1-2(a,b))b_{\half}((x_{\half}b_0)(y_{\half}b_0))-\frac 14(1-(a,b))(a,b)\Big[(b,x_{\half})y_{\half}b_0 +(b,y_{\half})x_{\half}b_0  \Big]}\\
\textstyle{+\frac 14(1-(a,b))(a,b)\Big[(b,y_{\half})x_{\half}b_0+(b,x_{\half})y_{\half}b_0\Big]}\\
\textstyle{+\frac 18(1-(a,b)^2(b_0,x_{\half}y_{\half})b_{\half}-\frac 18(1-(a,b))^2(b_0,x_{\half}y_{\half})b_{\half}
 +\frac 14(1-(a,b))(b_0,x_{\half}y_{\half})b_{\half}}\\=
\\
2(1-(a,b))
b_{\half}((x_{\half}b_0)(y_{\half}b_0)).
\end{gather*}
\medskip

\noindent
(b)
Replacing $x_{\half}, y_{\half}$ by $x_{\half}b_0, y_{\half}b_0$ in (1), and using Theorem \ref{SI}((1)\&(5)) and Lemma \ref{SSI}(8), we get,
\begin{gather*}
(2(a,b)-1)(1-(a,b))\Big((x_{\half}b_0)(y_{\half}b_0)\Big)_0\\
-(b,(y_{\half}b_0))\Big( b_{\half}(x_{\half}b_0)\Big)_0-(b,(x_{\half}b_0))\Big( b_{\half}(y_{\half}b_0)\Big)_0\\
+((x_{\half}b_0)(y_{\half}b_0))b_0-(b,(x_{\half}b_0)(y_{\half}b_0))b_0\qquad =\\
(a,b)(1-(a,b))\Big((x_{\half}b_0)(y_{\half}b_0)\Big)_0-4\Big((b_{\half}(x_{\half}b_0))( b_{\half}(y_{\half}b_0))\Big)_0\\
-(1-(a,b))^2\Big((x_{\half}b_0)(y_{\half}b_0)\Big)_0\\\iff
\\
\textstyle{((x_{\half}b_0)(y_{\half}b_0))b_0-\frac 14(b,x_{\half})(b,y_{\half})b_0-\frac 14(1-(a,b))(b_0,x_{\half}y_{\half})=}\\
\textstyle{-\frac 14(b,x_{\half})(b,y_{\half})b_0.}
\end{gather*}
 \medskip

 \noindent
 (8) Replacing in (6) $x_{\half}$ with $x_{\half}b_0,$ using Lemma \ref{SSI} and part (7) we get
\begin{gather*}
\textstyle{(1-2(a,b))b_{\half}((x_{\half}b_0)y_{\half})-\frac 18(1-(a,b)(a,b)(b,x_{\half})y_{\half}-\frac 14(a,b)(b,x_{\half})y_{\half}b_0}\\
\textstyle{-(a,b)(b_{\half}y_{\half} )(x_{\half}b_0)
+\half(a,b)(b,y_{\half})x_{\half}b_0+\frac 14(1-(a,b))(a,b)(b,x_{\half})y_{\half}}\\
\textstyle{+\frac 14(1-(a,b))(b_0,x_{\half}y_{\half})b_{\half}-\frac 14(1-(a,b))(b_0,x_{\half}y_{\half})b_{\half}
 +\half(b_0, x_{\half}y_{\half})b_{\half}}\\=
\\
\textstyle{\half(1-(a,b))(b_0,x_{\half}y_{\half})b_{\half}+(1-(a,b))b_{\half}(y_{\half} (x_{\half}b_0))}\\\iff
\\
\textstyle{-(a,b))b_{\half}((x_{\half}b_0)y_{\half})+\frac 18(1-(a,b)(a,b)(b,x_{\half})y_{\half}-\frac 14(a,b)(b,x_{\half})y_{\half}b_0}\\
\textstyle{-(a,b)(b_{\half}y_{\half} )(x_{\half}b_0) +\half(a,b)(b,y_{\half})x_{\half}b_0+\half(a,b)(b_0, x_{\half}y_{\half})b_{\half}}\\=0. 
\end{gather*}
\end{proof}

\section{The proof that $b_{\half}(x_0y_0)=(b_{\half}x_0)y_0+(b_{\half}y_0)x_0$}
%%%%%%%%%%%%%%%%%%%%%%%%%%%%%%%%%%%%%%%%%
%%%%%%%%%%%%%%%%%%%%%%%%%%%%%%%%%%%%%%%%%%
%%%%%%%%%%%%%%%%%%%%%%%%%%%%%%%%%%%%%%%%%%

In this section we prove.
%%%%%%%%%%%%%%%%%%%%%%%%%%%%%%%%%%%%%%%%
\begin{thm}\label{0012}
%%%%%%%%%%%%%%%%%%%%%%%%%%%%%%%%
If $(a,b)\ne \frac 14,$ then
$b_{\half}(x_0y_0)=(b_{\half}x_0)y_0+x_0(b_{\half}y_0).$
\end{thm}
\begin{proof}
We first show the following. 
\begin{equation}\label{eq0012a}
\begin{aligned}
&2(a,b)\Big((b_{\half}x_0)y_0+(b_{\half}y_0)x_0-b_{\half}(x_0y_0)\Big)\\
&+(b,x_0)b_{\half}y_0+4((b_{\half}x_0)y_0)b_0-4(b_{\half}x_0)(y_0b_0)=0.
%I corrected here 14.9.2025 I added +
\end{aligned}
\end{equation}

\noindent
{\bf Proof of \eqref{eq0012a}:} We replace $x_{\half}$ with $b_{\half}x_0$ in Lemma \ref{P120}(1), using $(b_{\half}x_0)b_0=\frac 14(b,x_0)b_{\half}$ and $2b_{\half}(b_{\half}x_0)=(a,b)(b,x_0)a+(a,b)x_0b_0,$ and $b_{\half}(y_0b_0)=\half(1-(a,b))b_{\half}y_0+\frac 14(b,y_0)b_{\half}$ (Theorem \ref{SI}). We get
\medskip

 \begin{gather*}
0=\\
(2(a,b)-1)(b_{\half}x_0)(y_0b_0)
+4((b_{\half}x_0)b_0)(y_0b_0)+4((b_{\half}x_0)b_{\half})( b_{\half}y_0)
-((b_{\half}x_0)b_0)y_0\\
\textstyle{-\half(b,y_0)(a,b)b_{\half}x_0-(b,y_0)(b_{\half}x_0)b_0-(b,b_{\half}x_0)b_{\half}y_0 
-(b,(b_{\half}x_0)y_0)b_{\half}+((b_{\half}x_0)y_0)b_0}\\=
\\
(2(a,b)-1)(b_{\half}x_0)(y_0b_0)
+(b,x_0)b_{\half}(y_0b_0)+(a,b)(b,x_0) b_{\half}y_0+2(a,b)(b_{\half}y_0)(x_0b_0)\\
\textstyle{-\frac 14 (b,x_0)b_{\half}y_0
-\half(b,y_0)(a,b)b_{\half}x_0-\frac 14 (b,x_0)(b,y_0)b_{\half}-(a,b)(b,x_0)b_{\half}y_0 
-(b,(b_{\half}x_0)y_0)b_{\half}}\\
+((b_{\half}x_0)y_0)b_0
\end{gather*}
\begin{gather*}
=\\
\textstyle{(2(a,b)-1)(b_{\half}x_0)(y_0b_0)
+(b,x_0)\Big(\half (1-(a,b))b_{\half}y_0+\frac 14 (b,y_0)b_{\half}\Big)+2(a,b)(b_{\half}y_0)(x_0b_0)}\\
\textstyle{-\frac 14 (b,x_0)b_{\half}y_0
-\half(b,y_0)(a,b)b_{\half}x_0-\frac 14 (b,x_0)(b,y_0)b_{\half}
-(b,(b_{\half}x_0)y_0)b_{\half}}\\
+((b_{\half}x_0)y_0)b_0\\\iff
\\
\textstyle{2(a,b)(b_{\half}x_0)(y_0b_0)+2(a,b)(b_{\half}y_0)(x_0b_0)-\half (a,b)(b,y_0)b_{\half}x_0-\half (a,b) (b,x_0)b_{\half}y_0}\\
\textstyle{-\half (a,b)(x_0b_0,y_0)b_{\half}}\\
\textstyle{+\frac 14 (b,x_0)b_{\half}y_0+((b_{\half}x_0)y_0)b_0-(b_{\half}x_0)(y_0b_0)=0}\\\iff'
\\ 
(a,b)\Big(4(b_{\half}x_0)(y_0b_0)+4(b_{\half}y_0)(x_0b_0)-(b,y_0)b_{\half}x_0-(b,x_0)b_{\half}y_0-(b,x_0y_0)b_{\half}\Big)\\
\textstyle{+\frac 12 (b,x_0)b_{\half}y_0+2((b_{\half}x_0)y_0)b_0-2(b_{\half}x_0)(y_0b_0)=0.}
\end{gather*}
Recall now that by Lemma \ref{P00}(\ref{P002}),
\begin{gather*}
4(b_{\half}y_0 )(x_0b_0)+4(b_{\half}x_0 )(y_0b_0)
-(b,y_0)b_{\half}x_0-(b,x_0)b_{\half}y_0
-(b,x_0y_0)b_{\half}\\
= (b_{\half}x_0 )y_0+(b_{\half}y_0 )x_0-b_{\half}(x_0y_0). 
\end{gather*}
Plugging this into the last identity gives \eqref{eq0012a}.\qedhere

We now complete the proof of the theorem.  Interchanging $x_0, y_0$ in \eqref{eq0012a} and adding, we get

\begin{gather*}
0=\\
4(a,b)\Big((b_{\half}x_0)y_0+(b_{\half}y_0)x_0-b_{\half}(x_0y_0)\Big)\\
+(b,x_0)b_{\half}y_0+4((b_{\half}x_0)y_0)b_0-4(b_{\half}x_0)(y_0b_0)\\
+(b,y_0)b_{\half}x_0+4((b_{\half}y_0)x_0)b_0-4(b_{\half}y_0)(x_0b_0)\\=
\\
4(a,b)\Big((b_{\half}x_0)y_0+(b_{\half}y_0)x_0-b_{\half}(x_0y_0)\Big)\\
+4((b_{\half}x_0)y_0)b_0+4((b_{\half}y_0)x_0)b_0\\
-4(b_{\half}x_0)(y_0b_0)-4(b_{\half}y_0)(x_0b_0)+(b,y_0)b_{\half}x_0+(b,x_0)b_{\half}y_0
\end{gather*}
\begin{gather*}
=\\
\\\text{ (using Lemma \ref{P00}(\ref{P003}), and Lemma \ref{P00}(2))}\\
4(a,b)\Big((b_{\half}x_0)y_0+(b_{\half}y_0)x_0-b_{\half}(x_0y_0)\Big)\\
+(b,x_0y_0)b_{\half}\\
+\Big(b_{\half}x_0)y_0-(b_{\half}y_0)x_0-b_{\half}(x_0y_0\Big)-(b,x_0y_0)b_{\half},
\end{gather*}
which completes the proof of the theorem.
\end{proof}

\section{The proof of the identity\\ $
b_{\half}(x_{\half}y_{\half})+(b_{\half}x_{\half} )y_{\half} +(b_{\half}y_{\half} )x_{\half}=
\half(b,y_{\half} )x_{\half} +\half(b,x_{\half} )y_{\half}+(a,x_{\half}y_{\half})b_{\half}
$}
%%%%%%%%%%%%%%%%%%%%%%%%%%%%%%%%%%%%%%%%
%%%%%%%%%%%%%%%%%%%%%%%%%%%%%%%%%%
%%%%%%%%%%%%%%%%%%%%%%%%%%%%%%%%%

In this section we prove,
%%%%%%%%%%%%%%%%%%%%%%%%%%%%%%%%%%%%%%
\begin{thm}\label{121212}
%%%%%%%%%%%%%%%%%%%%%%%%%%%%%%%%%%
If $(a,b)\notin\{1,\frac 14\},$ then\\
$
b_{\half}(x_{\half}y_{\half})+(b_{\half}x_{\half} )y_{\half} +(b_{\half}y_{\half} )x_{\half}=\\
\half(b,y_{\half} )x_{\half} +\half(b,x_{\half} )y_{\half}+(a,x_{\half}y_{\half})b_{\half}.
$
\end{thm}

We start with,

%%%%%%%%%%%%%%%%%%%%%%%%%%%%%%%%
\begin{lemma}\label{lem121212}$ $
%%%%%%%%%%%%%%%%%%%%%%%%%%%%%%%%
 \begin{enumerate}
 \item 
If $(a,b)\ne\frac 14,$ then
\begin{gather*}
(1-(a,b))(y_{\half}x_0)b_0-2((y_{\half}b_0)x_0)b_0\\
=(1-(a,b)) y_{\half}(x_0b_0)
-2(y_{\half}b_0)(b_0x_0).
\end{gather*}

\item 
If $(a,b)\notin\{1,\frac 14\},$ then
\begin{gather*}
((b_{\half}x_{\half})y_{\half})b_0+(y_{\half}b_0)(b_{\half}x_{\half})\\
\textstyle{=\half(1-(a,b))(b_{\half}x_{\half})y_{\half}-\frac 18(1-(a,b))(b,x_{\half})y_{\half}+\frac 34(b,x_{\half})y_{\half}b_0.}
\end{gather*}

\item If $(a,b)\notin\{1,\frac 14\},$ then
\begin{gather*}
(y_{\half}b_0)(b_{\half}x_{\half})+(x_{\half}b_0)(b_{\half}y_{\half})\\
\textstyle{=\half(1-(a,b))(b_{\half}x_{\half})y_{\half}+\half(1-(a,b))(b_{\half}y_{\half})x_{\half}}\\
\textstyle{-\frac 18(1-(a,b))(b,x_{\half})y_{\half}-\frac 18(1-(a,b))(b,y_{\half})x_{\half}}\\
\textstyle{\frac 14\Big[(b,y_{\half})x_{\half}b_0+(b,x_{\half})y_{\half}b_0\Big]}\\
\textstyle{+\frac 14(b,x_{\half}y_{\half})b_{\half}-\frac 14(a,x_{\half}y_{\half})b_{\half}.}
\end{gather*}
\end{enumerate}
\end{lemma}
\begin{proof}
(1)  By Lemma \ref{012}(2) and Theorem \ref{0012}. using also Theorem \ref{SI}(4a), we get,
 \begin{gather*}
\textstyle{\half(a,b)(1-(a,b))y_{\half}x_0+(1-(a,b))(y_{\half}x_0)b_0
-(a,b)(y_{\half}b_0)x_0-2((y_{\half}b_0)x_0)b_0}\\
-2b_{\half}((b_{\half}y_{\half})x_0)\\
=(1-(a,b)) y_{\half}(x_0b_0)
-2(y_{\half}b_0)(b_0x_0)
-2(b_{\half}x_0)(b_{\half}y_{\half})\\\iff
\\
\textstyle{\half(a,b)(1-(a,b))y_{\half}x_0+(1-(a,b))(y_{\half}x_0)b_0
-(a,b)(y_{\half}b_0)x_0-2((y_{\half}b_0)x_0)b_0}\\
-2\Big[b_{\half}\Big(b_{\half}y_{\half}\Big)_0\Big]x_0-2(b_{\half}x_0)\Big(b_{\half}y_{\half}\Big)_0\\
=(1-(a,b)) y_{\half}(x_0b_0)
-2(y_{\half}b_0)(b_0x_0)
-2(b_{\half}x_0)(b_{\half}y_{\half})\\\iff
\\
\textstyle{\half(a,b)(1-(a,b))y_{\half}x_0+(1-(a,b))(y_{\half}x_0)b_0
-(a,b)(y_{\half}b_0)x_0-2((y_{\half}b_0)x_0)b_0}\\
\textstyle{-2\Big[\frac 14(a,b)(1-(a,b))y_{\half}-\half(a,b)y_{\half}b_0+\frac 14(b,y_{\half})b_{\half}\Big]x_0+\half(b,y_{\half})b_{\half}x_0}\\
=(1-(a,b)) y_{\half}(x_0b_0)
-2(y_{\half}b_0)(b_0x_0)\\\iff
\\
(1-(a,b))(y_{\half}x_0)b_0
-(a,b)(y_{\half}b_0)x_0-2((y_{\half}b_0)x_0)b_0\\
+(a,b)x_0(y_{\half}b_0)\\
=(1-(a,b)) y_{\half}(x_0b_0)
-2(y_{\half}b_0)(b_0x_0)\\\iff
\\
(1-(a,b))(y_{\half}x_0)b_0-2((y_{\half}b_0)x_0)b_0\\
=(1-(a,b)) y_{\half}(x_0b_0)
-2(y_{\half}b_0)(b_0x_0).
\end{gather*}
\medskip

\noindent
(2) Taking $x_0=\Big(b_{\half}x_{\half}\Big)_0$ in (1), and using $\Big(b_{\half}x_{\half}\Big)_0z=(b_{\half}x_{\half})z-\frac 14(b,x_{\half})z,$ for $z\in \{y_{\half}, y_{\half}b_0\},$
%\begin{equation*}
%\begin{aligned}
%&\textstyle{\Big[\Big(b_{\half}x_{\half}\Big)_0y_{\half}\Big]b_0= \Big[(b_{\half}x_{\half})y_{\half}\Big]b_0-\frac 14(b,x_{\half})y_{\half}b_0,}\\
%&\textstyle{\Big(b_{\half}x_{\half}\Big)_0z=%(b_{\half}x_{\half})z-\frac 14(b,x_{\half}),\ \text{for $z\in \{y_{\half}, y_{\half}b_0\}$}.}
%\end{aligned}
%\end{equation*}
and also Lemma \ref{P1212}(5), we get
\begin{gather*}
(1-(a,b))\Big[\Big(b_{\half}x_{\half}\Big)_0y_{\half}\Big]b_0-2\Big[(y_{\half}b_0)\Big(b_{\half}x_{\half}\Big)_0\Big]b_0\\
=\textstyle{\Big[(1-(a,b)) y_{\half}-2y_{\half}b_0\Big]\Big[\half(1-(a,b))\Big(b_{\half}x_{\half}\Big)_0+\frac 14(b,x_{\half})b_0\Big]}
\\\iff
\\
\textstyle{(1-(a,b))\Big[(b_{\half}x_{\half})y_{\half}\Big]b_0-\frac 14(1-(a,b))(b,x_{\half})y_{\half}b_0-\frac 14(1-(a,b))(b,x_{\half})y_{\half}b_0}\\
\textstyle{=\half(1-(a,b))^2(b_{\half}x_{\half})y_{\half}-\frac 18(1-(a,b))^2(b,x_{\half})y_{\half}+\frac 14(1-(a,b))(b,x_{\half})y_{\half}b_0}\\
\textstyle{-(1-(a,b))(b_{\half}x_{\half})(y_{\half}b_0) +\frac 14(1-(a,b))(b,x_{\half})y_{\half}b_0}\\
\textstyle{-\frac 14(1-(a,b))(b,x_{\half})y_{\half}b_0}
\end{gather*}
\begin{gather*}
\\\iff
\\
((b_{\half}x_{\half})y_{\half})b_0+(y_{\half}b_0)(b_{\half}x_{\half})\\
\textstyle{=\half(1-(a,b))(b_{\half}x_{\half})y_{\half}-\frac 18(1-(a,b))(b,x_{\half})y_{\half}+\frac 34(b,x_{\half})y_{\half}b_0.}
\end{gather*}
\medskip

\noindent
(3)
Interchanging $x_{\half}, y_{\half}$ in (2) adding and using Equation \eqref{eqP1212}  for $(b_{\half}x_{\half})y_{\half})b_0+(b_{\half}y_{\half})x_{\half})b_0,$ we get

\begin{gather*}
\textstyle{\half\Big[(b,y_{\half})x_{\half}b_0+(b,x_{\half})y_{\half}b_0\Big]-\frac 14(b,x_{\half}y_{\half})b_{\half}+\frac 14(a,x_{\half}y_{\half})b_{\half}}\\
+(y_{\half}b_0)(b_{\half}x_{\half})+(x_{\half}b_0)(b_{\half}y_{\half})\\
\textstyle{=\half(1-(a,b))\Big[(b_{\half}x_{\half})y_{\half}+(b_{\half}y_{\half})x_{\half}\Big]}\\
\textstyle{-\frac 18(1-(a,b))\Big[(b,x_{\half})y_{\half}+(b,y_{\half})x_{\half}\Big]}\\
\textstyle{+\frac 34(b,x_{\half})y_{\half}b_0+\frac 34(b,y_{\half})x_{\half}b_0}\\\iff
\\
(y_{\half}b_0)(b_{\half}x_{\half})+(x_{\half}b_0)(b_{\half}y_{\half})\\
\textstyle{=\half(1-(a,b))(b_{\half}x_{\half})y_{\half}+\half(1-(a,b))(b_{\half}y_{\half})x_{\half}}\\
\textstyle{-\frac 18(1-(a,b))(b,x_{\half})y_{\half}-\frac 18(1-(a,b))(b,y_{\half})x_{\half}}\\
\textstyle{+\frac 14\Big[(b,y_{\half})x_{\half}b_0+(b,x_{\half})y_{\half}b_0\Big]}\\
\textstyle{+\frac 14(b,x_{\half}y_{\half})b_{\half}-\frac 14(a,x_{\half}y_{\half})b_{\half}.}\qedhere
\end{gather*}
\end{proof}
\begin{proof}[{\bf Proof of Theorem \ref{121212}} ]
We replace in Lemma \ref{P1212}(2), $4(b_{\half}x_{\half} )(y_{\half}b_0) +4(b_{\half}y_{\half})(x_{\half}b_0),$ by the formula of Lemma \ref{lem121212}(3) to get
 \begin{gather*}
 0=\\
(2(a,b)-1)\Big[(b_{\half}x_{\half} )y_{\half} +(b_{\half}y_{\half} )x_{\half}  \Big]\\
+2(1-(a,b))\Big[(b_{\half}x_{\half})y_{\half}+(b_{\half}y_{\half})x_{\half}\Big]\\
\textstyle{-\half(1-(a,b))\Big[(b,x_{\half})y_{\half}+(b,y_{\half})x_{\half}\Big]}\\
(b,y_{\half})x_{\half}b_0+(b,x_{\half})y_{\half}b_0\\
+(b,x_{\half}y_{\half})b_{\half}-(a,x_{\half}y_{\half})b_{\half}\\
\textstyle{-\half(a,b)\Big[(b,y_{\half})x_{\half}+(b,x_{\half})y_{\half}\Big]-(b,y_{\half})x_{\half}b_0-(b,x_{\half})y_{\half}b_0}\\
 -(b,x_{\half}y_{\half})b_{\half}+(x_{\half}y_{\half})b_{\half}
 \end{gather*}
 \begin{gather*}
 \\=
 \\
 b_{\half}(x_{\half}y_{\half}) +(b_{\half}x_{\half} )y_{\half} +(b_{\half}y_{\half} )x_{\half} \\
\textstyle{-\half\Big[b,x_{\half})y_{\half}+(b,y_{\half})x_{\half}\Big]-(a,x_{\half}y_{\half})b_{\half}.}\qedhere
 \end{gather*}
\end{proof}

%%%%%%%%%%%%%%%%%%%%%%%%%%%%%%%%%%%%%%%%%
%%%%%%%%%%%%%%%%%%%%%%%%%%%%%%%%%%%%%%%%%
%%%%%%%%%%%%%%%%%%%%%%%%%%%%%%%%%%%%%%
\section{The proof of the identity\\ $\Big(b_{\half}(x_{\half}y_0)\Big)_0-(b_{\half}x_{\half})y_0+\Big((b_{\half}y_0)x_{\half}\Big)_0=0.$}
%%%%%%%%%%%%%%%%%%%%%%%%%%%%%%%%%%%%%%
%%%%%%%%%%%%%%%%%%%%%%%%%%%%%%%%%%%%%%%
%%%%%%%%%%%%%%%%%%%%%%%%%%%%%%%%%%%%%%
The purpose of this section is to prove:

%%%%%%%%%%%%%%%%%%%%%%%%%%%%%
\begin{thm}\label{thm120}
%%%%%%%%%%%%%%%%%%%%%%%%%%%%%%%%%
If $(a,b)\notin\{1,\frac 14\},$ then\\
$\Big(b_{\half}(x_{\half}y_0)\Big)_0-(b_{\half}x_{\half})y_0+\Big((b_{\half}y_0)x_{\half}\Big)_0=0.$
\end{thm}

 %%%%%%%%%%%%%%%%%%%%%%%%%%%%%%%%%%%%%%%%%%%
\begin{prop}\label{importantprop1}
%%%%%%%%%%%%%%%%%%%%%%%%%%%%%%%%%%%%%%%%%%%%
\begin{equation}\label{importanteq}
\begin{aligned}
&\Big(b_{\half}(x_{\half}y_0)\Big)_0-(b_{\half}x_{\half})y_0+\Big((b_{\half}y_0)x_{\half}\Big)_0=\\
&2\Big[b_{\half}(x_{\half}y_0)-(b_{\half}x_{\half})y_0+(b_{\half}y_0)x_{\half}\Big]b_0+ \\
&\textstyle{2\Big(b_{\half}(x_{\half}(y_0b_0))\Big)_0-2\Big(b_{\half}((x_{\half}b_0)y_0)\Big)_0  -\half(b,x_{\half}y_0)b_0.}    
\end{aligned}
\end{equation}
\end{prop}
\begin{proof}
Recall that by Lemma \ref{012}(3),
\begin{equation}\label{120a}
\begin{aligned}
&(1-(a,b))b_{\half}(x_{\half}y_0)-(1-(a,b))(b_{\half}y_0)x_{\half}+2(b_{\half}y_0)(x_{\half}b_0)\\
&+2(b_{\half}x_{\half})(y_0b_0 )-2((b_{\half}x_{\half})y_0)b_0
-2b_{\half}((x_{\half}b_0)y_0)\\
&= 0.
\end{aligned}
\end{equation}
By Lemma \ref{120}(3),
\begin{equation}\label{120b}
\begin{aligned}
&b_{\half}(x_{\half}y_0)+(a,b)(b_{\half}y_0)x_{\half}-(b_{\half}x_{\half})y_0+2(x_{\half}b_0)(b_{\half}y_0)\\
&+2(b_{\half}x_{\half})(y_0b_0)-2b_{\half}(x_{\half}(y_0b_0) )\\
&-2(a,b)a(x_{\half}(b_{\half}y_0))-2((b_{\half}y_0)x_{\half} )b_0\\
&= 0.
\end{aligned}
\end{equation}
Using $2(b_{\half}(x_{\half}y_0))b_0-(1-(a,b))\Big(b_{\half}(x_{\half}y_0)\Big)_0-\half(b,x_{\half}y_0)b_0=0,$ from Theorem \ref{ISI}(6), we compute \eqref{120a} minus \eqref{120b} to get,
\begin{gather*}
0=\\
-(a,b)\Big(b_{\half}(x_{\half}y_0)\Big)_0-\Big((b_{\half}y_0)x_{\half}\Big)_0+(b_{\half}x_{\half})y_0\\
-2((b_{\half}x_{\half})y_0)b_0-2\Big(b_{\half}((x_{\half}b_0)y_0)\Big)_0\\
+2\Big(b_{\half}(x_{\half}(y_0b_0))\Big)_0+2((b_{\half}y_0)x_{\half} )b_0\\=
\\
-(a,b)\Big(b_{\half}(x_{\half}y_0)\Big)_0-\Big((b_{\half}y_0)x_{\half}\Big)_0+(b_{\half}x_{\half})y_0\\
-2((b_{\half}x_{\half})y_0)b_0+2((b_{\half}y_0)x_{\half} )b_0 \\
+2\Big(b_{\half}(x_{\half}(y_0b_0))\Big)_0-2\Big(b_{\half}((x_{\half}b_0)y_0)\Big)_0\\
+2(b_{\half}(x_{\half}y_0))b_0-(1-(a,b))\Big(b_{\half}(x_{\half}y_0)\Big)_0-\half(b,x_{\half}y_0)b_0\\=
\\
-\Big(b_{\half}(x_{\half}y_0)\Big)_0-\Big((b_{\half}y_0)x_{\half}\Big)_0+(b_{\half}x_{\half})y_0\\
+2(b_{\half}(x_{\half}y_0))b_0-2((b_{\half}x_{\half})y_0)b_0+2((b_{\half}y_0)x_{\half} )b_0 \\
\textstyle{+2\Big(b_{\half}(x_{\half}(y_0b_0))\Big)_0-2\Big(b_{\half}((x_{\half}b_0)y_0)\Big)_0  -\half(b,x_{\half}y_0)b_0.}\qedhere
\end{gather*}
\end{proof}
 
%%%%%%%%%%%%%%%%%%%%%%%%%%
\begin{comment}
%%%%%%%%%%%%%%%%%%%%%%%%%%%%
adding \eqref{120a} and \eqref{120b} we get
\begin{gather*}
0=\\
(2-(a,b))b_{\half}(x_{\half}y_0)+(2(a,b))-1)b_{\half}y_0)x_{\half}\\
+4((x_{\half}b_0)(b_{\half}y_0)+4(b_{\half}x_{\half})(y_0b_0)\\
-2((b_{\half}x_{\half})y_0)b_0-2b_{\half}((x_{\half}b_0)y_0)\\ -2(b_{\half}(x_{\half}(y_0b_0) )-2((b_{\half}y_0)x_{\half})b_0
\end{gather*}
Subtracting \eqref{120c} we get
\begin{gather*}
(1-(a,b))b_{\half}(x_{\half}y_0)+(b_{\half}x_{\half} )y_0\\
+(b,y_0)x_{\half}b_{\half}+(b,x_{\half})y_0b_0+(b,x_{\half}y_0)b_0\\
-2((b_{\half}x_{\half})y_0)b_0-2b_{\half}((x_{\half}b_0)y_0)\\ -2(b_{\half}(x_{\half}(y_0b_0) )-2((b_{\half}y_0)x_{\half})b_0\\\iff
\\
(1-(a,b))b_{\half}(x_{\half}y_0)+(b_{\half}x_{\half} )y_0\\
+(b,y_0)x_{\half}b_{\half}+(b,x_{\half})y_0b_0+(b,x_{\half}y_0)b_0\\
-2((b_{\half}x_{\half})y_0)b_0+2(b_{\half}y_0)(x_{\half}b_0)-\half(b,x_{\half})y_0b_0
\\ -2(b_{\half}(x_{\half}(y_0b_0) )-2((b_{\half}y_0)x_{\half})b_0\\\iff
\\
(1-(a,b))b_{\half}(x_{\half}y_0)+(b_{\half}x_{\half} )y_0\\
+(b,y_0)x_{\half}b_{\half}+\half(b,x_{\half})y_0b_0+(b,x_{\half}y_0)b_0\\
-2((b_{\half}x_{\half})y_0)b_0+2(b_{\half}y_0)(x_{\half}b_0)
\\ -2(b_{\half}(x_{\half}(y_0b_0) )-2((b_{\half}y_0)x_{\half})b_0
\end{gather*}
%%%%%%%%%%%%%%%%%%%%%%%%%%%%%%%%
\end{comment}
%%%%%%%%%%%%%%%%%%%%%%%%%%%%%

%%%%%%%%%%%%%%%%%%%%%%%%%%%%%
\begin{lemma}\label{eq2x120b-120c}
\begin{equation}\label{eq2x120b-120c}
\begin{aligned}
&\Big(b_{\half}(x_{\half}y_0)\Big)_0+\Big((b_{\half}y_0)x_{\half}\Big)_0-(b_{\half}x_{\half} )y_0=\\
&4\Big((b_{\half}(x_{\half}(y_0b_0))\Big)_0+4((b_{\half}y_0)x_{\half} )b_0\\
&-(b,y_0)\Big(b_{\half}x_{\half}\Big)_0 -(b,x_{\half})y_0b_0-(b,x_{\half}y_0)b_0.
\end{aligned}
\end{equation}
\end{lemma}
\begin{proof}
By Lemma \ref{P120}(2),
\begin{equation}\label{120c} 
\begin{aligned}
&0=\\
&b_{\half}(x_{\half}y_0) +(2(a,b)-1) (b_{\half}y_0)x_{\half}-(b_{\half}x_{\half} )y_0+4(x_{\half}b_0)(b_{\half}y_0) \\
&+4 (b_{\half}x_{\half} )(y_0b_0)\\
&-(b,y_0)x_{\half}b_{\half}-(b,x_{\half})y_0b_0-(b,x_{\half}y_0)(a,b)a-(b,x_{\half}y_0)b_0.
\end{aligned}
\end{equation}
Computing $2$ times \eqref{120b} minus \eqref{120c} we get
\begin{equation*}%\label{2x120b-120c}
\begin{aligned}
&\Big(b_{\half}(x_{\half}y_0)\Big)_0+\Big((b_{\half}y_0)x_{\half}\Big)_0-(b_{\half}x_{\half} )y_0\\
&-4\Big((b_{\half}(x_{\half}(y_0b_0))\Big)_0-4((b_{\half}y_0)x_{\half} )b_0\\
&+(b,y_0)\Big(b_{\half}x_{\half}\Big)_0 +(b,x_{\half})y_0b_0+(b,x_{\half}y_0)b_0=0.\qedhere
\end{aligned}
\end{equation*}
\end{proof}

%%%%%%%%%%%%%%%%%%%%%%%%%%%%%%%%%
\begin{prop}\label{prop120a}$ $
%%%%%%%%%%%%%%%%%%%%%%%%%%%%%%%%%
\begin{enumerate}
\item 
$\half(1-(a,b))\Big[\Big(b_{\half}(x_{\half}y_0)\Big)_0+\Big((b_{\half}y_0)x_{\half}\Big)_0-(b_{\half}x_{\half} )y_0\Big]=$\\
$\Big(b_{\half}(x_{\half}(y_0b_0))\Big)+\Big((b_{\half}(y_0b_0))x_{\half}\Big)_0-(b_{\half}x_{\half} )(y_0b_0).$
%\end{gather*}
\item 
%\begin{gather*}
$\Big(b_{\half}(x_{\half}(y_0b_0))\Big)_0 -(b_{\half}x_{\half} )(y_0b_0)=$\\
$\half(1-(a,b))\Big[\Big(b_{\half}(x_{\half}y_0)\Big)_0-(b_{\half}x_{\half} )y_0\Big]-\frac 14(b,y_0)\Big(b_{\half}x_{\half}\Big)_0.$
%\end{gather*}
\end{enumerate}
\end{prop}
\begin{proof}
We put $y_0b_0$ in place of $y_0$ in \eqref{eq2x120b-120c}, noticing that by Theorem \ref{ISI}, if we denote by $z$ the element after the equality sign in \eqref{eq2x120b-120c}, then replacing $y_0$ by $y_0b_0$ in $z$ gives 
\begin{gather*}
\textstyle{\half(1-(a,b)) z+2(b,y_0)\Big(b_{\half}(x_{\half}b_0)\Big)_0+(b,y_0)(b_{\half}x_{\half})b_0}\\
\textstyle{-\half(1-(a,b))(b,y_0)\Big(b_{\half}x_{\half}\Big)_0 -\half(b,x_{\half})(b,y_0)b_0-\frac 14(b,y_0)(b,x_{\half})b_0.}
\end{gather*}
Hence
\begin{gather*}
\Big(b_{\half}(x_{\half}(y_0b_0))\Big)_0+\Big((b_{\half}(y_0b_0))x_{\half}-(b_{\half}x_{\half} )(y_0b_0)\Big)_0=\\
\textstyle{\half(1-(a,b))\Big[\Big(b_{\half}(x_{\half}y_0)\Big)_0+\Big((b_{\half}y_0)x_{\half}\Big)_0-(b_{\half}x_{\half} )y_0\Big]}\\
+2(b,y_0)\Big(b_{\half}(x_{\half}b_0)\Big)_0+(b,y_0)(b_{\half}x_{\half})b_0\\
\textstyle{-\half(1-(a,b))(b,y_0)\Big(b_{\half}x_{\half}\Big)_0 -\half(b,x_{\half})(b,y_0)b_0-\frac 14(b,y_0)(b,x_{\half})b_0}\\\iff
\\
\Big(b_{\half}(x_{\half}(y_0b_0))\Big)_0+\Big((b_{\half}(y_0b_0))x_{\half}\Big)_0-(b_{\half}x_{\half} )(y_0b_0)=\\
\textstyle{\half(1-(a,b))\Big[\Big(b_{\half}(x_{\half}y_0)\Big)_0+\Big((b_{\half}y_0)x_{\half}\Big)_0-(b_{\half}x_{\half} )y_0\Big]}\\\iff
\\
\Big(b_{\half}(x_{\half}(y_0b_0))\Big)_0 -(b_{\half}x_{\half} )(y_0b_0)\\=
\\
\textstyle{\half(1-(a,b))\Big[\Big(b_{\half}(x_{\half}y_0)\Big)_0-(b_{\half}x_{\half} )y_0\Big]-\frac 14(b,y_0)\Big(b_{\half}x_{\half}\Big)_0.}\qedhere
\end{gather*}
\end{proof}

%%%%%%%%%%%%%%%%%%%%%%%%%%%%%%%%%%%%%%%%%%%%
\begin{lemma}\label{important}$ $
%%%%%%%%%%%%%%%%%%%%%%%%%%%%%%%%%%%%%%%%%%
\begin{enumerate}
\item 
$\Big(b_{\half}((x_{\half}b_0)y_0)\Big)_0 + \Big((b_{\half}y_0)(x_{\half}b_0)\Big)_0-\frac 14(b,x_{\half})y_0b_0=0.$

\item 
$\Big(b_{\half}((x_{\half}b_0)(y_0b_0))\Big)_0=\\
-\half(1-(a,b))\Big((b_{\half}y_0)(x_{\half}b_0)\Big)_0+\frac 18(1-(a,b))(b,x_{\half})y_0b_0\\
+\frac{1}{16}(b,y_0)(b,x_{\half})b_0=\\
\half(1-(a,b))\Big(b_{\half}((x_{\half}b_0)y_0)\Big)_0+\frac{1}{16}(b,y_0)(b,x_{\half})b_0.$

\item 
If $(a,b)\notin\{\frac 14, 1\},$ then\\
$\Big(b_{\half}(x_{\half}(y_0b_0))\Big)_0=-\Big((b_{\half}y_0)(x_{\half}b_0)\Big)_0+\frac 14(b,x_{\half})y_0b_0\\
+\frac {1}{4}(b,x_{\half}y_0)b_0\\
=b_{\half}((x_{\half}b_0)y_0)+\frac {1}{4}(b,x_{\half}y_0)b_0.$

%%%%%%%%%%%%%%%%%%%%%%%%%%%%%%%
\begin{comment}
%%%%%%%%%%%%%%%%%%%%%%%%%%%%%%
\item 
Let $y_0=\Big(b_{\half}y_{\half}\Big)_0,$ then\\
$\Big(b_{\half}((x_{\half}b_0)y_0)\Big)_0=\\
\\
-\frac 14(1-(a,b))(a,b)\Big(y_{\half}(x_{\half}b_0)\Big)_0+\half(a,b)\Big((y_{\half}b_0)(x_{\half}b_0)\Big)_0 \\
+\frac 18(b,x_{\half})(1-(a,b))y_0.$

\item
\begin{gather*}
b_{\half}((x_{\half}b_0)\Big(b_{\half}y_{\half}\Big)_0)\\=
\\
-(b_{\half}\Big(b_{\half}y_{\half}\Big)_0)(x_{\half}b_0)-\frac 14(b,x_{\half})\Big(b_{\half}y_{\half}\Big)_0b_0\\=
\\
-\frac 14(a,b)(1-(a,b))y_{\half}(x_{\half}b_0)+\half(a,b)(y_{\half}b_0)(x_{\half}b_0)-\frac 14(b,y_{\half})b_{\half}(x_{\half}b_0)\\
-\frac 18(1-(a,b))\Big(b_{\half}y_{\half}\Big)_0-\frac{1}{16}(b,x_{\half})(b,y_{\half})b_0
\end{gather*}
%%%%%%%%%%%%%%%%%%%%%%%%%%%%%%%%%%
\end{comment}
%%%%%%%%%%%%%%%%%%%%%%%%%%%%%%%%
\end{enumerate}
\end{lemma}
\begin{proof}
(1)
Putting in \eqref{120c} $x_{\half}b_0$ in place of $x_{\half}$ we get
\begin{gather*}
0=\\
\textstyle{\Big(b_{\half}((x_{\half}b_0)y_0)\Big)_0 +(2(a,b)-1) \Big((b_{\half}y_0)(x_{\half}b_0)\Big)_0-\frac 14(b,x_{\half})y_0b_0}\\
+2(1-(a,b))\Big((x_{\half}b_0)(b_{\half}y_0)\Big)_0+(b,x_{\half})(y_0b_0)b_0\\
\textstyle{-\frac 14(b,y_0)(b,x_{\half})b_0-\half(1-(a,b))(b,x_{\half})y_0b_0-\frac 14(b,x_{\half})(b,y_0)b_0}\\=
\\
\textstyle{\Big(b_{\half}((x_{\half}b_0)y_0)\Big)_0 + \Big((b_{\half}y_0)(x_{\half}b_0)\Big)_0-\frac 14(b,x_{\half})y_0b_0} \\
\textstyle{+\half(1-(a,b))(b,x_{\half})y_0b_0+\half(b,y_0)(b,x_{\half})b_0}\\
\textstyle{-\frac 14(b,y_0)(b,x_{\half})b_0-\half(1-(a,b))(b,x_{\half})y_0b_0-\frac 14(b,x_{\half})(b,y_0)b_0}\\=
\\
\textstyle{\Big(b_{\half}((x_{\half}b_0)y_0)\Big)_0 + \Big((b_{\half}y_0)(x_{\half}b_0)\Big)_0-\frac 14(b,x_{\half})y_0b_0.}
\end{gather*}
\medskip

\noindent
(2) Putting in (1) $y_0b_0$ in place of $y_0,$ we get
\begin{gather*}
\textstyle{\Big(b_{\half}((x_{\half}b_0)(y_0b_0))\Big)_0=-\Big((b_{\half}(y_0b_0))(x_{\half}b_0)\Big)_0+\frac 14(b,x_{\half})(y_0b_0)b_0}\\
\textstyle{=-\half(1-(a,b))\Big((b_{\half}y_0)(x_{\half}b_0)\Big)_0-\frac{1}{16}(b,y_0)(b,x_{\half})b_0}\\
\textstyle{+\frac 18(1-(a,b))(b,x_{\half})y_0b_0+\frac 18(b,x_{\half})(b,y_0)b_0.}
\end{gather*}
\medskip

\noindent
(3)
By Lemma \ref{lem121212}(1) and part (2), using Theorem \ref{SI},
\begin{gather*}
(1-(a,b))\Big(b_{\half}((y_{\half}x_0)b_0)\Big)_0-2\Big(b_{\half}(((y_{\half}b_0)x_0)b_0)\Big)_0\\
=(1-(a,b)) \Big(b_{\half}(y_{\half}(x_0b_0))\Big)_0
-2\Big(b_{\half}((y_{\half}b_0)(b_0x_0))\Big)_0\\\iff
\\
\textstyle{\frac 14(1-(a,b))(b,y_{\half}x_0)b_0-\half(b,(y_{\half}b_0)x_0))b_0} \\
=(1-(a,b)) \Big(b_{\half}(y_{\half}(x_0b_0))\Big)_0+(1-(a,b))\Big((b_{\half}x_0)(y_{\half}b_0)\Big)_0\\
\textstyle{-\frac 14(1-(a,b))(b,y_{\half})x_0b_0
-\frac{1}{8}(b,x_0)(b,y_{\half})b_0}
\end{gather*}
\begin{gather*}
\\\iff
\\
\textstyle{\frac {1}{4}(1-(a,b))(b,y_{\half}x_0)b_0-\frac 18(b,y_{\half})(b,x_0)b_0} \\
=(1-(a,b)) \Big(b_{\half}(y_{\half}(x_0b_0))\Big)_0\\
\textstyle{+(1-(a,b))\Big((b_{\half}x_0)(y_{\half}b_0)\Big)_0-\frac 14(1-(a,b))(b,y_{\half})x_0b_0}\\
\textstyle{-\frac{1}{8}(b,x_0)(b,y_{\half})b_0,}
\end{gather*}
so the first equality in (3) follows from by replacing $y_{\half}$ with $x_{\half},$ and $x_0$ with $y_0.$  The second equality follows from this and (1).
%%%%%%%%%%%%%%%%%%%%%%%%%%%%%%%
\begin{comment}
%%%%%%%%%%%%%%%%%%%%%%%%%%%%%%%%%
\medskip

\noindent
(4)
Let $y_0=\Big(b_{\half}y_{\half}\Big)_0,$ so that 
$b_{\half}y_0=\frac 14(1-(a,b))(a,b)y_{\half}-\half(a,b)y_{\half}b_0+\frac 14(b,y_{\half})b_{\half},$ and
$y_0b_0=\half(1-(a,b))y_0+\frac 14(b,y_{\half})b_0.$  By (1),
\begin{gather*}
\Big(b_{\half}((x_{\half}b_0)y_0)\Big)_0= -\Big((b_{\half}y_0)(x_{\half}b_0)\Big)_0+\frac 14(b,x_{\half})y_0b_0\\=
\\
-\frac 14(1-(a,b))(a,b)\Big(y_{\half}(x_{\half}b_0)\Big)_0+\half(a,b)\Big((y_{\half}b_0)(x_{\half}b_0)\Big)_0-\frac 14(b,y_{\half})\Big(b_{\half}(x_{\half}b_0)\Big)_0\\
+\frac 18(b,x_{\half})(1-(a,b))y_0+\frac{1}{16} (b,x_{\half})(b,y_{\half})b_0\\=
\\
-\frac 14(1-(a,b))(a,b)\Big(y_{\half}(x_{\half}b_0)\Big)_0+\half(a,b)\Big((y_{\half}b_0)(x_{\half}b_0)\Big)_0 \\
+\frac 18(b,x_{\half})(1-(a,b))y_0 
\end{gather*}
%%%%%%%%%%%%%%%%%%%%%%%%%%%%%%%
\end{comment}
%%%%%%%%%%%%%%%%%%%%%%%%%%%%%%%%
\end{proof}

%%%%%%%%%%%%%%%%%%%%%%%%%%%%%%%%%%
\begin{cor}\label{cor120}$ $
%%%%%%%%%%%%%%%%%%%%%%%%%%%%%%%%%
\begin{enumerate}
%\item 
%\begin{gather*}
%0=\\
%(2(a,b)-1)\Big[\Big(b_{\half}(x_{\half}y_0)\Big)_0-(b_{\half}x_{\half} )y_0+\Big((b_{\half}y_0)x_{\half}\Big)_0\Big]\\
%+4\Big((x_{\half}b_0)(b_{\half}y_0)\Big)_0+4\Big(b_{\half}(x_{\half}(y_0b_0))\Big)_0\\
% -(b,x_{\half})y_0b_0-(b,x_{\half}y_0)b_0. 
%\end{gather*}

\item 
%\begin{gather*}
$(2(a,b)-1)\Big[\Big(b_{\half}(x_{\half}y_0)\Big)_0-(b_{\half}x_{\half} )y_0+\Big((b_{\half}y_0)x_{\half}\Big)_0\Big]$\\
$-4b_{\half}((x_{\half}b_0)y_0)+4\Big(b_{\half}(x_{\half}(y_0b_0))\Big)_0-(b,x_{\half}y_0)b_0=0. $                
%\end{gather*}

\item 
If $(a,b)\notin\{\frac 14, 1\},$ then\\
$(2(a,b)-1)\Big[\Big(b_{\half}(x_{\half}y_0)\Big)_0-(b_{\half}x_{\half} )y_0+\Big((b_{\half}y_0)x_{\half}\Big)_0\Big]=0$
\end{enumerate}
\end{cor}
\begin{proof}
By Proposition \ref{prop120a}(2),
%\begin{gather*}
 $4(b_{\half}x_{\half} )(y_0b_0)-(b,y_0)\Big(b_{\half}x_{\half}\Big)_0\quad =\\
4\Big(b_{\half}(x_{\half}(y_0b_0))\Big)_0+2((a,b)-1)\Big[\Big(b_{\half}(x_{\half}y_0)\Big)_0-(b_{\half}x_{\half} )y_0\Big].$
%\end{gather*}

Hence by \eqref{120c}, and since $\Big(b_{\half}((x_{\half}b_0)y_0)\Big)_0 + \Big((b_{\half}y_0)(x_{\half}b_0)\Big)_0-\frac 14(b,x_{\half})y_0b_0=0,$ by Lemma \ref{important}(1),
\begin{gather*}
0=\\
\Big(b_{\half}(x_{\half}y_0)\Big)_0 +(2(a,b)-1) \Big((b_{\half}y_0)x_{\half}\Big)_0-(b_{\half}x_{\half} )y_0+4\Big((x_{\half}b_0)(b_{\half}y_0)\Big)_0 \\
+4 (b_{\half}x_{\half} )(y_0b_0)-(b,y_0)\Big(b_{\half}x_{\half}\Big)_0\\
  -(b,x_{\half})y_0b_0-(b,x_{\half}y_0)b_0\\=
  \\
(2(a,b)-1)\Big[\Big(b_{\half}(x_{\half}y_0)\Big)_0-(b_{\half}x_{\half} )y_0+\Big((b_{\half}y_0)x_{\half}\Big)_0\Big]\\
+4\Big((x_{\half}b_0)(b_{\half}y_0)\Big)_0+4\Big(b_{\half}(x_{\half}(y_0b_0))\Big)_0\\
 -(b,x_{\half})y_0b_0-(b,x_{\half}y_0)b_0\\=
 \\
 (2(a,b)-1)\Big[\Big(b_{\half}(x_{\half}y_0)\Big)_0-(b_{\half}x_{\half} )y_0+\Big((b_{\half}y_0)x_{\half}\Big)_0\Big]\\
-4b_{\half}((x_{\half}b_0)y_0)+4\Big(b_{\half}(x_{\half}(y_0b_0))\Big)_0-(b,x_{\half}y_0)b_0.
\end{gather*}
%
%
%Using Proposition \ref{prop120b} we get
%\begin{gather*}
%0=\\
%-\Big[\Big(b_{\half}(x_{\half}y_0)\Big)_0-(b_{\half}x_{\half} )y_0+\Big((b_{\half}y_0)x_{\half}\Big)_0\Big]\\
%-4((x_{\half}b_0)(b_{\half}y_0) +4((b_{\half}y_0)x_{\half} )b_0-(b,y_0)\Big(b_{\half}x_{\half}\Big)_0\\
%+4\Big((x_{\half}b_0)(b_{\half}y_0)\Big)_0+4\Big(b_{\half}(x_{\half}(y_0b_0))\Big)_0\\
% -(b,x_{\half})y_0b_0-(b,x_{\half}y_0)b_0.                 
%\end{gather*}
%\medskip
%
%\noindent
%(2)
%Since $b_{\half}((x_{\half}b_0)y_0) + (b_{\half}y_0)(x_{\half}b_0)-\frac 14(b,x_{\half})y_0b_0=0,$
%(1) implies (2).
\medskip

\noindent
(2)  Follows from (1) and Lemma \ref{important}(3).
\end{proof}

%%%%%%%%%%%%%%%%%%%%%%%%%%%%
\begin{cor}\label{cor120a}
%%%%%%%%%%%%%%%%%%%%%%%%%%%%%
$m=\half(1-(a,b))\Big((b_{\half}y_0)(x_{\half}b_0)\Big)_0+\frac{1}{16}(b,x_{\half})(b,y_0)b_0.$
\end{cor}
\begin{proof}
By Lemma \ref{important}(1), $\Big((b_{\half}y_0)(x_{\half}b_0)\Big)_0=-\Big(b_{\half}((x_{\half}b_0)y_0)\Big)_0+\frac 14(b,x_{\half})y_0b_0.$  Note that by Theorem \ref{SI}, for $z:=-\Big(b_{\half}((x_{\half}b_0)y_0)\Big)_0+\frac 14(b,x_{\half})y_0b_0,$ we have
\begin{gather*}
\textstyle{zb_0=\half(1-(a,b))z-\frac 14(b,(x_{\half}b_0)y_0)b_0+\frac 18(b,x_{\half})(b,y_0)b_0}\\
\textstyle{=\half(1-(a,b))z+\frac{1}{16} (b,x_{\half})(b,y_0)b_0.}
\end{gather*}
But $z=\Big((b_{\half}y_0)(x_{\half}b_0)\Big)_0,$ so the colrollary holds.
\end{proof}

%%%%%%%%%%%%%%%%%%%%%%%%%%%%%%%%%%%
%
\begin{prop}\label{prop120b}$ $
%%%%%%%%%%%%%%%%%%%%%%%%%%%%%%%%%%
\begin{enumerate}
\item
%\begin{gather*}
$(a,b)\Big[\Big(b_{\half}(x_{\half}y_0)\Big)_0-(b_{\half}x_{\half})y_0+\Big((b_{\half}y_0)x_{\half}\Big)_0\Big]$\\
$+2((x_{\half}b_0)(b_{\half}y_0) -2((b_{\half}y_0)x_{\half} )b_0+\half(b,y_0)\Big(b_{\half}x_{\half}\Big)_0=0.$
%\end{gather*}
 
\item
Let $z:=\Big(b_{\half}(x_{\half}y_0)\Big)_0- (b_{\half}x_{\half})y_0+\Big((b_{\half}y_0)x_{\half}\Big)_0.$ If $(a,b)\ne 0,$ then $zb_0=\half(1-(a,b)z.$
\end{enumerate}
\end{prop}
\begin{proof}
(1) By \eqref{120b} and Proposition \ref{prop120a}(2),
\begin{gather*}
0=\\
\Big(b_{\half}(x_{\half}y_0)\Big)_0+(a,b)\Big((b_{\half}y_0)x_{\half}\Big)_0-(b_{\half}x_{\half})y_0+2\Big((x_{\half}b_0)(b_{\half}y_0)\Big)_0\\
\textstyle{-(1-(a,b))\Big[\Big(b_{\half}(x_{\half}y_0)\Big)_0-(b_{\half}x_{\half})y_0\Big]+\half(b,y_0)\Big(b_{\half}x_{\half}\Big)_0.}\\
 -2((b_{\half}y_0)x_{\half} )b_0\\=
 \\
(a,b)\Big[\Big(b_{\half}(x_{\half}y_0)\Big)_0-(b_{\half}x_{\half})y_0+(b_{\half}y_0)x_{\half}\Big]\\
\textstyle{+2((x_{\half}b_0)(b_{\half}y_0) -2((b_{\half}y_0)x_{\half} )b_0+\half(b,y_0)\Big(b_{\half}x_{\half}\Big)_0.}
\end{gather*}
\medskip

\noindent
(2)
By (1), $(a,b)z=-w,$ with $w=2((x_{\half}b_0)(b_{\half}y_0) -2((b_{\half}y_0)x_{\half} )b_0+\half(b,y_0)\Big(b_{\half}x_{\half}\Big)_0.$  However, by Corollary \ref{cor120a}, and by Theorem \ref{ISI},
\begin{gather*}
\textstyle{wb_0=\half(1-(a,b))w+\frac 18(b,x_{\half})(b,y_0)b_0-(b_0,(b_{\half}y_0)x_{\half})b_0+\frac 18(b,y_0)(b,x_{\half})b_0}\\
\textstyle{=\half(1-(a,b))w.}\qedhere
\end{gather*}
\end{proof}
We can now prove Theorem \ref{thm120}.
\medskip

\noindent
\begin{proof}[{\bf Proof of Theorem \ref{thm120}.}]$ $
\medskip

\noindent
By Corollary \ref{cor120}(2), we may assume that $(a,b)=\half.$  

Let $z:=\Big(b_{\half}(x_{\half}y_0)\Big)_0-(b_{\half}x_{\half})y_0+\Big((b_{\half}y_0)x_{\half}\Big)_0.$  By Proposition \ref{prop120b}(2), $zb_0=\half(1-(a,b))z.$ By \eqref{importanteq}, 
\[
\textstyle{z=2zb_0+2b_{\half}(x_{\half}(y_0b_0) )-2b_{\half}((x_{\half}b_0)y_0)  -\half(b,x_{\half}y_0)b_0.} 
\]
But by Lemma \ref{important}(3), $2b_{\half}(x_{\half}(y_0b_0) )-2b_{\half}((x_{\half}b_0)y_0)  -\half(b,x_{\half}y_0)b_0=0.$  Hence, by Proposition \ref{prop120b}(2), $z=2zb_0=(1-(a,b))z.$  It follows that $(a,b)z=0,$ so $z=0.$
\end{proof}

%%%%%%%%%%%%%%%%%%%%%%%%%%%%%%%%%
%%%%%%%%%%%%%%%%%%%%%%%%%%%%%%%
%%%%%%%%%%%%%%%%%%%%%%%%%%%%%%
\section{The case $(a,b)=1$}
%%%%%%%%%%%%%%%%%%%%%%%%%%%%%%%%
%%%%%%%%%%%%%%%%%%%%%%%%%%%%%%%
%%%%%%%%%%%%%%%%%%%%%%%%%%%%%

%%%%%%%%%%%%%%%%%%%%%%%%%%%
In this section we assume that $a,b$ are axes and that $(a,b)=1.$ %We denote by $\frakR$ the radical of the Frobenius form.

\begin{thm}\label{ab12}
\begin{enumerate}
\item 
If $b_0=0,$ then $[L_a,L_b]$ is a derivation.

\item 
%Assume $b_0\ne 0.$ 
%\item[(A)]
%$b_0, b_{\half}\in \frakR.$
Assume $\charc(\ff)\ne 3.$  Then
\begin{itemize}
\item[(i)]
$(x_{\half}y_{\half})b_{\half}+(b_{\half}x_{\half} )y_{\half} +(b_{\half}y_{\half} )x_{\half}=$\\
$\half(b,y_{\half})x_{\half}+\half(b,x_{\half})y_{\half} 
 +(a,x_{\half}y_{\half})b_{\half}.$
 
\item[(ii)] 
$b_{\half}(x_{\half}y_0)+(b_{\half}y_0)x_{\half} 
-(b_{\half}x_{\half})y_0=0.$
\end{itemize}
\end{enumerate}
 In particular, $[L_a,L_b]$ is a derivation.
\end{thm}

\noindent
{\it Proof.}
{\bf (1)} Suppose $b_0=0.$  Then, by Lemma \ref{120}(3), $\Big(b_{\half}(x_{\half}y_0)\Big)_0=(b_{\half}x_{\half})y_0-\Big(x_{\half}(b_{\half}y_0)\Big)_0.$  By Lemma \ref{P1212}(2), $(x_{\half}y_{\half})b_{\half}+(b_{\half}x_{\half} )y_{\half} +(b_{\half}y_{\half} )x_{\half}=
 \half(b,y_{\half})x_{\half}+\half(b,x_{\half})y_{\half} 
 +(a,x_{\half}y_{\half})b_{\half},$ because $(b,x_{\half}y_{\half})=(a,x_{\half}y_{\half}).$  So (1) follows from Theorem \ref{0012}, and Theorem \ref{der}.
 \medskip

\noindent
%Note also that 
%\begin{equation}\label{eqfrakr}
%\begin{aligned}
%&(1)\  \text{If }x\in\frakR,\text{ then }x\in A_0(a)+A_{\half}%(a),\text{ because }(a,x)=0.\\
%&(2)\ (b,x)=(a,x),\text{ for }x\in\ff a+A_0(a),\text{ because }(a,b)=1,  \text{ and }\\
%&b_0\in\frakR.
%\end{aligned}
%\end{equation}
%\medskip
%
%\noindent
{\bf (2)} 
Note that since $(a,b)=1,$  theorem \ref{SI}  implies
\begin{equation}\label{eqSI}
\begin{aligned}
&(1)\ (x_{\half}b_0)b_0=0.\\
& (2)\ \textstyle{(x_0b_0)b_0= \half(b,x_0)b_0.}\\
& (3)\ \textstyle{b_{\half}(b_{\half}x_0)=\half(b,x_0)a+\half x_0b_0.}\\
&(4)\ \textstyle{b_{\half}(b_{\half}x_{\half})=-\half x_{\half}b_0+\frac 12(b,x_{\half})b_{\half}\text{ and}}\\
&\quad\ \ \textstyle{b_{\half}\Big(b_{\half}x_{\half}\Big)_0=-\half x_{\half}b_0+\frac 14(b,x_{\half})b_{\half} .}\\
&(5)\ \textstyle{b_{\half}(x_{\half}b_0)=(b_{\half}x_{\half})b_0=\frac 14(b,x_{\half})b_0.}\\
&(6)\ \textstyle{b_{\half}(x_0b_0)=\frac 14(b,x_0)b_{\half}.}\\
&(7)\ (b_0, (x_{\half}b_0)y)=0=(b_{\half}, yb_0),\ \forall y\in A.\\
&(8)\ b_{\half}b_0=b_0^2=0\ \text{and $b_{\half}^2=b_0,$ by Theorem \ref{Q}.}
\end{aligned}
\end{equation}

%First note that by Lemma \ref{P1212}(7), 
%\begin{equation}\label{x0y0}
%b_{\half}((x_{\half}b_0)(y_{\half}b_0))=0.
%\end{equation}

%%%%%%%%%%%%%%%%%%%%%%%%%%%%%%%
\begin{comment}
%%%%%%%%%%%%%%%%%%%%%%%%%%%%%

We now use Theorem \ref{SI}.  By Theorem \ref{SI}(1), $(x_{\half}b_0)b_0=0,$ and by Theorem \ref{SI}(5), $b_{\half}(x_{\half}b_0)=\frac 14(b,x_{\half})b_0.$  Note also that $(b,x_{\half}b_0)=(b_{\half},x_{\half}b_0)=(b_{\half}b_0,x_{\half})=(0,x_{\half})=0.$
Thus, replacing $x_{\half}$ with $x_{\half}b_0$ in \eqref{eq1P12121} we get,
\begin{align*}
&\Big((x_{\half}b_0)(y_{\half}b_0)\Big)_0+\frac 14(b,y_{\half})(b,x_{\half})b_0\\ 
&=\frac 14(b,x_{\half})(b,y_{\half})b_0-((x_{\half}b_0)y_{\half})b_0+(b,(x_{\half}b_0)y_{\half})b_0,
\end{align*}
so since $(a,(x_{\half}b_0)(y_{\half}b_0))=\half(x_{\half}b_0,y_{\half}b_0)=0,$ we get
\begin{equation}\label{eq112122}
(x_{\half}b_0)(y_{\half}b_0)=-((x_{\half}b_0)y_{\half})b_0+(b,(x_{\half}b_0)y_{\half})b_0.
\end{equation}  
Note now that $(b,\Big((x_{\half}b_0)y_{\half}\Big)_0)=(b_0,(x_{\half}b_0)y_{\half})=((x_{\half}b_0)b_0,y_{\half})=(0,y_{\half})=0.$ Similarly we get $(a,(x_{\half}b_0)(y_{\half}b_0)=0.$ Multiplying \eqref{eq112122} by $b_{\half}$ using Theorem \ref{SI}(7) shows \eqref{x0y0}.
\medskip

%%%%%%%%%%%%%%%%%%%%%%%%%%%%%%%%%%%%%%%%
\end{comment}
%%%%%%%%%%%%%%%%%%%%%%%%%%%%%%%%%%%%

 %%%%%%%%%%%%%%%%%%%%%%%%%%%%%%%%%%%%%%%
 \begin{lemma}\label{lemab11}$ $
%%%%%%%%%%%%%%%%%%%%%%%%%%%%%%%%%%%%
\begin{enumerate}
\item 
$(b_{\half}y_{\half})(x_{\half}b_0)=\half(b,y_{\half})x_{\half}b_0-\frac 14(b,x_{\half})y_{\half}b_0+\half(b_0,x_{\half}y_{\half})b_{\half}-b_{\half}(y_{\half}(x_{\half}b_0)).$

\item
$4(b_{\half}y_{\half})(x_{\half}b_0)+4(b_{\half}x_{\half})(y_{\half}b_0)=\\(b,y_{\half}) x_{\half}b_0+(b,x_{\half}) y_{\half}b_0+4(b_0,x_{\half}y_{\half})b_{\half}-4b_{\half}(y_{\half}(x_{\half}b_0))-4b_{\half}(x_{\half}(y_{\half}b_0)).$

\item 
$((b_{\half}y_{\half})(x_{\half}b_0))b_0=0.$

\item 
$((b_{\half}x_{\half} )y_{\half})b_0 +((b_{\half}y_{\half} )x_{\half})b_0 = 
\half\Big[(b,y_{\half})x_{\half}b_0+(b,x_{\half})y_{\half}b_0\Big]-\frac 14(b_0,x_{\half}y_{\half})b_{\half}.$

\item 
$4(b_{\half}y_{\half})(x_{\half}b_0)+4(b_{\half}x_{\half})(y_{\half}b_0)\quad =
\\
(b,y_{\half})x_{\half}b_0+(b,x_{\half})y_{\half}b_0 -\half (b_0,x_{\half}y_{\half})b_{\half}
+2b_{\half}(x_{\half}(y_{\half}b_0))+2b_{\half}(y_{\half}(x_{\half}b_0)).$

\item 
$6\Big[b_{\half}(y_{\half}(x_{\half}b_0))+b_{\half}(x_{\half}(y_{\half}b_0))\Big]=\frac 92 (b_0,x_{\half}y_{\half})b_{\half}.$

%$4(b_0,x_{\half}y_{\half})b_{\half}-4\Big[b_{\half}(y_{\half}%(x_{\half}b_0))+b_{\half}(x_{\half}(y_{\half}b_0))\Big]=-%\half(b_0,x_{\half}y_{\half})b_{\half}+2\Big[b_{\half}(y_{\half}(x_{\half}b_0))+b_{\half}(x_{\half}(y_{\half}b_0))\Big]$
%Hence
%$\frac 9 2(b_0,x_{\half}y_{\half})b_{\half}=6y\implies \frac 32 x=2y$

%\item 
%$b_{\half}(x_{\half}y_{\half})+(b_{\half}x_{\half} )y_{\half} +(b_{\half}y_{\half} )x_{\half}=\\
%\half(b,y_{\half})x_{\half}+\half(b,x_{\half})y_{\half}+(a,x_{\half}y_{\half})b_{\half}-3(b_0,x_{\half}y_{\half})+4\Big[b_{\half}(y_{\half}(x_{\half}b_0))+b_{\half}(x_{\half}(y_{\half}b_0))\Big].$
\end{enumerate}
 \end{lemma}
\begin{proof} 
(1)
By Lemma \ref{P1212}(2),
\begin{equation}\label{eqab17}
\begin{aligned}
&(b_{\half}x_{\half} )y_{\half} +(b_{\half}y_{\half} )x_{\half}+4(b_{\half}x_{\half} )(y_{\half}b_0) +4(b_{\half}y_{\half})(x_{\half}b_0)\quad =  \\
&\textstyle{\half\Big[(b,y_{\half})x_{\half}+(b,x_{\half})y_{\half}\Big]+(b,y_{\half})x_{\half}b_0+(b,x_{\half})y_{\half}b_0}\\
&+(b,x_{\half}y_{\half})b_{\half}-(x_{\half}y_{\half})b_{\half}.
\end{aligned}
\end{equation}
 Replacing $x_{\half}$ by $x_{\half}b_0$ in \eqref{eqab17} using \eqref{eqSI}, we get

\begin{gather*}%\label{eqab17}
\textstyle{\frac 14(b,x_{\half})y_{\half}b_0 +(b_{\half}y_{\half})(x_{\half}b_0)+0 +0\quad =  }\\
\textstyle{\half\Big[(b,y_{\half})x_{\half}b_0+0\Big]+0+0}\\
\textstyle{+(a,(x_{\half}b_0)y_{\half})b_{\half}-b_{\half}((x_{\half}b_0)y_{\half}),}
\end{gather*}
so (1) holds.
\medskip

\noindent
(2)  This follows by interchanging $x_{\half}$ and $y_{\half}$ in (1) and adding.
\medskip

\noindent
(3)
Multiplying (1) by $b_0$ using \eqref{eqSI} we get $((b_{\half}y_{\half})(x_{\half}b_0))b_0=0-0+0-(b_{\half}(y_{\half}(x_{\half}b_0)))b_0=\frac 14(b_0,(y_{\half}(x_{\half}b_0))b_0=0.$
\medskip

\noindent
(4)
Multiplying \eqref{eqab17} by $b_0$ using (3) and \eqref{eqSI}(6) we get 
\begin{gather*}
((b_{\half}x_{\half} )y_{\half})b_0 +((b_{\half}y_{\half} )x_{\half})b_0 
 +0 +0\quad =  \\
\textstyle{\half\Big[(b,y_{\half})x_{\half}b_0+(b,x_{\half})y_{\half}b_0\Big]+0+0}\\
\textstyle{+0-\frac 14(b_0,x_{\half}y_{\half})b_{\half}.}
\end{gather*}
\medskip

\noindent
(5)
By Lemma \ref{1212}(3b) we have
\begin{gather*}
+4(b_{\half}y_{\half})(x_{\half}b_0)+4(b_{\half}x_{\half})(y_{\half}b_0)\quad =
\\
2b_{\half}(x_{\half}(y_{\half}b_0))+2b_{\half}(y_{\half}(x_{\half}b_0))\\
+2((b_{\half}y_{\half})x_{\half})b_0+2((b_{\half}x_{\half})y_{\half})b_0.
\end{gather*} 
Using (4) we get
\begin{gather*}
+4(b_{\half}y_{\half})(x_{\half}b_0)+4(b_{\half}x_{\half})(y_{\half}b_0)\quad =
\\
2b_{\half}(x_{\half}(y_{\half}b_0))+2b_{\half}(y_{\half}(x_{\half}b_0))\\
\textstyle{+(b,y_{\half})x_{\half}b_0+(b,x_{\half})y_{\half}b_0-\half (b_0,x_{\half}y_{\half})b_{\half}.}
\end{gather*} 
\medskip

\noindent
(6) 
By (2) and (5),
\begin{gather*}
4(b_0,x_{\half}y_{\half})b_{\half}-4\Big[b_{\half}(y_{\half}(x_{\half}b_0))+b_{\half}(x_{\half}(y_{\half}b_0))\Big]=\\\textstyle{-\half(b_0,x_{\half}y_{\half})b_{\half}+2\Big[b_{\half}(y_{\half}(x_{\half}b_0))+b_{\half}(x_{\half}(y_{\half}b_0))\Big],}
\end{gather*}
so (6) holds.
\end{proof}
\begin{proof}[{\bf Proof of Theorem \ref{ab12}(2i)}]$ $
\medskip

\noindent
By  \eqref{eqab17} and Lemma \ref{lemab11}(2),\\
$b_{\half}(x_{\half}y_{\half})+(b_{\half}x_{\half} )y_{\half} +(b_{\half}y_{\half} )x_{\half}=\\
\half(b,y_{\half})x_{\half}+\half(b,x_{\half})y_{\half}+(a,x_{\half}y_{\half})b_{\half}-3(b_0,x_{\half}y_{\half})+4\Big[b_{\half}(y_{\half}(x_{\half}b_0))+b_{\half}(x_{\half}(y_{\half}b_0))\Big],$ so (2i) follows from Lemma \ref{lemab11}(6), since $\charc(\ff)\ne 3.$
\end{proof}
 
For the proof of part (2ii), we need the following three lemmas.
%
%%%%%%%%%%%%%%%%%%%%%%%%%%%%%%%%%%%%%
\begin{lemma}\label{claim1}$ $
%%%%%%%%%%%%%%%%%%%%%%%%%%%%%%%%%%%5
\begin{enumerate}
\item 
$\Big((b_{\half}y_0)(x_{\half}b_0)\Big)_0+2(x_{\half}b_{\half})(y_0b_0)=\\
\frac 34(b,y_0)\Big(x_{\half}b_{\half}\Big)_0+\half(b,x_{\half})y_0b_0
-(x_{\half}(b_{\half}y_0))b_0+\half(b,x_{\half}y_0)b_0.$

\item 
$\Big((x_{\half}b_0)(y_{\half}b_0)\Big)_0+((x_{\half}b_0)y_{\half})b_0-\half(b_0,x_{\half}y_{\half})b_0=0.$

\item 
$(x_{\half}b_0)(b_{\half}y_0)b_0=\frac{1}{16}(b,x_{\half})(b,y_0)b_0.$

\item 
$\Big(b_{\half}(x_{\half}(y_0b_0))\Big)_0+\frac 14(b,y_0)\Big(b_{\half}x_{\half}\Big)_0-(b_{\half}x_{\half})(y_0b_0)=0.$

\item 
%\begin{equation*}
%\begin{aligned} 
$\Big(b_{\half}(x_{\half}y_0)\Big)_0+\Big((b_{\half}y_0)x_{\half}\Big)_0 
-(b_{\half}x_{\half})y_0+ \\
2\Big((x_{\half}b_0)(b_{\half}y_0)\Big)_0
 +\half(b,y_0)\Big(b_{\half}x_{\half}\Big)_0 -2((b_{\half}y_0)x_{\half} )b_0= 0.$
%\end{aligned}
%\end{equation*}

\item 
$\Big(b_{\half}((x_{\half}b_0)y_0)\Big)_0+\Big((b_{\half}y_0)(x_{\half}b_0)\Big)_0 -\frac 14 (b,x_{\half}) y_0b_0= 0.$

\item 
\begin{itemize}
\item[(i)] 
$\Big[b_{\half}(x_{\half}y_0)\Big)_0+\Big((b_{\half}y_0)x_{\half}\Big)_0 
-(b_{\half}x_{\half})y_0\Big]b_0=\\\frac 14(b,x_{\half}y_0)b_0+((b_{\half}y_0)x_{\half})b_0 
-((b_{\half}x_{\half})y_0)b_0=0.$

\item[(ii)] 
$\Big(b_{\half}(x_{\half}y_0)\Big)_0-(b_{\half}x_{\half})y_0+\Big((b_{\half}y_0)x_{\half}\Big)_0=\\
2\Big(b_{\half}(x_{\half}(y_0b_0))\Big)_0-2\Big(b_{\half}((x_{\half}b_0)y_0)\Big)_0  -\half(b,x_{\half}y_0)b_0.$

\item[(iii)] 
$((b_{\half}y_0)x_{\half})b_0-\half(b,x_{\half}y_0)b_0=- \frac 34 (b,x_{\half}y_0)b_0+((b_{\half}x_{\half})y_0)b_0.$
\end{itemize}

%\item 
%$\frac 14 (b,y_0)\Big(x_{\half}b_{\half}\Big)_0+(b_{\half}y_0)
%$\Big((b_{\half}y_0)(x_{\half}b_0)\Big)_0+2(x_{\half}b_{\half})(y_0b_0)
%+(x_{\half}(b_{\half}y_0))b_0=0.$
 
%\item 
% \begin{equation}\label{eq11203} 
%\begin{aligned}
%&\Big(b_{\half}(x_{\half}y_0)\Big)_0+\Big((b_{\half}y_0)x_{\half}\Big)_0 
%-(b_{\half}x_{\half})y_0+2(x_{\half}b_0)(b_{\half}y_0)\\
%&+\half(b,y_0)\Big(b_{\half}x_{\half}\Big)_0
%&+\half(b,y_0)\Big(b_{\half}x_{\half}\Big)_0-2(x_{\half}(b_{\half}y_0)  )b_0= 0.
%\end{aligned}
%\end{equation} 

%\item 
%$((b_{\half}x_{\half})y_0)b_0-((b_{\half}y_0)x_{\half})b_0=0.$
\end{enumerate}
\end{lemma}
\begin{proof}
 By Lemma \ref{P1212}(1),
\begin{equation}\label{eq1P12121}
\begin{aligned}
&\Big(x_{\half}(y_{\half}b_0)\Big)_0+\Big(y_{\half}(x_{\half}b_0)\Big)_0+4\Big((x_{\half}b_{\half})(y_{\half}b_{\half})\Big)_0\\
&+4\Big((x_{\half}b_0)(y_{\half}b_0)\Big)_0
 \quad =\\
 &(b,y_{\half})\Big(x_{\half}b_{\half}\Big)_0+(b,x_{\half})\Big(y_{\half}b_{\half}\Big)_0
-(x_{\half}y_{\half})b_0+(b,x_{\half}y_{\half})b_0.
\end{aligned}
\end{equation}
(1) Replacing in \eqref{eq1P12121} $y_{\half}$ by $b_{\half}y_0$ we get
\begin{equation*}
\begin{aligned}
&\textstyle{\frac 14 (b,y_0)\Big(x_{\half}b_{\half}\Big)_0
+\Big((b_{\half}y_0)(x_{\half}b_0)\Big)_0+2(x_{\half}b_{\half})(y_0b_0)+\frac 14(b,y_0)(b,x_{\half})b_0=}\\
&\textstyle{(b,y_0)\Big(x_{\half}b_{\half}\Big)_0+\half(b,x_{\half})y_0b_0
-(x_{\half}(b_{\half}y_0))b_0+\half(b,x_{\half}y_0)b_0+\frac 14(b,y_0)(b, x_{\half})b_0.}
\end{aligned}
\end{equation*}
\medskip

\noindent
(2) Replacing $x_{\half}$ with $x_{\half}b_0$ in \eqref{eq1P12121}, using \eqref{eqSI} we get
\begin{equation*}
\begin{aligned}
&\textstyle{\Big((x_{\half}b_0)(y_{\half}b_0)\Big)_0+0+\frac 14(b,x_{\half})(b,y_{\half})b_0+0\quad =}\\
 &\textstyle{\frac 14(b,x_{\half})(b,y_{\half})b_0+0
-((x_{\half}b_0)y_{\half})b_0+\half(b_0,x_{\half}y_{\half})b_0.}
\end{aligned}
\end{equation*}
\medskip

\noindent
(3) Replacing in (1) $x_{\half}$ with $x_{\half}b_0$ we get
\begin{gather*}
\textstyle{0+\frac 14(b,x_{\half})(b,y_0)b_0 =}\\
\textstyle{\frac {3}{16}(b,y_0)(b,x_{\half})b_0+0
-((x_{\half}b_0)(b_{\half}y_0))b_0+\frac 18(b,x_{\half})(b,y_0)b_0.}
\end{gather*}
\medskip

\noindent
(4\&5)  By Lemma \ref{120}(3),
\begin{equation*}\tag{$*$}
\begin{aligned} 
&\Big(b_{\half}(x_{\half}y_0)\Big)_0+\Big((b_{\half}y_0)x_{\half}\Big)_0 
-(b_{\half}x_{\half})y_0+2\Big((x_{\half}b_0)(b_{\half}y_0)\Big)_0\\
&+2(b_{\half}x_{\half})(y_0b_0)-2\Big((b_{\half}(x_{\half}(y_0b_0)\Big)_0 -2((b_{\half}y_0)x_{\half} )b_0= 0.
\end{aligned}
\end{equation*}
Replacing $y_0$ with $y_0b_0$, and using  \eqref{eqSI}, we get
\begin{gather*}
\textstyle{\Big(b_{\half}(x_{\half}(y_0b_0))\Big)_0+\frac 14(b,y_0)\Big(b_{\half}x_{\half}\Big)_0-(b_{\half}x_{\half})(y_0b_0)
+\frac 18(b,y_0)(b,x_{\half})b_0}\\
\textstyle{+\frac 14(b,y_0)(b,x_{\half})b_0-\frac 14(b,y_0)(b,x_{\half})b_0 
-\frac 18(b,y_0)(b,x_{\half})b_0=0,}
\end{gather*}
so (4) holds; going back to ($*$) we get (5).
\medskip

\noindent
(6) Replacing $x_{\half}$ with $x_{\half}b_0$ in (5), we get
\begin{equation*}
\begin{aligned} 
&\Big(b_{\half}((x_{\half}b_0)y_0)\Big)_0+\Big((b_{\half}y_0)(x_{\half}b_0)\Big)_0 
-(b_{\half}(x_{\half}b_0))y_0+0\\
&\textstyle{+\half(b,y_0)\Big(b_{\half}(x_{\half}b_0)\Big)_0 -2((b_{\half}y_0)(x_{\half}b_0) )b_0= 0.}
\end{aligned}
\end{equation*}
So (6) follows from (3) and Theorem \ref{SI}(5).
\medskip

\noindent
(7) (i) Multiplying (5) by $b_0$ using \eqref{eqSI} we get
\begin{equation*}
\begin{aligned} 
&\Big[\Big(b_{\half}(x_{\half}b_0)\Big)_0+(b_{\half}y_0)x_{\half} -(b_{\half}x_{\half})y_0\Big]b_0\\
&\textstyle{+\frac 18(b,x_{\half})(b,y_0)b_0+\frac 18(b,x_{\half})(b,y_0)b_0-(b_0,(b_{\half}y_0)x_{\half})b_0
= 0.}
\end{aligned}
\end{equation*}
Since $(b_0,(b_{\half}y_0)x_{\half})=\frac 14(b,x_{\half})(b,y_0)b_0,$ and $b_{\half}(x_{\half}b_0)b_0=\frac 14(b,x_{\half}y_0)b_0,$ (7i) holds.
\medskip

\noindent
(ii) This follows from \eqref{importanteq} and (i).
\medskip

\noindent
(iii)  This follows from (i).
\medskip

\noindent

\end{proof}

%%%%%%%%%%%%%%%%%%%%%%%%%%%
\begin{lemma}\label{lem 8.5}
\begin{gather*}
 2(b_{\half}y_0)(x_{\half}b_0) +2(b_{\half}x_{\half})(y_0b_0 )-2((b_{\half}x_{\half})y_0)b_0\\
 -2b_{\half}((x_{\half}b_0)y_0)=0.
\end{gather*}
\end{lemma}
\begin{proof}
This follows from equation \eqref{120a}.
\end{proof}

%%%%%%%%%%%%%%%%%%%%%%%%%%%%%%%%%%%%%%%%%%%%
\begin{lemma}\label{fact7}
%%%%%%%%%%%%%%%%%%%%%%%%%%%%%%%%%%%%%%%%5
 \begin{enumerate}
\item 
$(x_0b_0)(y_0b_0)=(x_0(y_0b_0))b_0.$

\item
$(x_0y_0)b_0+2\Big((b_{\half}x_0)(b_{\half}y_0)\Big)_0=y_0(x_0b_0)+2\Big(b_{\half}((b_{\half}y_0)x_0)\Big)_0.$

\item
$4\Big((b_{\half}x_0)(b_{\half}y_0)\Big)_0  =y_0(x_0b_0)+x_0(y_0b_0)-(x_0y_0)b_0.$

\item %4
%\begin{enumerate}
\begin{itemize}
\item[(i)]
$2\Big((x_{\half}b_0)(b_{\half}y_0)\Big)_0-\frac 14(b,x_{\half})y_0b_0+(b_{\half}x_{\half})(y_0b_0)
 -((b_{\half}x_{\half})y_0)b_0=0.$

\item[(ii)] 
$2(b_{\half}y_0)(x_{\half}b_0)+2b_{\half}((x_{\half}b_0)y_0)-\half(b,x_{\half}) y_0b_0=0.$
%$\frac 14 (b,y_0)\Big(x_{\half}b_{\half}\Big)_0+
%$3(x_{\half}b_0)(b_{\half}y_0)+3(b_{\half}x_{\half})(y_0b_0)=0.$
\end{itemize} 
\item 
If $\charc(\ff)\ne 3,$ then
\begin{itemize}
\item[(i)]
$\Big((b_{\half}y_0)(x_{\half}b_0)\Big)_0+(x_{\half}b_{\half})(y_0b_0)- \frac 14 (b,x_{\half}y_0)b_0\\
=\frac 14(b,y_0)\Big(x_{\half}b_{\half}\Big)_0+\frac 14(b,x_{\half})y_0b_0.$

\item[(ii)] 
$\Big((b_{\half}y_0)(x_{\half}b_0)\Big)_0+\Big(b_{\half}(x_{\half}(y_0b_0))\Big)_0- \frac 14 (b,x_{\half}y_0)b_0\\
=\frac 14(b,x_{\half})y_0b_0.$
\item[(iii)] $-\Big(b_{\half}((x_{\half}b_0)y_0)\Big)_0+\Big(b_{\half}(x_{\half}(y_0b_0))\Big)_0- \frac 14 (b,x_{\half}y_0)b_0=0.$
 \end{itemize}
\end{enumerate} 
%\end{enumerate}
 %
 \end{lemma}
 \begin{proof}
(1\&2):\ By Lemma \ref{P}(9),
\begin{gather*}
(x_0y_0)b_0+2(x_0b_0)(y_0b_0)+2\Big((b_{\half}x_0)(b_{\half}y_0)\Big)_0\\
=y_0(x_0b_0)+2(x_0(y_0b_0))b_0+2\Big(b_{\half}((b_{\half}y_0)x_0)\Big)_0.
\end{gather*}
Replacing $y_0$ with $y_0b_0$ using \eqref{eqSI}, we get
\begin{gather*}
\textstyle{(x_0(y_0b_0))b_0+\half(b,x_0)(b,y_0)b_0+\frac 14(b,y_0)x_0b_0}\\
\textstyle{=(y_0b_0)(x_0b_0)+\half(b,x_0)(b,y_0)b_0 +\frac 14(b,y_0)x_0b_0.}
\end{gather*}
This shows (1), and then (2) follows.
\medskip

\noindent
(3) Interchanging $x_0, y_0$ in (2) and adding, and then using Theorem \ref{0012} and \eqref{eqSI}(5), we get
\begin{gather*}
2(x_0y_0)b_0+4\Big((b_{\half}x_0)(b_{\half}y_0)\Big)_0\\
=y_0(x_0b_0)+x_0(y_0b_0)+2\Big(b_{\half}((b_{\half}y_0)x_0)\Big)_0+2\Big(b_{\half}((b_{\half}x_0)y_0)\Big)_0\\=
\\
y_0(x_0b_0)+x_0(y_0b_0)+(x_0y_0)b_0.
\end{gather*}
\medskip

\noindent
(4)
(i) Putting $\Big(b_{\half}x_{\half}\Big)_0$ in place of $x_0$ in (3) we get using \eqref{eqSI},
\begin{gather*}
\textstyle{4\Big(\Big[-\half x_{\half}b_0+\frac 14(b,x_{\half})b_{\half} \Big](b_{\half}y_0)\Big)_0}\\
\textstyle{=\frac 14(b,x_{\half})y_0b_0+(b_{\half}x_{\half})(y_0b_0)-((b_{\half}x_{\half})y_0)b_0}\\\iff
\\
\textstyle{-2\Big((x_{\half}b_0)(b_{\half}y_0)\Big)_0+\half (b,x_{\half})y_0b_0}\\
\textstyle{=\frac 14(b,x_{\half})y_0b_0+(b_{\half}x_{\half})(y_0b_0)-((b_{\half}x_{\half})y_0)b_0.}
\end{gather*}
\medskip

\noindent
(ii) Computing $2$ times the equality in (i) minus the equality in Lemma \ref{lem 8.5} yields (ii).
\medskip

\noindent
(5)  (i) By Lemma \ref{claim1} (1) and (7)(iii),
\begin{gather*}
\textstyle{\Big((b_{\half}y_0)(x_{\half}b_0)\Big)_0+2(x_{\half}b_{\half})(y_0b_0)- \frac 34 (b,x_{\half}y_0)b_0+((b_{\half}x_{\half})y_0)b_0=}\\
\textstyle{\frac 34(b,y_0)\Big(x_{\half}b_{\half}\Big)_0+\half(b,x_{\half})y_0b_0.}
\end{gather*}
Adding (4i) we get
\begin{gather*}
\textstyle{3\Big((b_{\half}y_0)(x_{\half}b_0)\Big)_0+3(x_{\half}b_{\half})(y_0b_0)- \frac 34 (b,x_{\half}y_0)b_0}\\
\textstyle{=\frac 34(b,y_0)\Big(x_{\half}b_{\half}\Big)_0+\frac 34(b,x_{\half})y_0b_0.}
\end{gather*}
\medskip

\noindent
(ii) By Lemma \ref{claim1}(4), $(b_{\half}x_{\half})(y_0b_0)-\frac 14(b,y_0)\Big(b_{\half}x_{\half}\Big)_0=\Big(b_{\half}(x_{\half}(y_0b_0))\Big)_0,$ so (ii) follows from (i).
\medskip

\noindent
(iii)  By Lemma \ref{claim1}(6), $\Big((b_{\half}y_0)(x_{\half}b_0)\Big)_0 -\frac 14 (b,x_{\half}) y_0b_0=-\Big(b_{\half}((x_{\half}b_0)y_0)\Big)_0,$ so (iii) follows from (ii).\qedhere

%
%By Lemma \ref{claim1}(1), 
%\[
%\frac 14 (b,y_0)\Big(x_{\half}b_{\half}\Big)_0
%(b_{\half}y_0)(x_{\half}b_0)+2(x_{\half}b_{\half})(y_0b_0)
%+(x_{\half}(b_{\half}y_0))b_0=0.
%\]
%Adding this to (i) we get,
%\begin{gather*}
%\frac 14 (b,y_0)\Big(x_{\half}b_{\half}\Big)_0+
%3(b_{\half}y_0)(x_{\half}b_0)+3(x_{\half}b_{\half})(y_0b_0)\\
%-[((b_{\half}x_{\half})y_0)b_0-(x_{\half}(b_{\half}y_0))b_0]=0.
%\end{gather*}
%However, by Lemma \ref{claim1}(5), $((b_{\half}x_{\half})y_0)b_0-((b_{\half}y_0)x_{\half})b_0=0,$ so (ii) holds.
%\end{comment}
 \end{proof}
 \begin{proof}[{\bf Proof of Theorem \ref{ab12}(2ii)}]$ $
 \medskip

 \noindent
 The proof of this part is immediate from Lemma \ref{claim1}(7ii) and Lemma \ref{fact7}(5iii).
 \end{proof}

 \end{document}